\documentclass[12pt]{article}

\usepackage[english]{babel}
\usepackage{amsmath,amssymb,amsthm}
\usepackage[shortlabels]{enumitem}
\setlist[enumerate]{label=\arabic*., nosep}
\usepackage{mathtools}
\usepackage{nicefrac}
\usepackage{hyperref,xcolor}
\hypersetup{
colorlinks=true,
linkcolor=red,
citecolor = blue,      
}
\usepackage{comment}
\usepackage{dsfont}
\usepackage{xspace} 
\usepackage[bottom]{footmisc}
\usepackage{authblk}
\usepackage{stmaryrd}
\usepackage{nameref}
\usepackage{titling}
\usepackage[T1]{fontenc}
\usepackage{libertinus}

\usepackage[utf8]{inputenc}
\usepackage{blindtext}

\usepackage[backend=biber]{biblatex}
\DeclareFieldFormat[article]{title}{\mkbibitalic{#1}}
\DeclareFieldFormat[article]{citetitle}{\mkbibitalic{#1}}
\theoremstyle{plain}
\newtheorem{theorem}{{\bf{Theorem}}}[section]
\newtheorem{proposition}[theorem]{{\bf Proposition}}
\newtheorem*{assumption}{{\bf Assumption}}
\newtheorem{lemma}[theorem]{{\bf Lemma}}

\theoremstyle{definition}
\newtheorem*{remark}{{\bf {Remark}}}
\newtheorem*{notation}{{\bf Notation and conventions}}
\newtheorem{definition}{{\bf Definition}}[section]
\newtheorem*{acknowledgements}{{\bf {Acknowledgements}}}

\numberwithin{equation}{section}

\newcommand{\R}{\mathbb{R}}

\newcommand{\N}{\mathbb{N}}
\newcommand{\Z}{\mathbb{Z}}
\newcommand{\im}{\operatorname{i}}
\renewcommand{\d}{\mathrm{d}}

\newcommand{\rel}{\operatorname{r}}
\newcommand{\com}{\operatorname{c}}

\DeclarePairedDelimiter\floor{\lfloor}{\rfloor}
\DeclarePairedDelimiter\epsfloor{\lfloor}{\rfloor_{\varepsilon}}

\DeclarePairedDelimiter\epsfracpart{\{}{\}_{\varepsilon}}

\newcommand{\E}{\mathbf{E}}

\newcommand{\bbE}{\mathbb{E}}

\newcommand{\cali}[1]{\mathcal{#1}}

\newcommand{\Ltwo}{\mathcal{L}^{2}}
\newcommand{\Lpee}[1]{\mathcal{L}^{#1}}

\newcommand{\ind}[1]{\mathbf{1}_{\{ #1 \}}}
\newcommand{\pair}{\mathsf{Pair}}
\newcommand{\intv}[1]{[\![#1]\!]}
\newcommand{\dgm}{\mathsf{Dgm}}

\newcommand{\tblue}{\textcolor{blue}}

\newcommand{\conj}[1]{\overbracket[1.1pt][-3.1pt]{#1}}

\newcommand{\norm}[1]{\Vert #1\Vert}
\newcommand{\normop}[1]{\Vert #1\Vert_{\operatorname{op}}}
\newcommand{\Normop}[1]{\big\Vert #1\big\Vert_{\operatorname{op}}}
\newcommand{\NOrmop}[1]{\Big\Vert #1\Big\Vert_{\operatorname{op}}}
\newcommand{\NORmop}[1]{\Bigg\Vert #1\Bigg\Vert_{\operatorname{op}}}
\newcommand{\Norm}[1]{\big\Vert #1\big\Vert}

\newcommand{\F}[1]{\widehat{#1}}

\newcommand{\transpose}[1]{#1^{\mathrm{t}}}

\newcommand{\partfunction}[1]{\mathcal{Z}^{#1}}

\newcommand{\epsshe}{\mathcal{Z}^{\theta,\varepsilon}}
\newcommand{\shflow}[1]{\mathcal{Z}^{#1}}

\newcommand{\onebrown}{\mathcal{B}}
\newcommand{\epsbrown}{\mathcal{B}^{\varepsilon}}

\newcommand{\outg}[1]{\mathcal{G}^{\ast}_{#1}}
\newcommand{\inc}[1]{\mathcal{G}_{#1}}
\newcommand{\offdiag}[1]{\mathcal{G}_{#1}}
\newcommand{\diag}[2]{\mathcal{K}^{#1}_{#2}}

\newcommand{\inctrans}[1]{\mathcal{P}_{#1}}
\newcommand{\outgtrans}[1]{\mathcal{P}^{\ast}_{#1}}
\newcommand{\offdiagtrans}[1]{\mathcal{P}_{#1}}
\newcommand{\diagtrans}[2]{\mathcal{J}^{#1}_{#2}}
\newcommand{\diagramtrans}[2]{\mathcal{R}^{#1}_{\vec{#2}}}
\newcommand{\deltatrans}[1]{\mathcal{Q}^{\intv{n}, #1}}

\newcommand{\Szero}[1]{S_{#1}}

\newcommand{\Sdualzero}[1]{\widehat{S}_{#1}}

\newcommand{\conttrans}{\mathcal{P}}
\newcommand{\twoconttrans}{\mathsf{p}}
\newcommand{\jtheta}{\mathrm{j}}
\newcommand{\corr}{\cali{Q}^{\intv{n}}}
\newcommand{\epscorr}[1]{\cali{Q}^{\intv{n}, #1}_{\varepsilon}}

\newcommand{\preoutg}[1]{\mathcal{G}^{\ast}_{\varepsilon #1}}
\newcommand{\preinc}[1]{\mathcal{G}_{\varepsilon#1}}
\newcommand{\preoffdiag}[1]{\mathcal{G}_{\varepsilon #1}}
\newcommand{\prediag}[2]{\mathcal{K}^{#1}_{\varepsilon #2}}

\newcommand{\preoutgtrans}[1]{\mathcal{P}^{\ast}_{\varepsilon#1}}
\newcommand{\preinctrans}[1]{\mathcal{P}_{\varepsilon#1}}
\newcommand{\preoffdiagtrans}[1]{\mathcal{P}_{\varepsilon#1}}
\newcommand{\prediagtrans}[2]{\mathcal{J}^{#1}_{\varepsilon#2}}
\newcommand{\prediagramtrans}[2]{\mathcal{R}^{#1}_{\vec{\varepsilon#2}}}
\newcommand{\predeltatrans}[1]{\mathcal{Q}_{\varepsilon}^{\intv{n}, #1}}
\newcommand{\modpredeltatrans}[1]{\mathcal{Q}_{\varepsilon}^{\intv{2m}, #1}}

\newcommand{\Seps}[1]{S_{\varepsilon #1}}
\newcommand{\Sdualone}[1]{\widehat{S}_{1 #1}}
\newcommand{\Sdualeps}[1]{\widehat{S}_{\varepsilon #1}}

\newcommand{\Teps}[1]{S_{\varepsilon #1}^{-1}}
\newcommand{\Tdualone}[1]{\widehat{S}_{1 #1}^{-1}}
\newcommand{\Tdualeps}[1]{\widehat{S}_{\varepsilon #1}^{-1}}

\newcommand{\green}{\mathcal{G}}
\newcommand{\twogreen}{\mathsf{g}}

\newcommand{\epsgreens}{\mathcal{G}_{\varepsilon}}

\newcommand{\epstrans}{\mathcal{P}_{\varepsilon}}
\newcommand{\epstwotrans}{\mathsf{p}_{\varepsilon}}

\newcommand{\fouriergen}{\F{\mathcal{L}}_{1,n}}
\newcommand{\epsgen}{{\mathcal{L}}_{\varepsilon,n}}
\newcommand{\fourierepsgen}{\F{\mathcal{L}}_{\varepsilon,n}}
\newcommand{\epstwogen}{\mathcal{L}_{\varepsilon}}
\newcommand{\onetwogen}{\mathcal{L}_{1}}
\newcommand{\fourierepstwogen}{\F{\mathcal{L}}_{\varepsilon}}
\newcommand{\twogen}{\mathcal{L}}

\newcommand{\epscontgreens}{\mathcal{G}_{\varepsilon}^{\operatorname{cont}}}

\newcommand{\epsconttrans}{\mathcal{P}_{\varepsilon}^{\operatorname{cont}}}

\newcommand{\diff}{\mathrm{Q}}

\hypersetup{
colorlinks=true,
linkcolor=red,
citecolor = blue,      
}

\newcommand{\sqrrange}{\sqrt{\kappa}}
\newcommand{\range}{\kappa}
\newcommand{\half}{\frac{1}{2}}
\newcommand{\quart}{\frac{1}{4}}
\newcommand{\threequart}{\frac{3}{4}}
\newcommand{\threehalf}{\frac{3}{2}}
\newcommand{\id}[1]{\mathbf{I}_{#1}}
\newcommand{\lid}[1]{\mathsf{I}_{#1}}
\newcommand{\fid}[1]{\widehat{\mathbb{I}}_{#1}}
\newcommand{\ballorigin}[1]{\mathbf{B}(0; #1)}
\newcommand{\quadrange}{\mathfrak{K}}

\newcommand{\shfparameter}{\theta_{\star}}
\newcommand{\proj}{\mathbf{P}_{\range}}
\newcommand{\projperp}{\mathbf{P}_{\range}^{\perp}}
\newcommand{\ie}{\textit{i.e.}\xspace}
\newcommand{\SHF}{\mathsf{SHF}}

\newcommand{\hattorus}[1]{\widehat{\mathbb{T}}^{#1}}
\newcommand{\epshattorus}[1]{\widehat{\mathbb{T}}_{\varepsilon}^{#1}}
\newcommand{\epsalhattorus}[1]{\widehat{\mathbb{T}}_{\varepsilon\alpha}^{#1}}

\newcommand{\metorus}[1]{\widetilde{\mathbb{T}}_{\varepsilon}^{#1}}
\newcommand{\mealtorus}[2]{\widetilde{\mathbb{T}}_{\varepsilon #1}^{#2}}

\newcommand{\mealeta}[1]{\widetilde{\eta}_{\varepsilon #1}}
\newcommand{\mzaleta}[1]{\widetilde{\eta}_{0 #1}}

\newcommand{\epsrect}[1]{\mathsf{R}_{\varepsilon}^{#1}}

\newcommand{\epsalrect}[1]{\mathsf{R}_{\varepsilon\alpha}^{#1}}

\newcommand{\measspace}[1]{\cali{M}_{+}(#1)}

\title{Continuous-time directed polymers to the Critical SHF via moments}
\author{Sudheesh Surendranath  \thanks{Department of Mathematics, Duke University \texttt{sudheesh.surendranath@duke.edu}}}
\begin{document}
	\date{}
	\maketitle
\vspace{-2cm} 
\begin{abstract}
	We consider a class of two-dimensional continuous-time directed random polymers in a correlated Brownian environment. The partition functions of these polymers form a discrete approximation of the Critical Stochastic Heat Flow (SHF), constructed by Caravenna, Sun and Zygouras in \cite{csz23stochasticheatflow}. Using a recent moment based axiomatic characterization of the SHF by Tsai~\cite{tsai25stochasticheatflow}, we prove that the point-to-point partition functions, viewed as random measures, converge to the critical SHF. The proof proceeds by first establishing convergence of all moments through the resolvent method introduced by Gu, Quastel and Tsai~\cite{gu2021moments} and based on \cite{rajeev1999condensation,dimock2004multi}. This is followed by establishing tightness and verifying the remaining axioms in the characterization.
\end{abstract}

\tableofcontents
\section{Introduction}{\label{sec:intro}}

The directed polymer model has been a central object of study in statistical physics and probability theory since its introduction by Huse and Henley~\cite{husehenley85pinning}. In this model, a discrete or continuous-time random walk moves on the lattice under the influence of a random environment, with the interaction strength governed by a coupling constant $\beta$, which plays the role of an inverse temperature. The polymer measure is then defined as an exponential tilt of the random walk measure by the energy accumulated from the environment along its path, and the degree of the tilt depends on $\beta$. Larger values of $\beta$ strengthen this tilt, favoring paths that spend more time in regions where the environment takes large values, while the underlying random walk still tends to spread out. Thus, the model naturally exhibits an energy-entropy competition. One of the main questions is to understand how the polymer path behavior deviates from that of the underlying random walk, corresponding to the case $\beta = 0$, as the strength of the disorder increases. In the \textit{weak disorder} phase, entropy dominates and the polymer retains its diffusive behavior, similar to that of the underlying random walk. In the \textit{strong disorder} phase, the effect of the environment becomes dominant, leading to localization phenomena. It is known that the critical value of $\beta$ separating these phases is $0$ in dimensions $d =1,2$ and strictly positive for $d \geqslant 3$; see \cite{carmonahu02, comets03localization}. Diffusive behavior in the weak disorder regime has been studied in \cite{comets06polymer, junk23stability, bolthausen89polymer}, while localization in the strong disorder regime was shown in \cite{carmonahu02,comets03localization,vargas07localization}. We refer to \cite{comets17directed} for a detailed account of directed polymers and to \cite{zygouras24directed} for recent developments in the field. 

The central statistic by which these notions are investigated is the partition function, which represents the weighted average of the exponential tilt over all possible up-right paths between two points or from a point to a line. It serves as a discrete or semi-discrete approximation of the Stochastic Heat Equation or SHE for short:
\begin{align}{\label{eqn:she}}
 \partial_{t}Z_{t}(x)
 = \half \Delta Z_{t}(x) + \sqrt{\beta} Z_{t}(x) \dot{W}_{t}(x) \quad t> 0, \quad x \in \R^{d} \, , 
\end{align}
where $\dot{W}_{t}(x)$ is a space-time white noise. The point-to-point partition function (defined in the next subsection) corresponds to the solution to the above equation with a delta initial condition, while the point-to-line partition function corresponds to the solution started from a constant function.

In dimension $d=1$, the above equation admits a unique non-negative solution which can be represented as a stochastic integral~\cite{walsh84,chendalang14, chendalang15} or as a chaos series expansion. The connection to the polymer partition function was made precise by Alberts--Khanin--Quastel \cite{alberts14intermediate}, wherein they showed that after a diffusive space-time scaling $(t,x) \mapsto (\varepsilon^{-2}t, \varepsilon^{-1}x)$ and $\beta_{\varepsilon} = \beta \varepsilon^{1/2}$, the rescaled partition functions converge in distribution to the solution of the one-dimensional stochastic heat equation. This provides the interpretation of the latter as the partition function of a \textit{continuum} directed polymer, constructed by the same authors in \cite{alberts14continuum}. It is important to note that this result holds for any $\beta > 0$. There is now a vast body of literature now that pursue these questions and their connection to the KPZ universality class~\cite{corwin12kpz, zygouras22kpz}. It is expected that the partition function exhibits fluctuations that are characteristic of this class. While this has been proven for select models~\cite{barraquand21loggamma, seppalainen12loggamma}, the general conjecture remains open. 

In dimension $d=2$, the situation is quite different. The equation~\eqref{eqn:she} no longer admits solutions using the classical probabilistic methods as the smoothing effect of the Laplacian equalizes the singular effect of the noise. This is termed critical in the literature and as a result, does not lend itself to the existing analytical methods for singular SPDEs such as regularity structures~\cite{hairer13kpz, hairer14regularity} or paracontrolled calculus~\cite{gubinelli15}. In the context of polymers, we say that the model is \textit{marginally relevant}, a terminology introduced in~\cite{csz17universality}. Although the polymer is in the strong disorder phase for any $\beta > 0$, the intermediate disorder regime $\beta_{\varepsilon} = \beta/|\log\varepsilon|$ exhibits a phase transition as $\beta$ is increased. In the normalization used in this paper, this critical value is $2\pi$. Below this value, the sequence of second moments of the partition functions remains bounded. Around $2\pi$, the fluctuations become strong enough that the limiting object is expected to be too rough or singular to be described by the classical solution theory. Using a coarse-graining method, Caravenna, Sun, and Zygouras~\cite{csz23stochasticheatflow} constructed this nontrivial scaling limit, now called the \textit{Critical Stochastic Heat Flow}, or SHF for short. It was first obtained as the scaling limit of critical two-dimensional directed polymers, viewed as measures on $\R^2\times\R^2$. This construction provides a natural candidate for the solution of the two-dimensional stochastic heat equation with multiplicative space-time white noise in \eqref{eqn:she}.

The moments of the critical SHF are an interesting story in themselves. By a formal application of It\^{o}'s formula, the $n$-point correlation function of the solution to equation~\eqref{eqn:she} can be seen to satisfy a parabolic PDE generated by the delta-Bose Hamiltonian
\begin{equation}{\label{eqn:deltabosegas_semigroup}}
	\cali{H}
	\coloneqq
	-\half \sum_{m=1}^{n}\Delta^{(m)}
	-
	\beta\sum_{i<j}\delta(x_i-x_j), \qquad x_1, x_2, \ldots x_n \in \R^{2} \,  \, ,
\end{equation}
where $\Delta^{(m)}$ acts on the $x_m$ coordinate. This formal notation has to be interpreted in a limiting sense, after a smooth approximation scheme. It turns out that in two dimensions, the constant $\beta$ has to be tuned logarithmically as before, in a critical window around $2\pi$.  For the $n=2$ case, based on \cite{albeverio88solvable}, an explicit expression for the associated semigroup was obtained in \cite{bertini1998two} in relative and center-of-mass coordinates. The third moment was identified in \cite{csz19moments} through the polymer partition function. Finally, using the resolvent expansion of Dimock and Rajeev~\cite{rajeev1999condensation, dimock2004multi}, Gu, Quastel, and Tsai~\cite{gu2021moments} derived an explicit formula for this semigroup, for all $n$, in terms of certain diagrams. In particular, they proved the strong convergence of the resolvents of the mollified delta-Bose gas semigroups, as operators on $\Ltwo(\R^{2\intv{n}})$, and then used the Trotter--Kato theorem to conclude convergence of the corresponding semigroups. In the SHE context, this represents the limiting moments of the noise-mollified equation started from initial conditions in $\Ltwo$. Formally, one writes $\deltatrans{\shfparameter}(t)=e^{-t\cali{H}}$, where $\shfparameter$ is a constant arising from the limiting procedure; see equation~\eqref{eqn:deltabosegas_semigroup} for a precise definition. Subsequent work in \cite{sudhtsai25deltabose} extended the result to weighted $\Ltwo$ spaces, which include exponentially growing functions. In particular, this class includes bounded and constant initial conditions.

Recently, Tsai~\cite{tsai25stochasticheatflow} gave an axiomatic characterization of the critical SHF based on its moments and used it to prove the convergence of the mollified two-dimensional stochastic heat equation to the SHF. This provides an intrinsic definition of the SHF without any reference to approximation procedures. To state the same, we need some notation. Let $\measspace{\R^2 \times \R^2}$ denote the space of locally finite Borel measures on $\R^2 \times \R^2$ equipped with the vague topology, and let $\R_{\leqslant}^{2} \coloneqq \{(s,t) \in \R^2 : 0 <s <t \}$ denote an ordered time index set. Denote by $C(\R_{\leqslant}^{2}, \measspace{\R^2 \times \R^2})$ the set of continuous processes on $\R_{\leqslant}^{2}$ with values in $\measspace{\R^2 \times \R^2}$, equipped with the uniform-on-compact topology. Then, from \cite{tsai25stochasticheatflow}, we have the following definition:

\begin{definition}{\label{defn:axiomatic_shf}}
	A process $\shflow{\shfparameter}$ is called a Stochastic Heat Flow with strength $\shfparameter$, or $\SHF(\shfparameter)$, if
    \begin{enumerate}[label=\textup{(\roman*)}, ref=\thedefinition\textup{(\roman*)}]
    	\item\label{defn:axiomatic_shf_one}
    	 \textbf{(Continuity)} $\shflow{\shfparameter} = \{\shflow{\shfparameter}_{s,t}\} \in C(\R_{\leqslant}^{2}, \measspace{\R^2 \times \R^2})$. 
    	\item\label{defn:axiomatic_shf_two}
    	\textbf{(Chapman--Kolmogorov)} For all $t_{0} < t_{1} < t_{2}$, there exists a sequence $\{u_{\ell}\}$ in $\Ltwo(\R^2)$ for which the operator $\psi \mapsto u_{\ell} \ast \psi$ converges strongly to the identity, such that the difference $\shflow{\shfparameter}_{t_{0},t_{1}} \bullet_{\ell} \shflow{\shfparameter}_{t_{1},t_{2}} - \shflow{\shfparameter}_{t_{0},t_{2}}$ is well-defined and converges to $0$ vaguely, in probability as $\ell \to \infty$.
    	\item\label{defn:axiomatic_shf_three} 
    	\textbf{(Independent increments)} For all $ t_{0}  < \cdots < t_{k}$, the random variables $\shflow{\shfparameter}_{t_{0},t_{1}},  \ldots \shflow{\shfparameter}_{t_{k-1}, t_{k}}$ are independent.
    	\item\label{defn:axiomatic_shf_four}
    	\textbf{(Delta-Bose Moments)} For $n = 1,2,3,4$, $s<t$ and $h, h' \in \Ltwo(\R^2)$, the following holds:
    	\begin{equation*}
    		\E \Big(\Big\langle h, \shflow{\shfparameter}_{s,t} h' \Big\rangle \Big)^{n} = \Big\langle h^{\otimes n}, \deltatrans{\shfparameter}(t-s) h'^{\otimes n} \Big\rangle \, .
    	\end{equation*} 
    \end{enumerate}
\end{definition}

The product $\bullet_{\ell}$ in Definition~\ref{defn:axiomatic_shf_two} was originally introduced in \cite{clark24continuum} using the heat kernel as an approximate identity, as part of the construction of the polymer measure associated to the SHF. Tsai~\cite{tsai25stochasticheatflow} later redefined it in terms of more general $\Ltwo$ functions.

The main theorem of \cite{tsai25stochasticheatflow} then asserts that if $\shflow{\shfparameter}$ and $\widetilde{\cali{Z}}^{\shfparameter}$ are both $\SHF(\shfparameter)$, then they are equal in law. This reduces the task of proving convergence to the SHF to proving convergence of the moments and verifying the other axioms, thereby bypassing the traditional route of showing convergence of finite-dimensional distributions. Using this characterization, some further properties of the SHF have been investigated in the recent works \cite{gu2025stochasticheatflowblack, clark25conditional, gu2026loglog}. 

In this paper, we use this result to show the convergence of the partition functions of a large class of continuous-time polymers in a correlated Brownian environment to the critical Stochastic Heat Flow. Our approach is first to prove convergence of the moments of all orders, show tightness, and then verify the axioms above.

Before we describe our model and the key results, let us mention some related works. In the sub-critical regime \ie $\beta < 2\pi$, the limiting one-point distribution of the partition function and the random field fluctuations in $d=2$ were proven to be log-normal and Gaussian in~\cite{csz17universality}. In fact, this result fits into the larger framework of marginally relevant disorder systems, all of which posses a limiting log-normal one-point law and Gaussian random field fluctuations as shown in the same work, thus exhibiting universality. Similar results were obtained for the two-dimensional SHE and the KPZ equation in~\cite{csz20twokpz, gu2020kpz}. 

For the nonlinear SHE, the limiting one-point statistics was characterized in~\cite{dunlap22forward} as the solution to a forward backward SDE evaluated at a terminal time, and later extended in~\cite{dunlap2026renormalizationflow2dnonlinear} to include more general nonlinearities. The corresponding fluctuations were shown to be Gaussian in~\cite{tao24gaussian, dunlap25edwards}.

In the critical regime, the $n$-th moments, as calculated in \cite{gu2021moments} predicted a growth rate with an upper bound $\exp{({e^{cn^2}})}$, while works in the physics literature~\cite{rajeev1999condensation} conjectured a rate of $\exp{(e^{n})}$. A lower bound of this form was recently shown in \cite{ganguly2025sharpmomentuppertail}.   A probabilistic analysis of the delta-Bose semigroup, including a Feynman-Kac formula, was undertaken in the series of works \cite{chen25skew, chen1stochastic, chen2stochastic, chen3stochastic, chen4stochastic}.

\subsection{The model and main results}

Let us begin by describing the model we consider. The random environment is given by a collection of Brownian motions $\{\onebrown_{t}(z)\}_{t>0, z \in \Z^2}$ satisfying the covariance relation
\begin{equation*}
	\E\Big( \onebrown_{t}(\varphi)\onebrown_{t'}(\psi)\Big) = t \wedge t'\sum_{z \in \Z^2}\sum_{z' \in \Z^2}\varphi(z) \range(z - z')\psi(z') ,
\end{equation*}
for $\varphi,\psi \in \cali{S}(\Z^2)$. Here, $\onebrown_{t}(\varphi) \coloneqq \sum_{z \in \Z^2}\varphi(z)\onebrown_{t}(z)$, with the analogous definition for $\psi$, and $\range:\Z^2 \to \R$ is assumed to be a non-negative, non-negative definite symmetric function such that there exists an integer $K> 0$ for which
\begin{equation}{\label{eqn:range_energy}}
	\quadrange \coloneqq \mathop{\sum_{z \in \Z^2}\sum_{z' \in \Z^2}}_{\norm{z - z'} \geqslant K} \range(z) \, \big(\log \norm{z - z'} \big)^{2} \, \range(z') < \infty \, \, .
\end{equation}


The underlying random walk is taken to be a two-dimensional compound Poisson process $\{S_{t}\}_{t>0}$ with jump measure $\nu$. More precisely,
\begin{equation*}
	S_{t} = \sum_{m=0}^{N_{t}}X_{m}\quad , \qquad \qquad X_{m} \, \text{i.i.d.} \sim \nu \, \, ,
\end{equation*}
where $\{N_{t}\}_{t>0}$ is a rate-one Poisson process independent of the variables $\{X_m\}$. In particular, $\{S_{t}\}_{t>0}$ is a L\'{e}vy process with characteristic exponent
\begin{equation*}
	\Psi(\xi) = \sum_{m \in \Z^2}(1 - e^{\im \xi \cdot m})\nu(m) \, , \quad \xi \in [-\pi,\pi]^2 \, .
\end{equation*}

The generator $\twogen$ of the random walk is defined through its action in Fourier space by the symbol $-\Psi$, that is,
\begin{equation*}
	\widehat{\twogen\varphi}(\xi) = -\Psi(\xi)\widehat{\varphi}(\xi) \coloneqq -\F{\twogen}(\xi)\widehat{\varphi}(\xi) \, .
\end{equation*}

Regarding the jump measure $\nu$, we make the following standing assumption.
\begin{assumption} The jump measure $\nu$ satisfies the following:
	\begin{enumerate}[label=(\arabic*)]
		\item \label{assumption1} It is symmetric, \ie, $\nu(m) = \nu(-m)$ for all $m \in \Z^{2}$, and the closure of $\{m \in \Z^2 \, : \, \nu(m) \neq 0\}$ is $\Z^2$. 
		\item \label{assumption2} There exists a $\delta \in (0,1)$ such that 
		\begin{equation*}
			\sum_{m \in \Z^{2}} \norm{m}^{2+\delta} \nu(m) < \infty .
		\end{equation*}
		\item \label{assumption3} The sum $\sum_{m \in \Z^{2}} m \, \transpose{m} \nu(m) = \lid{1}$, where $\lid{k}$ is the identity operator on $\R^{2k}$.     
	\end{enumerate}
\end{assumption}
Under Assumption~\ref{assumption1}, the jumps have mean zero, \ie,
\[
\sum_{m \in \Z^{2}} m \, \nu(m) = 0.
\]
The second part of Assumption~\ref{assumption1} is equivalent to aperiodicity of the random walk on the lattice. In terms of the characteristic exponent $\Psi$, this is equivalent to imposing
\begin{equation*}
	\Psi^{-1}(0) = \{0\}   \qquad \text{in} \quad \hattorus{2} \coloneqq [-\pi,\pi]^{2} \, \, . 
\end{equation*} 
Assumptions~\ref{assumption2} and~\ref{assumption3} ensure that, under diffusive scaling, the generator symbol converges to $\half\norm{\xi}^2$, the symbol of $\half \Delta$, uniformly on compact sets. Moreover, the three assumptions together imply that the symbol is uniformly bounded from below by the symbol of the Laplacian; see Lemma~\ref{lem:uniformgreensbound}. The point-to-point partition function for the polymer is then defined, for $x,x' \in \Z^2$, by
\begin{equation*}
	\partfunction{\beta}_{s,t}(x,x') \coloneqq \bbE \bigg[\exp\bigg(\sqrt{\beta}\int_{s}^{t}\d\onebrown_{u}(S_{u}) -\frac{\beta}{2}\range(0)\bigg) \ind{S_t=x'} \bigg| S_{0} = x\bigg] \, \, .
\end{equation*}
Here, $\bbE$ denotes the expectation with respect to the underlying random walk and the factor $\exp(-\tfrac{\beta}{2}\range(0))$ is a normalization constant intended to keep the mean of the partition function constant with respect to the environment. The coupling constant $\beta$ is tuned according to $\varepsilon$ as follows:
\begin{equation}{\label{eqn:beta_epsilon}}
	\beta_{\varepsilon} = \frac{2\pi}{|\log\varepsilon|}\Bigg(1 + \frac{\theta}{|\log\varepsilon|}\Bigg),
\end{equation}
where $\theta \in \R$ is a fine-tuning constant. The corresponding diffusively rescaled partition function is defined as:
\begin{equation}{\label{eqn:rescaled_partition_function}}
	\epsshe_{s,t}(x, x') \coloneqq \partfunction{\beta_{\varepsilon}}_{s/\varepsilon^2, t/\varepsilon^2}(x/\varepsilon, x'/\varepsilon),
\end{equation}
for $x, x' \in \Z_{\varepsilon}^2 \coloneqq (\varepsilon \Z)^{2}$.
These partition functions then solve the semi-discrete stochastic heat equation
\begin{equation*}
	\partial_{t}\epsshe = \epstwogen\epsshe +\sqrt{\beta_{\varepsilon}}\epsshe \d\epsbrown_{t} \, \, ,
\end{equation*}
where $\epstwogen$ is the generator of the diffusively scaled walk so that $\onetwogen \equiv \twogen$ and $\epsbrown_{t}(z) \coloneqq \onebrown_{t/\varepsilon^2}(z/\varepsilon)$ has the same law as $\varepsilon^{-1}\onebrown_{t}(z/\varepsilon)$ for $z \in \Z_{\varepsilon}^2$. By Lemma~\ref{lem:uniformgreensbound}, the multiplier $\fourierepstwogen(\xi)$ converges to $\half \norm{\xi}^2$ on $\epshattorus{2}\coloneqq [-\pi/\varepsilon, \pi/\varepsilon]^{2}$. This implies that the generator $\epstwogen$ converges to $\half\Delta$, and hence, in the limit, the underlying random walk converges to Brownian motion.

Regarding the noise, it converges to space-time white noise, up to a constant. To show this, note that since the discrete noise is Gaussian, it suffices to prove the convergence of the covariance functions to assert the convergence in $\cali{S}'(\R_{\geqslant 0} \times \R^2)$. Fix $\varphi, \psi \in \cali{S}(\R_{\geqslant 0} \times \R^2).$
\begin{align*}
 \mathbf{E}
 \Bigg(\int_{\R_{\geqslant0}} \varepsilon^2\sum_{x \in \Z_{\varepsilon}^{2}} \varphi(t ,x)
 \, \d \epsbrown_{t}(x),\int_{\R_{\geqslant0}} \varepsilon^2
 &\sum_{x' \in \Z_{\varepsilon}^{2}} \psi(t' ,x') \, \d \epsbrown_{t'}(x') \Bigg) \\
 & =\sum_{x \in \Z_{\varepsilon}^{2}}\range\bigg(\frac{x}{\varepsilon}\bigg)\int_{\R_{ \geqslant 0}} \varepsilon^2 \sum_{x' \in \Z_{\varepsilon}^{2}} \varphi(t,x + x')\psi(t,x') \, \d t \\
 & \xrightarrow{\varepsilon \to 0} \sum_{x \in \Z^2}\range(x) \,\int_{\R_{\geqslant 0}} \varphi(t,x')\psi(t, x') \, \d t \, \, .
\end{align*}
Consequently, $\d \epsbrown$ converges to a space-time white noise in $\cali{S}'(\R_{\geqslant 0} \times \R^2)$, with a coefficient depending on the correlation structure of the discrete Brownian motions. Henceforth, without any loss of generality, we will assume 
\begin{equation*}
	\sum_{x \in \Z^2} \range(x)  = 1 \, ,  
\end{equation*} 
so that the limiting noise is a standard white noise. Let $\intv{n}$ denote the ordered set of integers $\{1, 2, 3 \ldots n\}$ and define the space
\begin{equation*}
	\R^{2\intv{n}} \coloneqq \{x = (x_i)_{i \in \intv{n}}: x_i \in \R^2\} \, .
\end{equation*}
Define the $n$-point correlation function of the partition function for $x, x' \in \Z^{2\intv{n}}_{\varepsilon}$ as
\begin{equation*}
	\predeltatrans{\theta}(t, x, x') = \E \bigg[\prod_{m=1}^{n} \epsshe_{0,t}(x_m, x'_m)\bigg] \, .
\end{equation*}
Using It\^{o}'s formula, one sees that the correlation function satisfies the semi-discrete PDE:
\begin{equation}{\label{eqn:predeltabosegas_semigroup}}
	\partial_{t}\predeltatrans{\theta}(t) = \epsgen\predeltatrans{\theta}(t) + \beta_{\varepsilon} \sum_{i < j}\range^{\varepsilon}(x_i - x_j) \predeltatrans{\theta}(t) \, \, ,
\end{equation}
where $\range^{\varepsilon}: \Z_{\varepsilon}^{2} \to \R$ is defined by $\range^{\varepsilon}(x_i) = \varepsilon^{-2}\range(x_i/\varepsilon)$ and $\epsgen = \sum_{k=1}^{n}\epstwogen^{(k)}$ is the generator of the random walk on $\Z_{\varepsilon}^{2\intv{n}} \subset \R^{2\intv{n}}$ corresponding to $n$ copies of the two-dimensional random walk on $\Z_{\varepsilon}^2$ generated by $\epstwogen^{(k)}$ for $k = 1, 2 \ldots n$. The jump measure for $\epsgen$ is the product measure $\nu_{n}(m) \coloneqq \prod_{k=1}^{n}\nu(m_k)$ on $\Z_{\varepsilon}^{2\intv{n}}$. For $\xi =(\xi_1, \xi_2, \ldots , \xi_n) \in \epshattorus{2\intv{n}} \coloneqq [-\pi/\varepsilon, \pi/\varepsilon]^{2\intv{n}}$, its symbol is given by
\begin{equation*}
	\fourierepsgen(\xi) = \sum_{k=1}^{n} \sum_{m_k \in \Z^{2}_{\varepsilon}}(1 - \cos(\xi_k \cdot m_k))\nu(m_k) \, .
\end{equation*}

We now state the first main result of this paper.
\begin{theorem}{\label{thm:maintheorem}}
	There exists $\shfparameter \in \R$ such that the  semigroup defined in equation~\eqref{eqn:predeltabosegas_semigroup} converges strongly to the semigroup generated by the Hamiltonian in equation~\eqref{eqn:deltabosegas_semigroup}. That is, for all $f \in \Ltwo(\R^{2\intv{n}})$, $\norm{\predeltatrans{\theta}(t)f - \deltatrans{\shfparameter}(t)f } \to 0$ as $\varepsilon \to 0$. The constant $\shfparameter$ is explicitly given by:
	\begin{align*}
		\shfparameter
		= 2\theta + \frac{1}{\pi}\int_{\R^2} \int_{\R^2} \int_{\hattorus{2}} &\range(y_{\rel}) \Bigg[\frac{e^{\im \eta_{\rel} \cdot (\floor{y_{\rel}} - \floor{y'_{\rel}}) } -1} {	\fouriergen(\mzaleta{\alpha})} \Bigg] \range(y'_{\rel}) \, \d \eta_{\rel} \, \d y_{\rel} \, \d y'_{\rel} \\  
		& +  \frac{1}{\pi}\int_{\hattorus{2}} \Bigg[\frac{1}{	\fouriergen(\mzaleta{\alpha})} - \frac{1}{\norm{\eta_{\rel}}^2}\Bigg] \, \d\eta_{\rel} +\frac{\log 2}{2\pi} - \frac{G}{\pi^2} + 2\log\pi \,  ,
	\end{align*}
where $G=0.915965\ldots$ is the Catalan's constant and
\begin{equation*}
  	\fouriergen(\mzaleta{\alpha})  \coloneqq 4 \sum_{m_1 \in \Z^2} \sin^{2}\Big(\frac{\eta_{\rel} \cdot m_1}{2} \Big) \nu(m_1) \, .
\end{equation*}
\end{theorem}

The proof uses the resolvent method as introduced in \cite{gu2021moments}, but adapted to the setting described here. In particular, we define the discrete delta-Bose resolvents and establish their convergence to the limiting resolvent from the same paper. As the random walk is symmetric, the generator is self-adjoint. This allows us to use the Trotter--Kato theorem to deduce the convergence of the corresponding semigroups. Once this is done, we establish certain weighted norm bounds on these semigroups, following \cite{tsai25stochasticheatflow}, which yield tightness for the family of partition functions in $C(\R_{\leqslant}^{2}, \measspace{\R^2 \times \R^2})$. Using Theorem~\ref{thm:maintheorem}, we then show that all limit points of this family satisfy the axioms in Definition~\ref{defn:axiomatic_shf}, whence, using Tsai's result, we obtain:
\begin{proposition}{\label{cor:shfconvergence}}
	The rescaled partition functions defined in equation~\eqref{eqn:rescaled_partition_function} converge in law to the $\SHF(\shfparameter)$ in the space $C(\R_{\leqslant}^{2}, \measspace{\R^2 \times \R^2})$.
\end{proposition}

\begin{remark}
During the completion of this work, we became aware of an upcoming work by Drillick, Hou and Parekh~\cite{drillick26universality} in which convergence to the SHF in the sense of finite-dimensional distributions is established for a large class of discrete and continuous-time models, including random walks in space-time random environments and polymer models with finite-range spatial correlations. 
\end{remark}

\begin{notation}
	For two quantities $a$ and $b$, we will use the notation $a \lesssim b$ or $a \leqslant cb$ to mean the inequality holds for some constant $c$ independent of the parameters under consideration, and only depending on universal parameters or those that are fixed throughout the paper. When $a \lesssim b$ and $b \lesssim a$, we write $a \sim b$. Set $a \vee b$ and $a \wedge b$ to denote the maximum and minimum of $a$ and $b$, respectively. For a Borel measurable set $A$, we will use $\mathbf{1}_{A}$ for the indicator function of the set. For $p \in \N$, we write $\norm{\cdot}_{p}$ for the $\cali{L}^{p}$ norm of a function. We omit the underlying space when it is clear from context. We will use $\norm{\cdot}$ and $| \cdot | $ to denote the $\ell^{2}$ and $\ell^{1}$ norms of a vector respectively, again without referencing the vector space explicitly.
	For a point $x \in \R^2$, we will use $\epsfloor{ x }$ to denote the closest point in $\Z_{\varepsilon}^{2}$ southwest of the given point. When $\varepsilon = 1$, we will just omit the subscript and simply resort to the more familiar notation $\floor{ \cdot }$. We will use $f$ and $g$ to denote generic test functions whose definitions are local and whose domains might change from place to place. 
	
	Here and throughout the paper, for a function $f$ on $\R^{2\intv{n}}$, we adopt the following convention for its Fourier transform and the corresponding inverse:
\begin{equation*}
	\F{f}(\xi) = \frac{1}{(2\pi)^{n}}\int_{\R^{2\intv{n}}} f(x) e^{\im \xi \cdot x} \, \d x \qquad f(x) = \frac{1}{(2\pi)^{n}} \int_{\widehat{\R}^{2\intv{n}}} \F{f}(\xi) e^{-\im \xi \cdot x} \, \d \xi \, ,
\end{equation*}
where $\widehat{\R}^{2\intv{n}}$ is the dual space (defined in the next section). Accordingly, Plancherel's Theorem takes the form
\begin{equation*}
	\langle f, g \rangle \coloneqq \int_{\R^{2\intv{n}}}f(x) \conj{g(x)} \, \d x = \int_{\widehat{\R}^{2\intv{n}}} \F{f}(\xi) \conj{\F{g}(\xi)} \, \d \xi = \langle \F{f}, \F{g} \rangle \, .
\end{equation*}
Finally, if $\cali{T}$ is an operator on an $\Ltwo$ space with Fourier representation $\F{\cali{T}}$, then, to avoid extra notation, we use $\langle f, \F{\cali{T}}g\rangle$ to mean the more precise expression $\langle \F{f}, \F{\cali{T}}\F{g}\rangle$.
\end{notation}

The paper is organized as follows: in Section~\ref{section:deltaboseresolvent}, we recall the delta-Bose gas semigroup and also present the corresponding discrete version. In Section~\ref{section:convergenceofresolvent}, we prove the convergence of the semigroups, Theorem~\ref{thm:maintheorem}, using the method of resolvents. Lastly, in Section~\ref{section:convergencetoshf}, we use this result to establish the convergence result of the partition functions to the SHF, namely Proposition~\ref{cor:shfconvergence}.

\begin{acknowledgements}
	The author would like to sincerely thank Li-Cheng Tsai for suggesting the problem and numerous insightful interactions, and for a careful reading of earlier drafts of this manuscript. This work is partially supported by the NSF through the grants DMS-2243112 and DMS-2245242.
\end{acknowledgements}

\section{The delta-Bose gas resolvent}{\label{section:deltaboseresolvent}}

This section is primarily devoted to describing the resolvent of the delta-Bose gas semigroup and its corresponding semi-discrete version for the problem at hand in this paper. In addition, we prove the boundedness of the resolvent. Although this result was already shown in \cite{gu2021moments}, we present it again here to keep the paper self-contained and because our setting differs from theirs slightly: in our case, the correlation is not compactly supported. Also, as mentioned in the introduction, the proof of our main theorem is based on comparing the discrete and continuous resolvents. It will therefore be useful to derive explicit expressions for the latter in the Fourier domain. Let us start by describing the semigroup, for which we need to introduce some notation. Denote by $\pair\intv{n}$ the set of unordered pairs $\alpha = ij \coloneqq \{i < j\}$ in $\intv{n}$. This indexes the pair of particles under consideration among the particles $1,2, \ldots n$, namely particles $i$ and $j$. For a fixed $\alpha \in \pair\intv{n}$, we define the center-of-mass and relative coordinate spaces of $\R^{2\intv{n}}$ with respect to $\alpha$ as follows:
\begin{equation*}
	\begin{split}
		& \R^{2\intv{n}}_{\alpha} \coloneqq \R^2  \times  \R^{2\intv{n}\setminus\rel}_{\alpha}\coloneqq \R^2  \times \R^2 \times \R^{2\intv{n}\setminus\alpha}_{\alpha} \coloneqq \{(y_{\rel}, y_{\com}, (y_k)_{k \in \intv{n} \setminus \alpha}): y_{\rel}, y_{\com}, y_k \in \R^2\}	 \, . \\
	\end{split}
\end{equation*}
The indices $\com$ and $\rel$ stand for \textit{center-of-mass} and \textit{relative}, respectively. Clearly, as vector spaces, $\R^{2\intv{n}}$ and $\R^{2\intv{n}}_{\alpha}$ are isomorphic.  We will use a hat to denote the respective Fourier duals. That is, 
\begin{equation*}
	\widehat{\R}^{2\intv{n}} \coloneqq \{(\xi_k)_{k \in \intv{n}}:  \xi_k \in \R^2\} \quad  \text{and} \quad  \widehat{\R}^{2\intv{n}}_{\alpha} \coloneqq \{(\eta_{\rel}, \eta_{\com}, (\eta_k)_{k \in \intv{n}\setminus\alpha} ): \eta_{\rel}, \eta_{\com}, \eta_k \in \R^2\}\, \, .
\end{equation*}
Likewise, for a subset $\omega \subset \intv{n}$, set 
\begin{equation*}
	\R^{2\intv{n}\setminus\omega} \coloneqq \{(x_k)_{k \in \intv{n}\setminus\omega} : x_k \in \R^2\} \quad \text{  and } \quad \widehat{\R}^{2\intv{n}\setminus\omega} \coloneqq \{ (\xi_k)_{k \in \intv{n}\setminus\omega} : \xi_k \in \R^2\} \, .
\end{equation*}
We will frequently omit the relative coordinate and use $y_{\com \cup \intv{n}\setminus\alpha}$ to denote $(y_{\com}, (y_k)_{k \in \intv{n}\setminus\alpha})$ in $\R^2 \times \R^{2\intv{n}\setminus\alpha}_{\alpha}$. The same convention, together with the notations introduced above, apply to points in $\widehat{\R}^{2\intv{n}}_{\alpha}$ and all the subsets contained in these spaces.

Define the linear map $\Szero{\alpha} :  \R^{2\intv{n}\setminus\rel}_{\alpha} \to \R^{2\intv{n}}$ by
\begin{equation}
	\big(\Szero{\alpha}y_{\com \cup \intv{n}\setminus\alpha} \big)_{k} \coloneqq  \begin{cases}
		y_{\com}   & \text{ when }  k \in \alpha \\
		y_k        & \text{ when } k \in \intv{n}\setminus\alpha \quad .
	\end{cases}
\end{equation}
This map has the interpretation of collapsing particles $i$ and $j$, so that their relative distance vanishes, placing them at their center-of-mass position, and switching to that frame of reference. Throughout the paper, we let $\conttrans(t)$ and $\twoconttrans(t)$ be the standard heat kernels on $\R^{2\intv{n}}$ and $\R^2$, respectively, with their corresponding resolvents denoted by $\green(\lambda)$ and $\twogreen(\lambda)$. For $\alpha \neq \alpha' \in \pair\intv{n}$, define a set of operators with non-negative kernels:
\begin{equation}{\label{eqn:defn_limiting_operators}}
		\begin{split}
			\outgtrans{\alpha}(t, x,y) & \coloneqq \conttrans(t, x- \Szero{\alpha}y_{\com \cup \intv{n}\setminus\alpha})\sqrrange(\floor{y_{\rel}}) \eqqcolon \inctrans{\alpha}(t, y, x) \\
			\offdiagtrans{\alpha\alpha'}(t, y, y') &\coloneqq \sqrrange(\floor{y_{\rel}})\conttrans(t, \Szero{\alpha}y_{\com\cup\intv{n}\setminus\alpha} - \Szero{\alpha'}y'_{\com\cup\intv{n}\setminus\alpha})\sqrrange(\floor{y'_{\rel}})          \\
			\diagtrans{\shfparameter}{\alpha}(t, y, y')	 & \coloneqq 4\pi\jtheta(t)\sqrrange(\floor{y_{\rel}})\twoconttrans\big(\nicefrac{t}{2}, y_{\com} - y'_{\com}\big) \cdot \prod_{k \in \intv{n}\setminus\alpha}\twoconttrans(t, y_k - y'_k) \cdot \sqrrange(\floor{y'_{\rel}}) \quad \text{ with } \\
			\jtheta(t)  & \coloneqq \int_{0}^{\infty} \frac{t^{u-1}e^{\shfparameter u}}{\Gamma(u)} \d u \, .
		\end{split}
\end{equation}
In the above definition, $x, x' \in \R^{2\intv{n}}$, while $y,y' \in \R^{2\intv{n}}_{\alpha}$.  Define the set of \textit{diagrams} on the set $\intv{n}$ as
\begin{equation*}
	\dgm\intv{n} \coloneqq \{\vec{\alpha} = (\alpha_k)_{k=1}^{m} \in \pair\intv{n}^{m} \, | \, m \eqqcolon |\vec{\alpha}|\in \Z_{>0}, \alpha_k \neq \alpha_{k+1} \text{ for } k =1, \ldots ,m-1 \} .
\end{equation*}

Finally, for $\vec{\alpha} \in \dgm\intv{n}$, let 
\begin{equation*}
	\diagramtrans{\shfparameter}{\alpha}(t) \coloneqq \int_{\Sigma(t)}  \outgtrans{\alpha_1}(\tau_{0}) \cdot \prod_{k=1}^{|\vec{\alpha}|-1} \diagtrans{\shfparameter}{\alpha_k}(\tau_{k - \half})\offdiagtrans{\alpha_k\alpha_{k+1}}(\tau_k) \cdot \diagtrans{\shfparameter}{\alpha_{|\vec{\alpha}|}}(\tau_{|\vec{\alpha}|- \half})\inctrans{\alpha_{|\vec{\alpha}|}}(\tau_{|\vec{\alpha}|}) \, \d \vec{\tau} \, ,
\end{equation*}
where $\Sigma(t)$ is the simplex $\{\vec{\tau} = (\tau_{\frac{k}{2}})_{k=0}^{2|\vec{\alpha}|} : \sum_{k} \tau_{\frac{k}{2}} = t\}$, which clearly depends on $|\vec{\alpha}|$ but we omit this for brevity.
The delta-Bose semigroup is then given by the expression:
\begin{equation}{\label{eqn:limiting_semigroup}}
	\deltatrans{\shfparameter}(t) \coloneqq \conttrans(t) + \sum_{\vec{\alpha} \in \dgm\intv{n}} \diagramtrans{\shfparameter}{\alpha}(t) \, .
\end{equation}

As mentioned in the introduction, we will prove convergence of the semigroup using the resolvent method. For this, we apply the Laplace transform to both sides of the above equation. The right-hand side above is a sum of convolutions in the time variable, which becomes a product after the transform;
\begin{align}{\label{eqn:limiting_resolvent}}
	\int_{0}^{\infty}e^{-\lambda t}\deltatrans{\shfparameter}(t) \, \d t 
	& = \green(\lambda) + \sum_{\vec{\alpha} \in \dgm\intv{n}}\int_{0}^{\infty}e^{-\lambda t} \diagramtrans{\shfparameter}{\alpha}(t) \, \d t \nonumber \\
	& = \green(\lambda) + \sum_{\vec{\alpha} \in \dgm\intv{n}} \outg{\alpha_1}(\lambda)\prod_{k=1}^{|\vec{\alpha}|-1} \diag{\shfparameter}{\alpha_k}(\lambda)\offdiag{\alpha_k\alpha_{k+1}}(\lambda) \cdot \diag{\shfparameter}{\alpha_{|\vec{\alpha}|}}(\lambda)\inc{\alpha_{|\vec{\alpha}|}}(\lambda) \, \, ,
\end{align}
where $\outg{\alpha}(\lambda),\inc{\alpha}(\lambda), \offdiag{\alpha\alpha'}(\lambda)$ and $\diag{\shfparameter}{\alpha}(\lambda)$ are the respective resolvents of the operators introduced in equation~\eqref{eqn:defn_limiting_operators}. In accordance with \cite{gu2021moments}, we label these resolvents and their corresponding semigroups as the limiting outgoing, incoming, offdiagonal and diagonal operators respectively.

\begin{remark}
	The kernels defined in equation~\eqref{eqn:defn_limiting_operators} differ slightly from those in~\cite{gu2021moments} and \cite{sudhtsai25deltabose} by the factor $\sqrrange(\floor{y_{\rel}})$. This does not affect the definition of $\deltatrans{\shfparameter}(t)$ since by our assumption that $\range$ sums to $1$ over the lattice, its piecewise constant extension $y_{\rel} \mapsto \range(\floor{y_{\rel}})$ integrates out to $1$ over $\R^2$ \ie $\norm{\range}_{1} = 1$.
\end{remark}

Before we describe the approximate delta-Bose semigroup for our problem, we need to introduce the $\varepsilon$-versions of the operators defined above. For $\alpha = ij \in \pair\intv{n}$, consider the linear map $\Seps{\alpha} : \R^{2\intv{n}}_{\alpha} \to \R^{2\intv{n}}$ defined as 
\begin{equation}{\label{defn:Seps}}
	\big(\Seps{\alpha}y \big)_{k} \coloneqq  \begin{cases}
		y_{\com}+ \frac{\varepsilon y_{\rel}}{2} &  \text{ when } k = i \\
		
		y_{\com}- \frac{\varepsilon y_{\rel}}{2} &  \text{ when } k = j \\
		
		y_k & \text{ when } k \in \intv{n}\setminus\alpha \, \, .
	\end{cases}
\end{equation} 
This transformation switches from the center-of-mass and relative coordinate frame to the usual frame, while scaling the relative distance by $\varepsilon$. The inverse $\Teps{\alpha}:\R^{2\intv{n}} \to \R^{2\intv{n}}_{\alpha}$ is then given by:
\begin{equation}{\label{defn:Teps}}
	\Teps{\alpha}x \coloneqq \Big(\frac{x_i - x_j}{\varepsilon}, \frac{x_i + x_j}{2}, x_{\intv{n}\setminus \alpha}\Big) \, .
\end{equation}
In our proofs, it will often be convenient to work in Fourier space. For this purpose, we need to know how these maps act there. To this end, define the Fourier dual of the map $\Seps{\alpha}$ as the unique linear transformation $\Sdualeps{\alpha}: \widehat{\R}^{2\intv{n}} \to \widehat{\R}^{2\intv{n}}_{\alpha}$ such that $\Sdualeps{\alpha}\xi \cdot y = \xi \cdot \Seps{\alpha}y $. That is, 
\begin{equation}{\label{defn:Sdualeps}}
	\Sdualeps{\alpha}\xi \coloneqq  \bigg(\frac{\varepsilon (\xi_i - \xi_j)}{2}, \xi_i +\xi_j,\xi_{\intv{n}\setminus\alpha} \bigg) \, .
\end{equation}
Its inverse is also easily calculated as $\Tdualeps{\alpha}: \widehat{\R}^{2\intv{n}}_{\alpha} \to \widehat{\R}^{2\intv{n}}$
\begin{equation}{\label{defn:Tdualeps}}
	\big(\Tdualeps{\alpha}\eta \big)_{k} \coloneqq  
	\begin{cases}
		\frac{\eta_{\com}}{2}+ \frac{\eta_{\rel}}{\varepsilon} &  \text{ when } k = i \\
		
		\frac{\eta_{\com}}{2}- \frac{\eta_{\rel}}{\varepsilon} &  \text{ when } k = j \\
		
		\eta_k & \text{ when } k \in \intv{n}\setminus\alpha  \,.
	\end{cases}
\end{equation} 

We will use the notation $\epstrans(t)$ and $\epstwotrans(t)$ to refer to the semigroups generated by the symbols $\epsgen$ and $\epstwogen$ respectively. Set $\Z^{2\intv{n}}_{\varepsilon\alpha} \coloneqq \Teps{\alpha}\Z^{2\intv{n}}_{\varepsilon} \subset \R^{2\intv{n}}_{\alpha}$. Using this, we can define the discrete analogs of the operators from equation~\eqref{eqn:defn_limiting_operators}, namely the \textit{prelimiting} operators: for $x, x' \in \Z^{2\intv{n}}_{\varepsilon}, y \in \Z^{2\intv{n}}_{\varepsilon\alpha}$ and $y' \in \Z^{2\intv{n}}_{\varepsilon\alpha'}$, 
\begin{equation}{\label{eqn:defn_prelimiting_operators}}
	\begin{split}
		\preoutgtrans{\alpha}(t, x, y) & \coloneqq \epstrans(t, \Seps{\alpha}y-x) \sqrrange(y_{\rel}) \eqqcolon \preinctrans{\alpha}(t, y, x) \\
		\preoffdiagtrans{\alpha\alpha'}(t, y, y') &\coloneqq \sqrrange(y_{\rel})\epstrans(t, \Seps{\alpha}y - \Seps{\alpha'}y') \sqrrange(y'_{\rel})          \\
		\prediagtrans{\theta}{\alpha}(t)	 & \coloneqq \sum_{m=0}^{\infty}\beta_{\varepsilon}^{m+1} \int_{\Sigma(t)} \preoffdiagtrans{\alpha\alpha}(\tau_1)\preoffdiagtrans{\alpha\alpha}(\tau_2) \cdots \preoffdiagtrans{\alpha\alpha}(\tau_m) \, \d \vec{\tau}  \, ;		
	\end{split}
\end{equation}
where, by convention, the $m= 0$ term is defined to $\beta_{\varepsilon} \delta_{0}(t)\id{n}$, with $\delta_{0}$ denoting the Dirac mass at $0$ and $\id{k}$ denotes the identity operator on $\Ltwo(\R^{2\intv{k}})$ or $\Ltwo(\R^{2\intv{k}}_{\alpha})$.

Now we begin to give an expression for the approximate delta-Bose semigroup. By an application of It\^{o}'s formula, the $n$-point correlation kernel, given by
\begin{equation*}
	 \epscorr{\theta}(t-s,x, x') \coloneqq \E\bigg[\prod_{k=1}^{n} \epsshe_{s,t}(x_k,x'_k)\bigg] \, \, ,
\end{equation*} 
can be seen to satisfy the approximate delta-Bose gas equation:
\begin{equation*}
	\partial_{t}\epscorr{\theta}(t,x, x') = (\epsgen\epscorr{\theta})(t,x,x') + \beta_{\varepsilon} \sum_{ij \in \pair\intv{n}}\range_{\varepsilon}(x_i -x_j)\epscorr{\theta}(t,x,x') \, \, .
\end{equation*}
For notational convenience, we also denote, for $x \in \Z^{2\intv{n}}_{\varepsilon}$ and $\alpha = ij \in \pair\intv{n}$,  $x_{\alpha} \coloneqq x_i - x_j$ and let $\range_{\varepsilon}^{\Sigma}(x) \coloneqq \sum_{\alpha \in \pair\intv{n}}\range_{\varepsilon}(x_{\alpha})$. To get an expression for the semigroup kernel $\epscorr{\theta}(t,x, x')$, we first repeatedly apply Duhamel's principle to the equation above whence we get the expression:
\begin{align*}
	\epscorr{\theta}
	(t,x, x') = \epstrans(t, x - x') + 
	& \sum_{m=1}^{\infty}
	 \varepsilon^{2nm}\sum_{\vec{\alpha} \in \pair\intv{n}^{m}} \beta_{\varepsilon}^{m}\sum_{(x^{(k)}) \in \Z_{\varepsilon}^{2\intv{n}}} \int_{\Sigma(t)}\epstrans(\tau_{0}, x-x^{(1)}) \range_{\varepsilon}\big(x_{\alpha_{1}}^{(1)}\big) \\
	& \times \prod_{k=1}^{m-1}\epstrans(\tau_k, x^{(k)} - x^{(k+1)})\range_{\varepsilon}\big(x_{\alpha_{k+1}}^{(k+1)}\big) \times \epstrans(\tau_m, x^{(m)} - x') \, \d \vec{\tau}  .
\end{align*}
In the above, express $\range_{\varepsilon}$ as $\sqrrange_{\varepsilon} \cdot \sqrrange_{\varepsilon}$ and make a change of variables $x^{(k)} = \Seps{\alpha}y^{(k)}$, where $y^{(k)} \in \Z_{\varepsilon\alpha}^{2\intv{n}} \coloneqq \Teps{\alpha}\Z_{\varepsilon}^{2\intv{n}} \subset \R^{2\intv{n}}_{\alpha}$. We also cancel the $\varepsilon^{-1}$ in all the $\sqrrange$ from the $\varepsilon^{2nm}$ factor outside the sum, which then becomes $\varepsilon^{2(n-1)m}$. Thus, for $y^{(k)} \in \Z^{2\intv{n}}_{\varepsilon\alpha_{k}}$, $ \sqrrange_{\varepsilon}(x_{\alpha_{k}}^{(k)})$ changes to $\sqrrange(y_{\rel}^{(k)})$. For $m \geqslant 1$ and $\vec{\alpha} \in \pair\intv{n}^{m}, \vec{\tau} \in \Sigma(t)$, each term in the summation then becomes
\begin{align*}
	&\varepsilon^{2(n-1)m} \beta_{\varepsilon}^{m}  \, \epstrans(\tau_{0}, x -\Seps{\alpha_{1}}y^{(1)}) \sqrrange \big(y_{\rel}^{(1)}\big) \\
	& \, \, \times \prod_{k=1}^{m-1} \sqrrange\big(y_{\rel}^{(k)}\big)
	\epstrans(\tau_k , \Seps{\alpha_{k}}y^{(k)} - \Seps{\alpha_{k+1}}y^{(k+1)}) \sqrrange\big(y_{\rel}^{(k+1)}\big) 
	 \times \sqrrange\big(y_{\rel}^{(m)}\big) \epstrans(\tau_m, \Seps{\alpha_{m}}y^{(m)} - x') \\
	&  \, \,= \varepsilon^{2(n-1)m} \, \preoutgtrans{\alpha_{1}}(\tau_{0}, x, y^{(1)}) \times  \beta_{\varepsilon}^{m}\prod_{k=1}^{m-1} \preoffdiagtrans{\alpha_{k}\alpha_{k+1}}(\tau_{k}, y^{(k)} - y^{(k+1)}) \times \preinctrans{\alpha_{m}}(\tau_{m}, y^{(m)} - x') \, .
\end{align*}
Inserting this, the summation becomes over all $\vec{\eta} = (\alpha_{1}^{k_1}, \alpha_{2}^{k_2}, \ldots ) \in \cup_{m=1}^{\infty}\pair\intv{n}^m$, where $\alpha^k \coloneqq (\alpha, \alpha, \ldots \alpha) \in \pair\intv{n}^{k}$ and $\alpha_1 \neq \alpha_2, \alpha_2 \neq \alpha_3, \ldots \, .$ Moreover, since the unit lattice box in $\Z^{2\intv{n}}_{\varepsilon}$ has volume $\varepsilon^{2n}$, and $\det\Teps{\alpha} = \varepsilon^{-2}$, the transformed unit box in $\Z^{2\intv{n}}_{\varepsilon\alpha}$ has volume $\varepsilon^{2(n-1)}$ which implies that the summation is indeed a convolution in $\Z^{2\intv{n}}_{\varepsilon\alpha}$. The resulting expression is
\begin{align*}
	&\epstrans(t) +  
	\sum_{\vec{\alpha} \in \dgm\intv{n}} 
	 \int_{\Sigma(t)} \preoutgtrans{\alpha_1}(\tau_{0}) \\
	&\times \prod_{k=1}^{|\vec{\alpha}| - 1} \Bigg(\sum_{m_k =0}^{\infty}\beta_{\varepsilon}^{m_k+1} \int_{\Sigma(\tau_{k-\half})}\preoffdiagtrans{\alpha_{k}\alpha_{k}}(\sigma_1)\preoffdiagtrans{\alpha_{k}\alpha_{k}}(\sigma_2) \cdots \preoffdiagtrans{\alpha_{k}\alpha_{k}}(\sigma_{m_k}) \, \d \vec{\sigma}\Bigg) \preoffdiagtrans{\alpha_k \alpha_{k+1}}(\tau_k) \\
	&   \times \Bigg(\sum_{m_{|\vec{\alpha}|} =0}^{\infty}\beta_{\varepsilon}^{m_{|\vec{\alpha}|}+1} \int_{\Sigma(\tau_{|\vec{\alpha}|-\half})}\preoffdiagtrans{\alpha_{|\vec{\alpha}|}\alpha_{|\vec{\alpha}|}}(\sigma_1)\preoffdiagtrans{\alpha_{|\vec{\alpha}|}\alpha_{|\vec{\alpha}|}}(\sigma_2) \cdots \preoffdiagtrans{\alpha_{|\vec{\alpha}|}\alpha_{|\vec{\alpha}|}}(\sigma_{m_{|\vec{\alpha}|}}) \, \d \vec{\sigma}\Bigg) \preinctrans{\alpha_{|\vec{\alpha}|}}(\tau_{|\vec{\alpha}|}) \, \d \vec{\tau} \, .
\end{align*}
 Since the kernel $\epscorr{\theta}(t)$ is defined on the lattice $\Z_{\varepsilon}^{2\intv{n}}$, we extend it to a piecewise constant function on $\R^{2\intv{n}}$, which allows us to study its convergence to $\corr(t)$. Henceforth, for $x, x' \in \R^{2\intv{n}}$, we set $\epscorr{\theta}(t,x,x') \coloneqq \epscorr{\theta}(t, \epsfloor{x}, \epsfloor{x'}), \epstrans(t, x ,x') \coloneqq \epstrans(t, \epsfloor{x} - \epsfloor{x'})$ and $\range(y_{\rel}) \coloneqq \range(\floor{y_{\rel}})$. As the kernels of $\preoutgtrans{\alpha}(t)$, $\preoffdiagtrans{\alpha\alpha'}(t)$ and $\prediagtrans{\theta}{\alpha}(t)$ are defined in terms of $\epstrans(t)$, they can also be extended similarly. Defining
 \begin{equation*}
 	\prediagramtrans{\theta}{\alpha}(t)  \coloneqq \int_{\Sigma(t)}  \preoutgtrans{\alpha_1}(\tau_{0}) \cdot \prod_{k=1}^{|\vec{\alpha}|-1} \prediagtrans{\theta}{\alpha_k}(\tau_{k - \half})\preoffdiagtrans{\alpha_k\alpha_{k+1}}(\tau_k) \cdot \prediagtrans{\theta}{\alpha_{|\vec{\alpha}|}}(\tau_{|\vec{\alpha}|- \half})\preinctrans{\alpha_{|\vec{\alpha}|}}(\tau_{|\vec{\alpha}|}) \, \d \vec{\tau} \, ,
 \end{equation*}
 we get the following compact expression for the prelimiting delta-Bose semigroup:
 \begin{equation*}
 	\predeltatrans{\theta}(t) = \epstrans(t) + \sum_{\vec{\alpha} \in \dgm\intv{n}} \prediagramtrans{\theta}{\alpha}(t) \, \, .
 \end{equation*}
Taking the Laplace transform on both sides, for $\lambda  > 0$,
\begin{align}{\label{eqn:prelimiting_resolvent}}
	\int_{0}^{\infty}e^{-\lambda t}\epscorr{\theta}(t) \, \d t 
	& = \epsgreens(\lambda) + \sum_{\vec{\alpha} \in \dgm\intv{n}}\int_{0}^{\infty}e^{-\lambda t} \prediagramtrans{\theta}{\alpha}(t) \, \d t \nonumber \\
	& = \epsgreens(\lambda) + \sum_{\vec{\alpha} \in \dgm\intv{n}} \preoutg{\alpha_1}(\lambda)\cdot  \prod_{k=1}^{|\vec{\alpha}|-1} \prediag{\theta}{\alpha_k}(\lambda)\preoffdiag{\alpha_k\alpha_{k+1}}(\lambda) \cdot \prediag{\theta}{\alpha_{|\vec{\alpha}|}}(\lambda)\preinc{\alpha_{|\vec{\alpha}|}}(\lambda) ,
\end{align}
where $\preoutg{\alpha}(\lambda), \preinc{\alpha}(\lambda),\preoffdiag{\alpha\alpha'}(\lambda)$ and $\prediag{\theta}{\alpha}(\lambda)$ are the respective \textit{prelimiting} resolvent operators of the extended kernels from equation~\eqref{eqn:defn_prelimiting_operators}. The interchange of the integrals and the summations is justified here as all the kernels involved are non-negative.

We end this section with a lemma that derives norm bounds on the limiting resolvents. As previously mentioned, this was first proved in \cite{gu2021moments} for the case when the correlation function is compactly supported. We include it here so that the article remains self-contained. In the course of the proof, we will derive Fourier representations for the resolvents which will be useful in the proof of Theorem~\ref{thm:maintheorem}. Recall by our chosen convention, the Fourier transform of the Green's function $\green(\lambda)$ is given by:
\begin{equation*}
	\F{\green(\lambda)}(\xi) = \frac{1}{(2\pi)^{n}} \frac{1}{\lambda + \half\norm{\xi}^2} \, .
\end{equation*}
\begin{lemma}{\label{lem:limitingresolventboundedness}}
	Fix $\alpha = ij,\alpha' = kl \in \pair\intv{n}$ such that $\alpha \neq \alpha'$. Then, for all sufficiently $\lambda \gtrsim e^{\shfparameter} + 1$, the limiting operators defined in equation~\eqref{eqn:defn_limiting_operators} satisfy the following norm bounds:
	\begin{subequations}
		\begin{align}
		\normop{\outg{\alpha}(\lambda)} 
		& \lesssim \lambda^{-\half} \label{eqn:outgoing_bound} \\
		\normop{\inc{\alpha}(\lambda)}
		& \lesssim \lambda^{-\half} \label{eqn:incoming_bound} \\
		\normop{\offdiag{\alpha\alpha'}(\lambda)}
		& \lesssim 1               \label{eqn:offdiagonal_bound} \\
		\normop{\diag{\shfparameter}{\alpha}(\lambda)}  
		& \lesssim (\log\lambda - \shfparameter)^{-1}    \, \, .\label{eqn:diagonal_bound}
		\end{align}
	\end{subequations}
	Moreover, with these norm bounds, the sum in equation~\eqref{eqn:limiting_resolvent} converges absolutely in norm.
\end{lemma}
\begin{proof}
	We first prove equation~\eqref{eqn:outgoing_bound}. Fix $f\in \cali{S}(\R^{2\intv{n}})$ and $g \in \cali{S}(\R^{2\intv{n}}_{\alpha})$. Notice that the kernel of $\outgtrans{\alpha}(t)$ acts separately on the center-of-mass and relative coordinates. Because of this, it is natural to consider functions $g$ of the form $g(y'_{\rel}, y'_{\com}, y'_{\intv{n}\setminus\alpha}) = g_{1}(y'_{\rel})g_{2}(y'_{\com \cup \intv{n}\setminus\alpha})$ for $g_{1} \in \cali{S}(\R^2)$ and $g_{2} \in \cali{S}(\R^{2\intv{n}\setminus\rel}_{\alpha})$, so that $\norm{g}_{2}=\norm{g_{1}}_{2} \norm{g_{2}}_{2}$. This assumption poses no loss of generality, since finite linear combinations of such functions are dense in $\cali{L}^{2}(\R^{2\intv{n}}_{\alpha})$, and so it suffices to prove the bound for this class. From the definition of the kernel and $\Sdualzero{\alpha}$ we have, for $\xi \in \widehat{\R}^{2\intv{n}}$,
	\begin{align}{\label{eqn:outgoing_fourier_one}}
		\F{\outg{\alpha}(\lambda)g}(\xi) 
		& = \frac{1}{(2\pi)^{n}} \int_{\R^{2\intv{n}}}\int_{\R^{2\intv{n}}_{\alpha}}  e^{\im \xi \cdot x} \green(\lambda, x - \Szero{\alpha}y'_{\com\cup\intv{n}\setminus\alpha}) \sqrrange(y'_{\rel}) g_{1}(y'_{\rel}) g_{2}(y'_{\com \cup \intv{n}\setminus\alpha}) \,  \d x \, \d y'  \nonumber \\
		& =   \int_{\R^{2\intv{n}}_{\alpha}} \F{\green(\lambda)}(\xi) e^{\im \xi \cdot \Szero{\alpha}y'_{\com\cup\intv{n}\setminus\alpha}} \, \sqrrange(y'_{\rel}) \, g_{1}(y'_{\rel}) \, g_{2}(y'_{\com \cup\intv{n}\setminus\alpha}) \,  \d y' \nonumber \\
		&  = \frac{1}{(2\pi)^{n}} \int_{\R^{2\intv{n}}_{\alpha}}\frac{e^{\im  (\xi_i + \xi_j)  \cdot y'_{\com}+ \im \xi_{\intv{n}\setminus\alpha} \cdot y'_{\intv{n}\setminus\alpha}} \, \sqrrange(y'_{\rel}) \, g_{1}(y'_{\rel}) \, g_{2}(y'_{\com\cup\intv{n}\setminus\alpha}) }{\lambda + \half 
			\norm{\xi}^2} \, \d y' \nonumber \\
		& = \frac{1}{2\pi}  \frac{\F{g_{2}}(\xi_i + \xi_j, \xi_{\intv{n}\setminus\alpha})}{\lambda + \half \norm{\xi}^2}\int_{\R^2} \sqrrange(y'_{\rel})g_{1}(y'_{\rel})  \, \d y'_{\rel} \, \, .
	\end{align}
	Therefore, by Plancherel's theorem,
	\begin{align*}
		\langle f, \outg{\alpha}(\lambda)g \rangle  
		& = \frac{1}{2\pi}\int_{\widehat{\R}^{2\intv{n}}}   \frac{\F{f}(\xi)\conj{\F{g_{2}}(\xi_i + \xi_j, \xi_{\intv{n}\setminus\alpha})}}{\lambda + \half \norm{\xi}^2} \, \d \xi \int_{\R^2} \sqrrange(y'_{\rel})g_{1}(y'_{\rel})  \, \d y'_{\rel} \, \, .
	\end{align*}
	If we make a change of variables $\xi \mapsto \Tdualone{\alpha}\eta$, the domain changes from $\widehat{\R}^{2\intv{n}} \to \widehat{\R}^{2\intv{n}}_{\alpha}$ and $\xi_i + \xi_j \mapsto \eta_{\com}$.  The right-hand side of the preceding equation then becomes
	\begin{equation*}
		\frac{1}{2\pi}\int_{\widehat{\R}^{2\intv{n}}_{\alpha}}\frac{\F{f}(\Tdualone{\alpha}\eta)\conj{\F{g_{2}}(\eta_{\com\cup\intv{n}\setminus\alpha})}}{\lambda + \norm{\eta_{\rel}}^2 + \quart\norm{\eta_{\com}}^2 + \half \norm{\eta_{\intv{n}\setminus\alpha}}^2} \, \d \eta \int_{\R^2}\sqrrange(y_{\rel})g_{1}(y_{\rel}) \, \d y_{\rel} \, \, .
	\end{equation*}
	The innermost $y'_{\rel}$ integral above can be bounded by $\norm{\sqrrange}_{2} \norm{g_{1}}_{2} = \sqrt{\norm{\range}_{1}}\norm{g_{1}}_{2}$ using the Cauchy--Schwarz inequality. As for the $\eta$ integral, we use the same inequality along with the fact that $\lambda + \norm{\eta_{\rel}}^2 + \quart\norm{\eta_{\com}}^2 + \half \norm{\eta_{\intv{n}\setminus\alpha}}^2 \geqslant \lambda + \norm{\eta_{\rel}}^2$. Consequently, this yields:
	\begin{equation*}
		\Bigg|\int_{\widehat{\R}^{2\intv{n}}}\frac{\F{f}(\Tdualone{\alpha}\eta)\F{g_{2}}(\eta_{\com\cup\intv{n}\setminus\alpha})}{\lambda + \norm{\eta_{\rel}}^2 + \quart\norm{\eta_{\com}}^2 + \half \norm{\eta_{\intv{n}\setminus\alpha}}^2} \, \d \eta\Bigg| \lesssim \norm{f}_{2} \norm{g_{2}}_{2} \Big(\int_{\widehat{\R}^2} \frac{\d \eta_{\rel}}{(\lambda + \norm{\eta_{\rel}}^2)^2}\Big)^{\half} \, .
	\end{equation*}
	The $\eta_{\rel}$ integral can be explicitly computed by switching to polar coordinates:
	\begin{equation}{\label{eqn:outgoing_polar}}
		\int_{\widehat{\R}^2} \frac{\d \eta_{\rel}}{(\lambda + \norm{\eta_{\rel}}^2)^2} = \pi\int_{0}^{\infty} \frac{2r \, \d r}{(\lambda + r^2)^2} = \pi\int_{\lambda}^{\infty} \frac{\d r}{r^2} = \frac{\pi}{\lambda} \, .
	\end{equation}
	Combining the bounds, we get:
	\begin{equation*}
		\normop{\outg{\alpha}(\lambda)} = \sup_{\norm{f}_{2}, \norm{g}_{2} \leqslant 1} |\langle f, \outg{\alpha}(\lambda)g \rangle | \lesssim \lambda^{-\half} \, .
	\end{equation*}
	This proves equation~\eqref{eqn:outgoing_bound}. The result for $\inc{\alpha}(\lambda)$ follows by adjointness. Now we prove equation~\eqref{eqn:offdiagonal_bound}. For this case, assume $f \in \Ltwo(\R^{2\intv{n}}_{\alpha})$ instead, and of the form $f_{1}f_{2}$ like the function $g$. From the definition it is easy to see that $\langle f, \offdiag{\alpha\alpha'}(\lambda)g \rangle = \langle \outg{\alpha}(\lambda)f, \outg{\alpha'}(\lambda)g \rangle$. Hence, using Plancherel's isometry and equation~\eqref{eqn:outgoing_fourier_one}, we have that $\langle f, \offdiag{\alpha\alpha'}(\lambda)g \rangle$ is equal to:
	\begin{equation}{\label{eqn:offdiagonal_fourier_one}}
		\frac{1}{(2\pi)^{2}}\int_{\R^2} \int_{\R^2}  \int_{\widehat{\R}^{2\intv{n}}} f_{1}(y_{\rel})\sqrrange(y_{\rel}) \frac{\F{f_{2}}(\xi_i + \xi_j, \xi_{\intv{n}\setminus\alpha})\conj{\F{g_{2}}(\xi_k + \xi_l, \xi_{\intv{n}\setminus\alpha'})}}{\lambda + \half \norm{\xi}^{2}} g_{1}(y'_{\rel}) \sqrrange(y'_{\rel}) \, \d \xi \, \d y_{\rel}  \, \d y'_{\rel}  \ .
	\end{equation}
	To bound the integral on the right-hand side, we use Lemma 5.1 from \cite{gu2021moments}. Since we will use this lemma again in our proof of the main theorem, we state it here.
	\begin{lemma}{\label{lem:guquasteltsailemma}}
		Fix $f_{2}, g_{2} \in \Ltwo(\R^{2\intv{n}\setminus\rel}_{\alpha})$. Then for $ \alpha = ij, \alpha' =kl \in \pair\intv{n}$ such that $\alpha \neq \alpha'$, we have
		\begin{equation*}
			\sup_{\lambda > 0} \int_{\widehat{\R}^{2\intv{n}}}\frac{|\F{f_{2}}(\xi_i + \xi_j, \xi_{\intv{n}\setminus\alpha})| \, |\F{g_{2}}(\xi_k + \xi_l, \xi_{\intv{n}\setminus\alpha'})|}{\lambda + \half\norm{\xi}^2} \, \d \xi \lesssim \norm{f_{2}}_{2} \norm{g_{2}}_{2} \, \, .
		\end{equation*}
	\end{lemma}
	Using this lemma, the  $\xi$ integral in equation~\eqref{eqn:offdiagonal_fourier_one} is bounded uniformly in $\lambda > 0$ by $\norm{f_{2}}_{2}\norm{g_{2}}_{2}$ up to a constant. As for the $\d y_{\rel}$ and $\d y'_{\rel}$ integrals, a simple application of Cauchy--Schwarz bounds them by $\norm{f_{1}}_{2}$ and $\norm{g_{1}}_{2}$ respectively. Therefore,
	\begin{equation*}
		\normop{\offdiag{\alpha\alpha'}} = \sup_{\norm{f}_{2},\norm{g}_{2} \leqslant 1} |\langle f, \offdiag{\alpha\alpha'}g \rangle | \lesssim \norm{f_{1}}_{2}\norm{f_{2}}_{2} \, \norm{g_{1}}_{2}\norm{g_{2}}_{2} = \norm{f}_{2}\norm{g}_{2},
	\end{equation*}
	which proves equation~\eqref{eqn:offdiagonal_bound}. Finally, we proceed to equation~\eqref{eqn:diagonal_bound}. Just as in the previous cases,  we will bound the operator norm by passing to the Fourier space. We see from the definition of the kernel that the operator acts as a convolution in the $y_{\com}, y_{\intv{n}\setminus\alpha}$ variables, and as multiplication in the $y_{\rel}$ variable. For this reason, it will be convenient to take the Fourier transform in the former variables so that it becomes a product. That is, 
	\begin{align}{\label{eqn:diagonal_fourier_one}}
		& \langle f, \diag{\shfparameter}{\alpha}(\lambda) g \rangle =\int_{\R^{2\intv{n}}_{\alpha}} \d y \int_{\R^{2\intv{n}}_{\alpha}} \d y' f(y) \diag{\shfparameter}{\alpha}(\lambda, y, y') g(y') \nonumber  \\
		& = \int_{\R^{2}} \d y_{\rel} \int_{\R^2} \d y'_{\rel} \int_{\R^{2\intv{n}\setminus\rel}_{\alpha}} \d \eta_{\com \cup \intv{n} \setminus \alpha} \, f_{1}(y_{\rel}) \F{f_{2}}(\eta_{\com \cup\intv{n}\setminus\alpha}) \F{\diag{\shfparameter}{\alpha}(\lambda)}(y_{\rel}, y'_{\rel}, \eta_{\com \cup \intv{n}\setminus\alpha})\F{g_{2}}(\eta_{\com \cup \intv{n}\setminus\alpha}) g_{1}(y'_{\rel}) \, , 
	\end{align}
	where $\F{\diag{\shfparameter}{\alpha}(\lambda)}(y_{\rel}, y'_{\rel}, \eta_{\com \cup \intv{n} \setminus \alpha})$ denotes the \textit{partial} Fourier transform of the operator $\diag{\shfparameter}{\alpha}(\lambda)$ in the $y_{\com}, y_{\intv{n}\setminus\alpha}$ variables. We now evaluate this partial Fourier transform. From the definition of the kernel in equation~\eqref{eqn:defn_limiting_operators}, we have:
	\begin{equation}{\label{eqn:diagonal_partial_fourier_one}}
		\F{\diag{\shfparameter}{\alpha}(\lambda)}(y_{\rel}, y'_{\rel},\eta_{\com \cup \intv{n}\setminus\alpha}) =  4\pi\sqrrange(y_{\rel})\sqrrange(y'_{\rel})\int_{0}^{\infty}  \int_{0}^{\infty} \frac{e^{-\lambda t}  e^{\shfparameter u} t^{u-1}e^{-t\big(\quart\norm{\eta_{\com}}^2 + \half\norm{\eta_{\intv{n}\setminus\alpha}}^2\big)}}{\Gamma(u)} \, \d u \, \d t \, .
	\end{equation}
	After a change of variables 
	\begin{equation*}
		\big(\lambda + \quart\norm{\eta_{\com}}^2 + \half\norm{\eta_{\intv{n}\setminus\alpha}}^2\big) t \mapsto t \, \, ,
	\end{equation*} 
	the $\d t$ integral evaluates to $\Gamma(u)$ which cancels the same in the denominator. This change of variables introduces a factor of $ (\cdot ) \coloneqq \big(\lambda + \quart\norm{\eta_{\com}}^2 + \half\norm{\eta_{\intv{n}\setminus\alpha}}^2\big)^{-u}$ which we write as $e^{-\log (\cdot)}$. Hence, for large enough $\lambda >0$, we have
	\begin{align}{\label{eqn:diagonal_partial_fourier_two}}
		\F{\diag{\shfparameter}{\alpha}(\lambda)}(y_{\rel}, y'_{\rel},\eta_{\com \cup \intv{n}\setminus\alpha}) 
		& =  4\pi\sqrrange(y_{\rel})\sqrrange(y'_{\rel})\int_{0}^{\infty}  e^{-u\big[\log \big(\lambda + \quart\norm{\eta_{\com}}^2 + \half\norm{\eta_{\intv{n}\setminus\alpha}}^2\big) - \shfparameter\big]} \, \d u \nonumber \\
		& =  \frac{4\pi\sqrrange(y_{\rel})\sqrrange(y'_{\rel})}{\log \big(\lambda + \quart\norm{\eta_{\com}}^2 + \half\norm{\eta_{\intv{n}\setminus\alpha}}^2\big) - \shfparameter} \\
		& \lesssim \frac{4\pi\sqrrange(y_{\rel})\sqrrange(y'_{\rel})}{\log \lambda - \shfparameter} \, \, . \nonumber
	\end{align}
	Plugging this back into equation~\eqref{eqn:diagonal_fourier_one} and using the  Cauchy--Schwarz inequality in the $y_{\rel}, y'_{\rel}$ and $\eta_{\com \cup \intv{n} \setminus \alpha}$ variables yields:
	\begin{equation*}
		\normop{\diag{\shfparameter}{\alpha}(\lambda)} = \sup_{\norm{f}_{2}, \norm{g}_{2} \leqslant 1} |\langle f, \diag{\shfparameter}{\alpha}(\lambda)g \rangle | \lesssim (\log\lambda - \shfparameter)^{-1} \, .
	\end{equation*}
	From equation~\eqref{eqn:diagonal_partial_fourier_two}, we see that the $y_{\rel}, y'_{\rel}$ component of the kernel is just the projection onto the subspace generated by the function $\sqrrange$. That is, if we denote $\proj: \Ltwo(\R^2) \to \Ltwo(\R^2)$ to be the projection onto this subspace defined by
	\begin{equation*}
		(\proj \,  g_{1})(y_{\rel}) = \frac{1}{\norm{\range}_{1}}\sqrrange(y_{\rel}) \int_{\R^2} g_{1}(y'_{\rel}) \sqrrange(y'_{\rel}) \, \d y'_{\rel}
	\end{equation*}
	then for fixed $\eta_{\com \cup \intv{n} \setminus \alpha} \in \widehat{\R}^{2\intv{n}\setminus\rel}_{\alpha}$, the operator $\F{\diag{\shfparameter}{\alpha}(\lambda)}(\eta_{\com \cup \intv{n} \setminus \alpha}): \Ltwo(\R^2) \to \Ltwo(\R^2)$ is defined as 
	\begin{equation}{\label{eqn:diagonal_projection_form}}
		\F{\diag{\shfparameter}{\alpha}(\lambda)}(\eta_{\com \cup \intv{n} \setminus \alpha}) \coloneqq \frac{4\pi}{\log(\lambda + \quart \norm{\eta_{\com}}^2 + \half\norm{\eta_{\intv{n} \setminus \alpha}}^2) - \shfparameter} \proj \, \, \, ,
	\end{equation}
	with kernel $\F{\diag{\shfparameter}{\alpha}(\lambda)}(y_{\rel}, y'_{\rel},\eta_{\com \cup \intv{n}\setminus\alpha})$ defined above.
	
	Using these norm bounds, let us show the limiting resolvent in equation~\eqref{eqn:limiting_resolvent} converges in norm. For a fixed $m \in \Z_{>0}$, the set $\{\vec{\alpha} \in \dgm\intv{n} \, : \, |\vec{\alpha}| = m\}$ has cardinality at most $\binom{n}{2}^{m} \lesssim 2^{-m}n^{2m}$. Inserting the above bounds in equation~\eqref{eqn:limiting_resolvent}, we obtain the following bound:
	\begin{equation*}
		\NORmop{\int_{0}^{\infty}e^{-\lambda t} \diagramtrans{\shfparameter}{\alpha}(t) \, \d t \, } \lesssim \sum_{m=1}^{\infty}\frac{2^{-m}n^{2m}}{\lambda (\log\lambda -\shfparameter )^{m}} \, . 
	\end{equation*}
	For large enough $\lambda$, say of order greater than $e^{n^2}$, the above series is convergent. Moreover, the free resolvent $\green(\lambda)$ satisfies the norm bound $\normop{\green(\lambda)} \leqslant \lambda^{-1}$ for $\lambda > 0$. Combining this with the bound shown above, we get the same for the resolvent of $\deltatrans{\shfparameter}$. This completes the proof of the lemma.
\end{proof}

\section{Convergence of the resolvent}{\label{section:convergenceofresolvent}}

In this section, we will prove Theorem~\ref{thm:maintheorem}. As noted in the introduction, the self-adjointness of the generator implies that convergence of the resolvent suffices to deduce convergence of the semigroup. We will first show that the prelimiting incoming/outgoing, off-diagonal and diagonal resolvents are bounded and converge to their respective limiting operators. Together with a uniform norm bound, this will ensure the convergence of the entire prelimiting resolvent to its limiting counterpart. This is summarized in the following lemma.
\begin{lemma}{\label{lem:resolventconvergence}}
	 Fix $\alpha, \alpha' \in \pair\intv{n}$ with $\alpha \neq \alpha'$. Then, there exists a $\shfparameter \in \R$ for which the following holds: for every $\lambda \gtrsim e^{2\theta+ 4\pi\quadrange} + 1$, there exists $\varepsilon_{0}(\lambda) > 0$ such that for all $\varepsilon < \varepsilon_{0}(\lambda)$, the prelimiting resolvents satisfy the following:
	\begin{subequations}
	\begin{align}
	\normop{\preinc{\alpha}(\lambda)}  & \lesssim \lambda^{-\half} \qquad \qquad \qquad   \text{ and } \qquad \qquad \quad \normop{\preinc{\alpha}(\lambda) -\inc{\alpha}(\lambda)} \to 0, \label{eqn:preincoming_bound}   \\ 
	\normop{\preoutg{\alpha}(\lambda)} & \lesssim \lambda^{-\half} \qquad \qquad \qquad   \text{ and } \qquad \qquad \quad \normop{\preoutg{\alpha}(\lambda) -\outg{\alpha}(\lambda)} \to 0, \label{eqn:preoutgoing_bound} \\
	\normop{\preoffdiag{\alpha\alpha'}(\lambda)} & \lesssim 1 \qquad \qquad \qquad  \quad \, \text{  and } \qquad \qquad \quad \normop{\preoffdiag{\alpha\alpha'}(\lambda) -\offdiag{\alpha\alpha'}(\lambda)} \to 0, \label{eqn:preoffdiagonal_bound} \\
	\normop{\prediag{\theta}{\alpha}(\lambda)} & \lesssim (\log\lambda - \theta )^{-1}  \quad \quad \,   \text{ and } \qquad \qquad \quad \prediag{\theta}{\alpha}(\lambda) -\diag{\shfparameter}{\alpha}(\lambda) \to 0 \, \, \text{strongly}\, \, . \label{eqn:prediagonal_bound}
	\end{align}
	\end{subequations}
	Moreover, the series in equation~\eqref{eqn:prelimiting_resolvent} converges absolutely in operator norm, and the whole expression converges strongly to the limiting resolvent defined in equation~\eqref{eqn:limiting_resolvent}. 
\end{lemma}
\begin{proof}
	First, suppose the above norm bounds and convergence results hold true. Under these assumptions, the norm bounds are the same as their limiting counterparts. Hence, absolute convergence in norm follows exactly as in Lemma~\ref{lem:limitingresolventboundedness}. Strong convergence follows from the following standard fact from functional analysis: if $\{A_{\varepsilon}\}$, $\{B_{\varepsilon}\}$ are families of uniformly bounded operators the operator $A_{\varepsilon}B_{\varepsilon}$ converges to $AB$ if $\sup_{\varepsilon > 0} \normop{A_{\varepsilon}} + \normop{B_{\varepsilon}} < \infty$, and $A_{\varepsilon} \to A$, $B_{\varepsilon} \to B$ strongly.
\end{proof}

Given this lemma, our task has been reduced to proving the bounds and establishing the convergence results. Before we proceed to do the same, we establish another lemma regarding the lattice Fourier symbol $\fourierepsgen(\xi)$, which will be a key ingredient in our proofs.

\begin{lemma}{\label{lem:uniformgreensbound}}
	Recall that $\epsgen = \sum_{k=1}^{n}\epstwogen^{(k)}$. Then the following inequality holds:
	\begin{equation*}
		\frac{1}{\lambda + \fourierepsgen(\xi)} \lesssim \frac{1}{\lambda + \half \norm{\xi}^2} \, ,
	\end{equation*}
	uniformly over $\xi \in \epshattorus{2\intv{n}}$ and $\varepsilon > 0$. Moreover, we have the bound
	\begin{equation*}
     \Big|\fourierepsgen(\xi) - \half\norm{\xi}^2 \,  \Big| \lesssim \varepsilon^{\delta}\norm{\xi}^{2+\delta} \, \, ,
	\end{equation*}
	where $\delta$ is as in Assumption~\ref{assumption2}. In particular, this implies that the symbol $\fourierepsgen(\xi)$ converges to $\half \norm{\xi}^2$ uniformly on compact sets.
\end{lemma}
\begin{proof}
    Observe that $\fourierepsgen(\xi) = \varepsilon^{-2}\fouriergen(\varepsilon \xi)$. Hence, it suffices to prove the corresponding result for $\fouriergen$. Furthermore, since $\fouriergen(\xi) = \sum_{k=1}^{n}\F{\twogen^{(k)}}(\xi_k)$ for $\xi = (\xi_k)_{k \in \intv{n}}$, it is enough to show that $\F{\twogen}^{(k)}(\xi_k) \gtrsim\norm{\xi_k}^2$ for $\xi_k \in \widehat{\mathbb{T}}^{2}$ and then rescale to obtain the desired inequality. Recall
	\begin{equation*}
	    \F{\twogen}^{(k)}(\xi_k) = \sum_{m_k \in \Z^2}(1 - \cos(\xi_k \cdot m_k) )\nu(m_k) \, .
	\end{equation*}
	Fix $R \in \Z$ and let $\diff_{R} \coloneqq \sum_{\norm{m_k} \leqslant R} m_k \transpose{m_k} \, \nu(m_k)$. Then, $\forall \, \xi_k \in \hattorus{2}$,
	\begin{equation*}
		\bigg| \xi_k \cdot \bigg( \, \lid{1} - \diff_{R}\bigg) \xi_k  \bigg| \leqslant \sum_{\norm{m_k} > R} |  \xi_k \cdot  m_k \transpose{m_k} \xi_k  | \, \nu(m_k) \leqslant \norm{\xi_k}^2  \sum_{\norm{m_k} > R} \norm{m_k}^2 \nu(m_k) \, \, .
	\end{equation*}
   The right-hand side of the above equation is the tail of a summable series. Hence, we may take the supremum over the set $\{ \xi_k \in \hattorus{2} : \norm{\xi_k} = 1\}$ and let $R \to \infty$ to see that $\Normop{\lid{1} - \diff_{R}} \to 0$. Consequently, there exists an $R_{0} > 0$ sufficiently large such that $ \xi_k \cdot  \diff_{R} \xi_k \geqslant  \half \norm{\xi_k}^2$ for all $\xi_k \in \widehat{\mathbb{T}}^{2}$ and $R \geqslant R_{0}$. Now assume $\xi_k$ is such that $\norm{\xi_k} \leqslant \nicefrac{\pi}{2R_{0}}$. For $\norm{m} \leqslant R_{0}$, we have $| \xi_k \cdot m_k | \leqslant \norm{\xi_k}\norm{m_k} \leqslant \nicefrac{\pi}{2}$. Therefore, in this region, using the inequality $1-\cos x \geqslant \nicefrac{2x^2}{\pi^2}$, we have:
   \begin{align*}
   	\F{\twogen}^{(k)}(\xi) \geqslant  \sum_{\norm{m_k} \leqslant R_{0}}(1 - \cos (\xi_k \cdot m_k))  \nu(m_k)
   	&\geqslant \frac{2}{\pi^2} \sum_{\norm{m_k} \leqslant R_{0}} (\xi_k \cdot m_k)^2 \nu(m_k)\\ 
   	& = \frac{2}{\pi^2} \xi_k \cdot \diff_{R_{0}} \xi_k \geqslant \frac{1}{\pi^2} \norm{\xi_k}^2 \, \, ,
   \end{align*} 
   for all $\xi_k$ such that $\norm{\xi_k} \leqslant \nicefrac{\pi}{2R_{0}}$. In the compact complement region $\hattorus{2} \cap \{\xi_k: \norm{\xi_k} \ge \nicefrac{\pi}{2R_{0}}\}$, the continuous function $\F{\twogen}^{(k)}$ attains its minimum value. Due to the aperiodicity assumption, this minimum is strictly positive. That is, there exists $L > 0$ such that $\F{\twogen}^{(k)}(\xi_k) \geqslant L$ for all $\xi_k \in \hattorus{2} \cap \{\xi_k: \norm{\xi_k} \geqslant \nicefrac{\pi}{2R_{0}}\}$. Multiplying and dividing the right-hand side of this inequality by $\norm{\xi_k}^2$, we obtain
   \begin{equation*}
   	 \F{\twogen}^{(k)}(\xi_k) \geqslant \frac{L \norm{\xi_k}^2}{\norm{\xi_k}^2} \geqslant \frac{2L}{\pi^2} \norm{\xi_k}^2.
   \end{equation*}
   Combining the above inequalities, we can write
   \begin{equation*}
   	 \F{\twogen}^{(k)}(\xi_k) \geqslant \bigg(\frac{1}{\pi^2} \wedge \frac{2L}{\pi^2}\bigg) \norm{\xi_k}^2,  \quad \text{ for all } \xi_k \in \hattorus{2} \, .
   \end{equation*}
   Making the change of variables $\xi_k \mapsto \varepsilon \xi_k$ and summing over $k$ yields
   \begin{equation*}
   	 \fourierepsgen(\xi) = \varepsilon^{-2} \fouriergen(\varepsilon \xi) \geqslant \bigg(\frac{1}{\pi^2}\wedge \frac{2L}{\pi^2}\bigg) \norm{\xi}^2,  \quad \text{ for all } \xi \in \epshattorus{2\intv{n}} \, .
   \end{equation*}   
   Adding $\lambda$ to both sides proves the required inequality. We now turn to the second part of the lemma. First, we show the result for $\fourierepstwogen(\xi_k)$.  Using the elementary bounds $|\sin \big(\frac{\varepsilon m_k\cdot \xi_k}{2}\big) - \frac{\varepsilon m_k \cdot \xi_k}{2}| \leqslant 2|\frac{\varepsilon m_k \cdot \xi_k}{2}|^{1 + \delta}$ and $|\sin \big(\frac{\varepsilon m_k \cdot \xi_k}{2}\big) + \frac{\varepsilon m_k \cdot \xi_k}{2}| \leqslant 2|\frac{\varepsilon m_k \cdot \xi_k}{2}|$, we get 
   \begin{equation*}
   \bigg|\sin^2\bigg(\frac{\varepsilon m_k \cdot \xi_k}{2}\bigg) - \bigg|\frac{\varepsilon m_k \cdot \xi_k}{2}\bigg|^{2}\bigg| \leqslant 4\bigg|\frac{\varepsilon m_k \cdot \xi_k}{2}\bigg|^{2+\delta}.
   \end{equation*}
   Note $1-\cos(\varepsilon\xi_k\cdot m_k) = 2 \sin^2(\frac{\varepsilon\xi_k \cdot m_k}{2})$ and from Assumption~\ref{assumption3}, $\norm{\xi_k}^2 = \sum_{m_k \in \Z^{2}}|\xi_k \cdot m_k|^2\nu(m_k)$. Using this and the above inequality,
   \begin{align*}
    \Big|\fourierepstwogen(\xi_k) - \half \norm{\xi_k}^2\Big| 
    & = \bigg| 2\varepsilon^{-2}\sum_{m_k \in \Z^{2}} \sin^2\bigg(\frac{\varepsilon\xi_k\cdot m_k}{2}\bigg)\nu(m_k) - 2\varepsilon^{-2}\sum_{m_k \in \Z^{2}}\frac{|\varepsilon\xi_k \cdot m_k|^2}{4}\nu(m_k)\bigg| \\
    &\lesssim \varepsilon^{-2}\sum_{m_k \in \Z^{2}}\varepsilon^{2 + \delta}\norm{\xi_k}^{2+\delta}\norm{m_k}^{2+\delta}\nu(m_k) \\
    & \lesssim \varepsilon^{\delta}\norm{\xi_k}^{2+\delta} \, \, .
   \end{align*} 
    An application of triangle inequality and the fact that $\sum_{k=1}^{n}\norm{\xi_k}^{2+\delta} \lesssim \norm{\xi}^{2+\delta}$ completes the proof.
\end{proof}

The rest of this section is dedicated to proving the bounds in Lemma~\ref{lem:resolventconvergence}. For the convenience of the reader, we split the proof into several subsections.

\subsection{Incoming and Outgoing Operators}

In this subsection, we prove the results in equations~\eqref{eqn:preincoming_bound} and \eqref{eqn:preoutgoing_bound}. It suffices to prove the result for $\preoutg{\alpha}(\lambda)$, since the corresponding result for $\preinc{\alpha}(\lambda)$ follows immediately by adjointness. Fix $\alpha \in \pair\intv{n}, f\in \cali{S}(\R^{2\intv{n}})$ and $g \in \cali{S}(\R^{2\intv{n}}_{\alpha})$. Following the proof of Lemma~\ref{lem:limitingresolventboundedness}, we may assume that $g$ is of the form
$g(y'_{\rel}, y'_{\com}, y'_{\intv{n}\setminus\alpha}) \coloneqq g_{1}(y'_{\rel})g_{2}(y'_{\com\cup\intv{n}\setminus\alpha})$, so that $\norm{g}_{2} = \norm{g_{1}}_{2}\norm{g_{2}}_{2}$. The strategy we employ for proving the operator norm bound is similar to that used in Lemma~\ref{lem:limitingresolventboundedness}: we pass to Fourier space and estimate it there. To show the convergence in norm, we will compare the operator to its limiting counterpart by decomposing the difference into a sum of error terms and show that these vanish in the limit $\varepsilon \to 0$, uniformly in $f$ and $g$. To that end, consider the inner product
	\begin{equation*}
	\langle f, \preoutg{\alpha}(\lambda)g \rangle = \int_{\R^{2\intv{n}}}  \int_{\R^{2\intv{n}}_{\alpha}}  \, f(x) \conj{\preoutg{\alpha}(\lambda, \epsfloor{x} - \epsfloor{\Seps{\alpha}y'})}\conj{\sqrrange(y'_{\rel}) g(y')} \, \d y' \, \d x \, \, .
	\end{equation*}

    First, we make the substitution $\Seps{\alpha}y' = x'$. This maps the domain $\R^{2\intv{n}}_{\alpha}$ onto $\R^{2\intv{n}}$ and $\d y' \mapsto \varepsilon^{-2} \d x'$. Letting $G_{\varepsilon\alpha}(x') \coloneqq \varepsilon^{-2}(\sqrrange g \circ \Teps{\alpha})(x')$ for ease of notation, we find that the inner product can be equivalently expressed as:
	\begin{equation}{\label{eqn:preoutgoing_position_zero}}
	\langle f, \preoutg{\alpha}(\lambda)g \rangle =  \int_{\R^{2\intv{n}}}  \int_{\R^{2\intv{n}}} \, f(x) \conj{\epsgreens(\lambda, \epsfloor{x}-\epsfloor{x'})} \conj{G_{\varepsilon\alpha}(x')} \, \d x \, \d x' \, \, .
	\end{equation}
	Here, we would like to use Plancherel's theorem to move to Fourier space just as we did in Lemma~\ref{lem:limitingresolventboundedness}, but due to the periodicity of the kernel, the resulting expression is harder to analyze with the arguments used there. To circumvent this issue, we first replace the lattice kernel by its \textit{continuous modification} defined as
	\begin{equation*}
	\epscontgreens(\lambda, x) \coloneqq \frac{1}{(2\pi)^{2n}}\int_{\epshattorus{2\intv{n}}} \frac{e^{-\im \xi \cdot x}}{\lambda + \fourierepsgen(\xi)} \, \d \xi \, \, .
	\end{equation*}
	This kernel has the same Fourier multiplier as the lattice Green's function, but has the advantage of depending on the continuous variable $x$ rather than a discretized argument involving floor functions. Moreover, since its Fourier transform is compactly supported, it is also smooth. With this definition, we can rewrite equation~\eqref{eqn:preoutgoing_position_zero} as:
	\begin{equation}{\label{eqn:preoutgoing_position_one}}
		\int_{\R^{2\intv{n}}}  \int_{\R^{2\intv{n}}}  \, f(x) \conj{\epscontgreens(\lambda, x-x')}\conj{G_{\varepsilon\alpha}(x')} \, \d x \, \d x' + \cali{I}(\varepsilon, f, g) \, ,
	\end{equation}
	where $\cali{I}(\varepsilon, f, g)$ is the error incurred when we replaced the kernel $\epsgreens(\lambda)$ by $\epscontgreens(\lambda)$. That is, 
	\begin{equation}{\label{eqn:preoutgoingerror_position}}
	\cali{I}(\varepsilon, f, g) \coloneqq \int_{\R^{2\intv{n}}}  \int_{\R^{2\intv{n}}}  \, f(x)\big[\conj{\epsgreens(\lambda, \epsfloor{x} - \epsfloor{x'}) - \epscontgreens(\lambda, x-x')}\big] \conj{G_{\varepsilon \alpha}(x')} \, \d x \, \d x' \, .
	\end{equation}
    The first term in equation~\eqref{eqn:preoutgoing_position_zero} above can be bounded directly by passing to Fourier space, as in the proof of Lemma~\ref{lem:limitingresolventboundedness}. The error term requires a separate argument and will therefore be estimated by a more direct method. We begin with the former. Using Plancherel's theorem, we may write it as:
	\begin{equation}{\label{eqn:preoutgoingcont_fourier}}
		\int_{\R^{2\intv{n}}}  \int_{\R^{2\intv{n}}}  \, f(x) \conj{\epscontgreens(\lambda, x-x')}\conj{G_{\varepsilon\alpha}(x')} \, \d x \, \d x' = \int_{\epshattorus{2\intv{n}}} \frac{\F{f}(\xi)\conj{\F{G_{\varepsilon\alpha}}(\xi)}}{\lambda + \fourierepsgen(\xi)} \, \d \xi \, \, .
	\end{equation}
    
    To compute $\F{G_{\varepsilon\alpha}}(\xi)$, we first make the substitution $x' = \Seps{\alpha}y'$ (so that $\d x' = \varepsilon^{2}\d y'$) in its definition, and use the identity $ \xi \cdot \Seps{\alpha} x = \Sdualeps{\alpha}\xi \cdot x$ from equation~\eqref{defn:Sdualeps}. Thus,
	\begin{align*}
	\F{G_{\varepsilon\alpha}}(\xi) 
	& = \frac{\varepsilon^{-2}}{(2\pi)^{n}}\int_{\R^{2\intv{n}}} (\sqrrange g)(\Teps{\alpha}x') e^{\im \xi \cdot x'} \, \d x' \\
	& = \frac{\varepsilon^{-2} \cdot\varepsilon^2}{(2\pi)^{n}} \int_{\R^{2\intv{n}}_{\alpha}} \sqrrange(y'_{\rel}) g_{1}(y'_{\rel})g_{2}(y'_{\com\cup\intv{n}\setminus\alpha}) e^{\im \Sdualeps{\alpha}\xi \cdot y'} \, \d y' \\
	& = \frac{ \F{g_{2}}(\xi_i + \xi_j, \xi_{\intv{n}\setminus\alpha})}{2\pi} \int_{\R^2} \sqrrange(y'_{\rel}) g_{1}(y'_{\rel}) e^{\im \frac{\varepsilon}{2}(\xi_i-\xi_j)\cdot y'_{\rel}} \, \d y'_{\rel} \, \, \, .
	\end{align*}
	Insert this expression into equation~\eqref{eqn:preoutgoingcont_fourier}. Taking absolute values and applying Lemma~\ref{lem:uniformgreensbound}, we obtain
	\begin{align*}
		\Bigg|\,\int_{\R^{2\intv{n}}}  \int_{\R^{2\intv{n}}}  \, f(x) 
		& \conj{\epscontgreens(\lambda, x-x') G_{\varepsilon\alpha}(x')} \, \d x \, \d x' \, \Bigg| \\
		& \lesssim \int_{\epshattorus{2\intv{n}}} \, \int_{\R^2} \frac{|\F{f}(\xi)| \, |\F{g_{2}}(\xi_i + \xi_j, \xi_{\intv{n}\setminus\alpha})|\sqrrange(y'_{\rel}) |g_{1}(y'_{\rel})|}{\lambda + \half\norm{\xi}^2} \, \, \d y'_{\rel} \, \d \xi \, \, .
	\end{align*}
    To estimate this integral, make the change of variables $\xi = \Tdualone{\alpha}\eta$. This transforms the domain $\epshattorus{2\intv{n}}$ into $\epsalhattorus{2\intv{n}}$, and maps $(\frac{\xi_i-\xi_j}{2},\xi_i+\xi_j)$ to $(\eta_{\rel},\eta_{\com})$. Consequently,
	\begin{equation*}
		\half\norm{\xi}^2 \mapsto \norm{\eta_{\rel}}^2 + \quart\norm{\eta_{\com}}^2 + \half\norm{\eta_{\intv{n}\setminus\alpha}}^2 \, \, .
	\end{equation*}
	Then, bound the $\eta$ integrand as follows:
	\begin{equation*}
	\frac{|\F{f}(\Tdualone{\alpha}\eta)| \, |\F{g_{2}}(\eta_{\com \cup\intv{n}\setminus\alpha})|}{\lambda + \norm{\eta_{\rel}}^2 + \quart\norm{\eta_{\com}}^2 + \half\norm{\eta_{\intv{n}\setminus\alpha}}^2} \lesssim \frac{|\F{f}(\Tdualone{\alpha}\eta)| \, |\F{g_{2}}(\eta_{\com \cup\intv{n}\setminus\alpha})|}{\lambda + \norm{\eta_{\rel}}^2} \, \, .
	\end{equation*}
	Applying the Cauchy--Schwarz inequality in the variables $(\eta_{\rel},\eta_{\com},\eta_{\intv{n}\setminus\alpha})$, we find that
	\begin{align}{\label{eqn:outgoing_cs_bound_one}}
	\int_{\epsalhattorus{2\intv{n}}}
	& \frac{|\F{f}(\Tdualone{\alpha}\eta)| \, |\F{g_2}(\eta_{\com \cup\intv{n}\setminus\alpha})|}{\lambda + \norm{\eta_{\rel}}^2} 
	\d \eta  \nonumber \\
	& \lesssim \norm{f}_{2} \Bigg(\int_{\widehat{\R}^2}\frac{\d \eta_{\rel}}{(\lambda + \norm{\eta_{\rel}}^2)^2} \int_{\widehat{\R}^{2\intv{n}\setminus\rel}_{\alpha}} |\F{g_{2}}(\eta_{\com \cup\intv{n}\setminus\alpha})|^2\d \eta_{\com \cup\intv{n}\setminus\alpha}\Bigg)^{\half} .
	\end{align} 
	The $(\eta_{\com},  \eta_{\intv{n}\setminus\alpha})$ integral is just $\norm{g_{2}}_{2}$. As for the $\eta_{\rel}$ integral, it can be evaluated by a change to polar coordinates, exactly as in equation~\eqref{eqn:outgoing_polar}:
	\begin{equation}{\label{eqn:outgoing_cs_bound_two}}
	\int_{\widehat{\R}^{2}}\frac{\d \eta_{\rel}}{(\lambda + \norm{\eta_{\rel}}^2)^2} = \pi \int_{0}^{\infty} \frac{2r \, \d r}{(\lambda + r^2)^2} = \pi \int_{\lambda}^{\infty} \frac{\d r}{r^2} = \frac{\pi}{\lambda} \, .
	\end{equation}
	Using the Cauchy--Schwarz inequality once more, in the $y_{\rel}$ variable, we also get
	\begin{equation}{\label{eqn:outgoing_cs_bound_three}}
	\Bigg|\int_{\R^2} \sqrrange(y'_{\rel})g_{1}(y'_{\rel}) \, \d y'_{\rel}\Bigg| \lesssim \sqrt{\norm{\range}_{1}} \, \norm{g_{1}}_{2} \, .
	\end{equation}
	Combining these estimates yields the bound
	\begin{equation*}
		\Bigg|\int_{\R^{2\intv{n}}}  \int_{\R^{2\intv{n}}}  \, f(x) \conj{\epscontgreens(\lambda, x-x') G_{\varepsilon\alpha}(x')} \, \d x \, \d x'\Bigg| \lesssim \lambda^{-\half}\norm{f}_{2}\norm{g}_{2} \, \, .
	\end{equation*}
	  Next, we turn to the remaining term $\cali{I}(\varepsilon, f, g)$ in equation~\eqref{eqn:preoutgoing_position_one}. As before, we start with a change of variables $\Seps{\alpha}y' = x'$. Then, express the difference between the lattice and the continuous Green's function in Fourier space using the inversion formula. It then becomes:
	\begin{align*}
		\cali{I}(\varepsilon, f, g) 
		& = \frac{1}{(2\pi)^{2n}} \int_{\R^{2\intv{n}}} \int_{\R^{2\intv{n}}} \int_{\epshattorus{2\intv{n}}} f(x) \frac{e^{\im \xi \cdot (\epsfloor{x} - \epsfloor{x'})} - e^{\im \xi \cdot (x-x')}}{\lambda + \fourierepsgen(\xi)}G_{\varepsilon\alpha}(x') \, \d x \, \d x' \, \d \xi  \, \\
		& = \frac{1}{(2\pi)^{2n}}\int_{\R^{2\intv{n}}} \int_{\R^{2\intv{n}}} \int_{\epshattorus{2\intv{n}}} f(x)e^{\im \xi \cdot x} \frac{e^{-\im \xi \cdot (\epsfracpart{x} - \epsfracpart{x'})} - 1}{\lambda + \fourierepsgen(\xi)}G_{\varepsilon\alpha}(x') e^{-\im \xi \cdot x'} \, \d x \, \d x' \, \d \xi \, \, ,
	\end{align*}
	where $\epsfracpart{x} \coloneqq x - \epsfloor{x}$ is the fractional part in the lattice $\Z_{\varepsilon}^{2\intv{n}}$. Using the mean value theorem, we can write the complex exponential in the numerator as:
	\begin{equation*}
		 e^{-\im \xi \cdot (\epsfracpart{x} - \epsfracpart{x'})}-1 = -\im \xi \cdot (\epsfracpart{x} - \epsfracpart{x'}) \int_{0}^{1} e^{-\im z \xi \cdot (\epsfracpart{x} - \epsfracpart{x'})} \, \d z \, .
	\end{equation*} 
	Here, the fractional part $\epsfracpart{\cdot}$ applies to each of its $\R^2$ components separately. That is, $\epsfracpart{x} = (\epsfracpart{x_1}, \epsfracpart{x_2} , \ldots , \epsfracpart{x_n})$ so that $\xi \cdot (\epsfracpart{x} - \epsfracpart{x'}) = \sum_{k=1}^{n}\xi_k \cdot (\epsfracpart{x_k} - \epsfracpart{x'_k})$. For $l, m \in \intv{n}$ and $\xi \in \epshattorus{2\intv{n}}$, define the functions $\widetilde{F}_{l}$ and $\widetilde{G}_{\varepsilon\alpha, m}$ as
	\begin{equation}{\label{eqn:preoutgoing_fourier_frac_one}}
	\begin{split}
		\widetilde{F}_{l}(\xi) 
		&  \coloneqq \frac{1}{(2\pi)^{n}}\int_{\R^{2\intv{n}}} \int_{0}^{1}f(x) \big(\xi_l \cdot \epsfracpart{x_l}\big) e^{- \im z \xi \cdot \epsfracpart{x}+\im \xi \cdot x}\d x \, \d z \, , \\ \widetilde{G}_{\varepsilon\alpha, m}(\xi) 
		& \coloneqq \frac{1}{(2\pi)^{n}}\int_{\R^{2\intv{n}}}\int_{0}^{1} G_{\varepsilon \alpha}(x') \big(\xi_m \cdot \epsfracpart{x'_m}\big) e^{-\im z \xi \cdot \epsfracpart{x'} + \im \xi \cdot x'}\d x' \, \d z \, .
    \end{split}
	\end{equation} 
	Let $\widetilde{F}_{0}(\xi)$ and $\widetilde{G}_{\varepsilon\alpha,0}(\xi)$ denote the corresponding expressions with the factors $\xi_l \cdot \epsfracpart{x_l}$ and $\xi_m \cdot \epsfracpart{x'_m}$ omitted, respectively. With this notation, we can rewrite the expression for the error term as:
	\begin{equation}{\label{eqn:preoutgoingerror_fourier}}
		\cali{I}(\varepsilon, f, g) = -\im\sum_{l=1}^{n} \int_{\epshattorus{2\intv{n}}}\frac{\widetilde{F}_{l}(\xi)\conj{\widetilde{G}_{\varepsilon \alpha, 0}(\xi)}}{\lambda + \fourierepsgen(\xi)} \d \xi + \im \sum_{m=1}^{n} \int_{\epshattorus{2\intv{n}}}\frac{\widetilde{F}_{0}(\xi)\conj{\widetilde{G}_{\varepsilon \alpha,m}(\xi)}}{\lambda + \fourierepsgen(\xi)} \d \xi \, \, .
	\end{equation} 
    To bound this error term by the $\Ltwo$ norms of $f$ and $g$, it suffices to do the same for each of the terms in the above sum. Previously, we did this by explicitly computing the Fourier transform of $G_{\varepsilon\alpha}(x')$ in terms of $\F{g_{2}}$ and used the Cauchy--Schwarz inequality. Here, however, due to the presence of the fractional part terms, computing the Fourier transforms $\widetilde{G}_{\varepsilon\alpha,m}$ and $\widetilde{F}_{l}$ is more difficult than in the previous argument. Instead, we follow a more direct approach and decompose the integrals in their definitions in equation~\eqref{eqn:preoutgoing_fourier_frac_one} into sums over the lattice boxes. On each such box, every point can be written as a lattice point plus a remainder lying in $[0,\varepsilon)^{2\intv{n}}$, so the fractional part becomes an ordinary continuous variable ranging over this box. This allows us to treat the fractional part explicitly and estimate the Fourier transform in a cleaner way. More precisely, we write every $x \in \R^{2\intv{n}}$ as $p + u$ where $p \in \Z^{2\intv{n}}_{\varepsilon}$ and $u \in [0,\varepsilon)^{2\intv{n}}$. We begin with the term $\widetilde{G}_{\varepsilon\alpha,m}(\xi)$. In what follows, the symbol $\epsrect{2\intv{n}}$ denotes the set $[0,\varepsilon)^{2\intv{n}}$, and $\epsalrect{2\intv{n}}$ denotes $\Teps{\alpha}\big(\epsrect{2\intv{n}}\big)$. From the definition,
\begin{align*}
	\widetilde{G}_{\varepsilon\alpha, m}(\xi) 
	& = \frac{\varepsilon^{-2}}{(2\pi)^{n}} \int_{\R^{2\intv{n}}}\int_{0}^{1}(\sqrrange g)(\Teps{\alpha}x) (\xi_m \cdot \epsfracpart{x_m})\, e^{-\im z \xi \cdot \epsfracpart{x}} e^{\im \xi \cdot x} \,\d x \, \d z\\
	& = \frac{\varepsilon^{-2}}{(2\pi)^{n}} \sum_{p \in \Z^{2\intv{n}}_{\varepsilon}} \int_{\epsrect{2\intv{n}}} \int_{0}^{1} (\sqrrange g )(\Teps{\alpha}p + \Teps{\alpha}u) \big(\xi_m \cdot u_m \big) \, e^{-\im z \xi \cdot u} e^{\im \xi \cdot p}e^{\im \xi \cdot u } \, \d u \, \d z \, .
	\end{align*}
	Now, we make a change of variables $u \mapsto \Seps{\alpha}v$, $p \mapsto \Seps{\alpha}q$ and $\xi \mapsto \Tdualone{\alpha}\eta$. This makes the summation over $\Z^{2\intv{n}}_{\varepsilon\alpha}$ and changes the domains $\epsrect{2\intv{n}}$ and $\epshattorus{2\intv{n}}$ to $\epsalrect{2\intv{n}}$ and $\epsalhattorus{2\intv{n}}$ respectively, while also canceling the factor of $\varepsilon^{-2}$ outside the summation. Recall also the notation $(q_{\com\cup\intv{n}\setminus\alpha}+ v_{\com\cup\intv{n}\setminus\alpha} ) \coloneqq (q_{\com} + v_{\com}, q_{\intv{n}\setminus\alpha}+ v_{\intv{n}\setminus\alpha})$. The preceding expression for $\widetilde{G}_{\varepsilon\alpha, m}$ then becomes
	\begin{align}{\label{eqn:preoutgoing_fourier_frac_two}}
		\widetilde{G}_{\varepsilon\alpha, m}(\Tdualone{\alpha}\eta) = \frac{1}{(2\pi)^{n}}\sum_{q \in \Z^{2\intv{n}}_{\varepsilon\alpha}} \int_{\epsalrect{2\intv{n}}} \int_{0}^{1} 
		& (\sqrrange g_{1})(q_{\rel} + v_{\rel})g_{2}(q_{\com\cup\intv{n}\setminus\alpha}+ v_{\com\cup\intv{n}\setminus\alpha} ) \nonumber \\
		& \times  (\Tdualone{\alpha}\eta)_m \cdot(\Seps{\alpha}v)_m \times e^{\im (1 - z) \Tdualone{\alpha}\eta \cdot \Seps{\alpha}v} e^{\im \Tdualone{\alpha}\eta \cdot \Seps{\alpha}q} \, \d v \, \d z \, .
	\end{align} 
	From the definitions of the operators $\Tdualone{\alpha}$ and $\Seps{\alpha}$, note that 
	\begin{equation*}
		\Tdualone{\alpha}\eta \cdot \Seps{\alpha}q = \varepsilon\eta_{\rel} \cdot q_{\rel} + \eta_{\com} \cdot q_{\com} + \eta_{\intv{n} \setminus \alpha} \cdot q_{\intv{n}\setminus\alpha} \, \, .
	\end{equation*}
	and similarly
	\begin{equation}{\label{eqn:preoutgoing_fourier_frac_three}}
		\Tdualone{\alpha} \eta  \cdot \Seps{\alpha} v = \varepsilon \eta_{\rel} \cdot v_{\rel} + \eta_{\com} \cdot v_{\com} + \eta_{\intv{n}\setminus\alpha} \cdot v_{\intv{n}\setminus\alpha} \, \, .
	\end{equation}
    Moreover, since $\norm{u_k} \lesssim \varepsilon$ for all $k \in \intv{n}$, the vector $v = \Teps{\alpha}u$ satisfies $\norm{v_{\rel}} \lesssim 1$ and $\norm{v_{\com}}, \norm{v_{\intv{n}\setminus\alpha}} \lesssim \varepsilon$. As $\eta \in \epsalhattorus{2\intv{n}}$, its components satisfy the bounds $\norm{\eta_{\rel}}$, $\norm{\eta_{\com}}, \norm{\eta_{\intv{n} \setminus \alpha}} \lesssim \varepsilon^{-1}.$ Thus, from this fact, we observe that $\norm{\Tdualone{\alpha}\eta} \cdot \norm{\Seps{\alpha}v} \lesssim 1$. 
	For further ease of notation, we set:
	\begin{equation}{\label{eqn:preoutgoing_fourier_frac_four}}
		\widetilde{g}_{2}(v_{\com \cup \intv{n}\setminus\alpha}, \eta_{\com \cup \intv{n}\setminus\alpha}) \coloneqq \frac{
		1}{(2\pi)^{n-1}}\sum_{q_{\com \cup \intv{n}\setminus\alpha}}g_{2}(q_{\com \cup\intv{n}\setminus\alpha} + v_{\com\cup\intv{n}\setminus\alpha}) e^{\im \eta_{\com \cup \intv{n} \setminus \alpha} \cdot q_{\com \cup \intv{n}\setminus\alpha}} \, \, .
	\end{equation}
	Taking absolute values in equation~\eqref{eqn:preoutgoing_fourier_frac_two} and applying the Cauchy--Schwarz inequality with respect to the $v_{\com \cup \intv{n}\setminus\alpha}$ variables between $\Tdualone{\alpha}\eta \cdot \Seps{\alpha}v$ and $\widetilde{g}_{2}(v_{\com \cup \intv{n}\setminus\alpha}, \eta_{\com \cup \intv{n}\setminus\alpha})$, we obtain the bound:
	\begin{align*}
	|\widetilde{G}_{\varepsilon\alpha, m} (\Tdualone{\alpha}\eta)|^{2}
	&  \lesssim  \int \bigg|\sum_{q_{\rel}} \int  (\sqrrange g_{1})(v_{\rel} + q_{\rel}) e^{\im \varepsilon\eta_{\rel} \cdot q_{\rel}}  \d v_{\rel}\bigg|^2  \norm{(\Tdualone{\alpha}\eta )_m\cdot (\Seps{\alpha}v)_m}^2 \, \d v_{\com \cup \intv{n}\setminus\alpha} \\
	& \qquad \qquad \qquad \qquad \qquad \qquad \times \int |\widetilde{g}_{2}(v_{\com \cup \intv{n}\setminus\alpha}, \eta_{\com \cup \intv{n}\setminus\alpha})|^2 \d v_{\com \cup \intv{n}\setminus\alpha} \\
	&  \lesssim  \bigg|\sum_{q_{\rel}} \int   (\sqrrange g_{1})(v_{\rel} + q_{\rel})e^{\im \varepsilon\eta_{\rel} \cdot q_{\rel}}  \d v_{\rel}\bigg|^2\varepsilon^{2n-2} \int |\widetilde{g}_{2}(v_{\com \cup \intv{n}\setminus\alpha}, \eta_{\com \cup \intv{n}\setminus\alpha})|^2 \d v_{\com \cup \intv{n}\setminus\alpha} \, .  
	\end{align*}
	Note the $\d z $ integral has been bounded trivially. Using the Cauchy--Schwarz inequality in the $\d v_{\rel}$ variables, the term in the absolute values can be bounded as
	\begin{align*}
		\bigg|\sum_{q_{\rel}} \int   (\sqrrange g_{1})(v_{\rel} + q_{\rel})e^{\im \varepsilon\eta_{\rel} \cdot q_{\rel}}  \d v_{\rel}\bigg|^2 
		& \lesssim \Big(\sum_{q_{\rel}} \int \range(q_{\rel} + v_{\rel}) \, \d v_{\rel} \Big) \cdot \Big(\sum_{q_{\rel}} \int |g_{1}(q_{\rel} + v_{\rel})|^2 \, \d v_{\rel}\Big) \\ 
		& = \norm{\range}_{1} \, \norm{g_{1}}_{2}^{2} \, \, .
	\end{align*}
	Note that this bound is uniform in the variable $\eta_{\rel}$. 
	Thus,
	\begin{equation*}
       \int |\widetilde{G}_{\varepsilon\alpha, m} (\Tdualone{\alpha}\eta)|^{2} \d \eta_{\com\cup\intv{n}\setminus\alpha} \lesssim \norm{\range}_{1} \norm{g_{1}}_{2}^2 \iint \varepsilon^{2n-2}|\widetilde{g}_{2}(v_{\com \cup \intv{n}\setminus\alpha}, \eta_{\com \cup \intv{n}\setminus\alpha})|^2  
       \,\d v_{\com \cup \intv{n}\setminus\alpha}\,\d \eta_{\com \cup \intv{n}\setminus\alpha} \, .
	\end{equation*}
	For the term on the right-hand side, we can use Plancherel's theorem for Fourier series to see that it is bounded from above by
	\begin{equation*}
		\int\sum_{q_{\com \cup \intv{n}\setminus\alpha}} |g_{2}(q_{\com \cup \intv{n}\setminus\alpha} + v_{\com \cup \intv{n}\setminus\alpha})|^2 \, \d v_{\com \cup \intv{n}\setminus\alpha} = \norm{g_{2}}^2 \, \, .
	\end{equation*}
	Using Lemma~\ref{lem:uniformgreensbound}, we also have the bound
    \begin{equation*}
    \frac{1}{\lambda + \fourierepsgen(\Tdualone{\alpha}\eta)} \lesssim \frac{1}{\lambda + \norm{\eta_{\rel}}^2} \, \, .
    \end{equation*} 
	Therefore, using the preceding estimate for all $m = 0, 1, 2, \ldots n$, we get the bound
	\begin{align*}
		\int_{\epsalhattorus{2\intv{n}}} \frac{|\widetilde{G}_{\varepsilon\alpha, m}(\Tdualone{\alpha}\eta)|^2}{(\lambda + \fourierepsgen(\Tdualone{\alpha}\eta))^2} \, \d \eta 
		& \leqslant \int\frac{\mathbf{1}_{\epsalhattorus{2\intv{n}}}(\eta)}{(\lambda + \norm{\eta_{\rel}}^2)^2} \, \d \eta_{\rel} \, \int |\widetilde{G}_{\varepsilon\alpha, m}(\Tdualone{\alpha}\eta)|^2 \, \d \eta_{\com \cup \intv{n} \setminus \alpha} \\
		& \lesssim \frac{\norm{\range}_{1} \norm{g}_{2}^{2}}{\lambda} \, \, .
	\end{align*}
	A similar argument can be used to show that 
	\begin{equation*}
		\int_{\epsalhattorus{2\intv{n}}} |\widetilde{F}_{l}(\Tdualone{\alpha}\eta)|^2 \d \eta \lesssim \norm{f}_{2}^{2} \, \, ,
	\end{equation*} 
	for $l = 0, 1,2, \ldots n$. Thus,
	\begin{align*}
		\Bigg| \int_{\epshattorus{2\intv{n}}}\frac{\widetilde{F}_{l}(\xi)\conj{\widetilde{G}_{\varepsilon \alpha, m}(\xi)}}{\lambda + \fourierepsgen(\xi)} \d \xi \Bigg| 
		& \leqslant \Bigg(\int_{\epsalhattorus{2\intv{n}}} |\widetilde{F}_{l}(\Tdualone{\alpha}\eta)|^2 \d \eta\Bigg)^{\half} \cdot  \Bigg(\int_{\epsalhattorus{2\intv{n}}} \frac{|\widetilde{G}_{\varepsilon\alpha, m}(\Tdualone{\alpha}\eta)|^2}{(\lambda + \fourierepsgen(\Tdualone{\alpha}\eta))^2} \, \d \eta \Bigg)^{\half} \\
		& \lesssim \lambda^{-\half}\norm{f}_{2} \, \norm{g}_{2} \, \, .
	\end{align*}
	Combining the bounds for all choices of $l'$s and $m'$s, we get the error bound
	\begin{equation*}
	| \cali{I}(\varepsilon, f, g)| \lesssim \lambda^{-\half}\norm{f}_{2} \norm{g}_{2} \, \, .
	\end{equation*}
	Therefore, we have just shown that
    \begin{equation*}
    \normop{\preoutg{\alpha}(\lambda)} = \sup_{\norm{f}_{2}, \norm{g}_{2}\leqslant 1} |\langle f, \preoutg{\alpha}(\lambda)g \rangle| \lesssim \sup_{\norm{f}_{2}, \norm{g}_{2}\leqslant 1} ( \lambda^{-\half}\norm{f}_{2}\norm{g}_{2} + \lambda^{-\half}\norm{f}_{2}\norm{g}_{2}) \lesssim \lambda^{-\half} \, \, .
    \end{equation*}
    This proves the required norm bound for the operator $\preoutg{\alpha}(\lambda)$. We now turn to the norm convergence to the limiting operator $\outg{\alpha}(\lambda)$. As mentioned at the start of this subsection, we will compare it to the prelimiting one and show that the operator norm of their difference tends to $0$ uniformly with respect to $f$ and $g$. It will be convenient in our analysis to treat this difference separately in the low and high frequency regions. To this end, fix $R > 0$ such that $R \ll \varepsilon^{-1}$ and define the error terms:
    \begin{align*}
    	\cali{I}_{1}(R, \varepsilon) &\coloneqq   \frac{\ind{\norm{\xi} \leqslant R}(e^{\im \xi \cdot (\epsfloor{x} - \epsfloor{x'})} - e^{\im \xi \cdot (x-x')})}{\lambda + \fourierepsgen(\xi)} \\ 
    	\cali{I}_{2}(R,\varepsilon) &\coloneqq \frac{ \ind{R < \norm{\xi} \lesssim \pi/\varepsilon} (e^{\im \xi \cdot (\epsfloor{x} - \epsfloor{x'})} - e^{\im \xi \cdot (x-x')})}{\lambda + \fourierepsgen(\xi)} \\    
    	\cali{I}_{3}(R, \varepsilon) &\coloneqq  \ind{\norm{\xi} \leqslant R} \frac{e^{\im \frac{\varepsilon}{2}(\xi_i -\xi_j)\cdot y'_{\rel}} - 1}{\lambda + \fourierepsgen(\xi)}  & \cali{I}_{4}(R,\varepsilon) &\coloneqq \ind{R < \norm{\xi} \lesssim \pi/\varepsilon}   \frac{e^{\im \frac{\varepsilon}{2}(\xi_i -\xi_j)\cdot y'_{\rel}} - 1}{\lambda + \fourierepsgen(\xi)}      \\ 
    	\cali{I}_{5}(R, \varepsilon) & \coloneqq  \frac{\ind{\norm{\xi} \leqslant R}}{\lambda + \fourierepsgen(\xi)} - \frac{\ind{\norm{\xi} \leqslant R}}{\lambda + \half\norm{\xi}^2} & 	\cali{I}_{6}(R, \varepsilon) & \coloneqq \frac{\ind{R < \norm{\xi} \lesssim \pi/\varepsilon}}{\lambda + \fourierepsgen(\xi)} - \frac{\ind{R < \norm{\xi} \lesssim \pi/\varepsilon}}{\lambda + \half\norm{\xi}^2} \\
    	\cali{I}_{7}(R, \varepsilon) & \coloneqq  \frac{-\ind{\norm{\xi} \gtrsim \pi/\varepsilon}}{\lambda + \half\norm{\xi}^2} \, \, .
    \end{align*}
    Then, from equations~\eqref{eqn:outgoing_fourier_one} and \eqref{eqn:preoutgoing_position_one}, we get the following expression for the difference between the prelimiting and the limiting operator.
   \begin{align}{\label{eqn:outgoingerror_zero}}    
    	\langle f, [\preoutg{\alpha}(\lambda) - \outg{\alpha}(\lambda)]g \rangle 
    	& = \frac{1}{(2\pi)^{2n}}\sum_{r=1}^{2}\int_{\R^{2\intv{n}}} \int_{\R^{2\intv{n}}} f(x) \cali{I}_{r}(R, \varepsilon) G_{\varepsilon\alpha}(x') \, \d x \, \d x' \nonumber  \\
    	&  + \frac{1}{2\pi}\sum_{r=3}^{7}\int_{\R^2}\sqrrange(y'_{\rel})g_{1}(y'_{\rel}) \d y'_{\rel}  \int_{\epshattorus{2\intv{n}}} \F{f}(\xi)  \cali{I}_{r}(R, \varepsilon) \conj{ \F{g_{2}}(\xi_i + \xi_j, \xi_{\intv{n}\setminus\alpha})}  \d \xi \nonumber \\
    	& \coloneqq \sum_{r=1}^{7} \cali{I}^{f,g}_{r}(R, \varepsilon) \, \, .
    \end{align}
    We now proceed to show that each of the terms in the sum above vanish in the limit. More precisely, we will bound each term by a constant multiple of the $\Ltwo$ norms of $f$ and $g$, where the constant tends to zero as $\varepsilon \to 0$ first, followed by letting $R \to \infty$. We start with $\cali{I}^{f,g}_{1}(R,\varepsilon)$. Here the argument is similar to that used in bounding the term $\cali{I}(\varepsilon, f, g)$ defined in equation~\eqref{eqn:preoutgoingerror_position}. First, use the mean value theorem to write $e^{-\im \xi \cdot (\epsfracpart{x} - \epsfracpart{x'})}-1 = -\im \xi \cdot (\epsfracpart{x} - \epsfracpart{x'}) \int_{0}^{1}e^{-\im z \xi \cdot (\epsfracpart{x} - \epsfracpart{x'})} \, \d z$. Recalling the definitions of $\widetilde{F}_{l}$ and $\widetilde{G}_{\varepsilon\alpha,m}$, we may then rewrite $\cali{I}^{f,g}_{1}(R,\varepsilon)$ as
    \begin{equation}{\label{eqn:outgoingerror_one_a}}
    	-\im\sum_{l=1}^{n} \int_{\{\norm{\xi} \leqslant R\}}\frac{\widetilde{F}_{l}(\xi)\conj{\widetilde{G}_{\varepsilon \alpha, 0}(\xi)}}{\lambda + \fourierepsgen(\xi)} \d \xi + \im \sum_{m=1}^{n} \int_{\{\norm{\xi} \leqslant R\}}\frac{\widetilde{F}_{0}(\xi)\conj{\widetilde{G}_{\varepsilon \alpha,m}(\xi)}}{\lambda + \fourierepsgen(\xi)} \d \xi \, \, .
    \end{equation} 
     This expression is identical to that in equation~\eqref{eqn:preoutgoingerror_fourier}, except now the domain is restricted to the set $\{\norm{\xi} \leqslant R\}$. Following the same proof, we use the substitution $\xi = \Tdualone{\alpha}\eta$ and Lemma~\ref{lem:uniformgreensbound} to note that
     \begin{equation*}
     	\frac{1}{\lambda + \fourierepsgen(\Tdualone{\alpha}\eta)} \lesssim \frac{1}{\lambda + \norm{\eta_{\rel}}^2} \, \, .
     \end{equation*} 
     Using the Cauchy--Schwarz inequality, the inequality below is also evident:
     \begin{align*}
     	|\widetilde{G}_{\varepsilon\alpha, m} (\Tdualone{\alpha}\eta)|^{2}
     	&  \lesssim  \norm{\range}_{1} \, \norm{g_{1}}_{2}^{2}  \int \norm{(\Tdualone{\alpha}\eta )_m\cdot (\Seps{\alpha}v)_m}^2 \, \d v_{\com \cup \intv{n}\setminus\alpha} \nonumber \\
     	& \qquad \qquad \qquad \qquad \qquad \qquad \times \int |\widetilde{g}_{2}(v_{\com \cup \intv{n}\setminus\alpha}, \eta_{\com \cup \intv{n}\setminus\alpha})|^2 \d v_{\com \cup \intv{n}\setminus\alpha} \, \, ,
     \end{align*}
      where $\widetilde{g}_{2}(v_{\com \cup \intv{n}\setminus\alpha}, \eta_{\com \cup \intv{n}\setminus\alpha})$ is as defined in equation~\eqref{eqn:preoutgoing_fourier_frac_four}. Here, unlike what we previously did, we use the sharper bound
      \begin{equation*}
      	\norm{(\Tdualone{\alpha}\eta )_m\cdot (\Seps{\alpha}v)_m} \lesssim \varepsilon R \, \, .
      \end{equation*}
      This follows from equation~\eqref{eqn:preoutgoing_fourier_frac_three} and the fact that $\norm{\xi} \leqslant R$ implies $\norm{\eta} = \norm{\Sdualone{\alpha}\xi} \lesssim R$. Now, using Plancherel's theorem for Fourier series yields the estimate
      \begin{align*}
      	\int |\widetilde{G}_{\varepsilon\alpha, m} (\Tdualone{\alpha}\eta)|^{2} \, \d \eta_{\com \cup \intv{n} \setminus \alpha}
      	& \lesssim  \varepsilon^2 R^2 \norm{\range}_{1} \, \norm{g_{1}}_{2}^{2} \, \, \iint   \varepsilon^{2n-2} |\widetilde{g}_{2}(v_{\com \cup \intv{n}\setminus\alpha}, \eta_{\com \cup \intv{n}\setminus\alpha})|^2 \d v_{\com \cup \intv{n}\setminus\alpha} \, \, \\
      	& \lesssim \varepsilon^{2} R^2 \norm{\range}_{1} \, \norm{g}_{2}^{2} \, \, .
      \end{align*}
      Therefore,
      \begin{align*}
        \int_{\{\norm{\xi} \leqslant R\}} \frac{|\widetilde{G}_{\varepsilon\alpha, m} (\xi)|^{2}}{(\lambda + \fourierepsgen(\xi))^2} \, \d \xi 
        & \lesssim \int \frac{\ind{\norm{\eta} \lesssim R} \d \eta_{\rel}}{(\lambda + \norm{\eta_{\rel}}^2)^2} \int |\widetilde{G}_{\varepsilon\alpha, m} (\Tdualone{\alpha}\eta)|^{2} \, \d \eta_{\com \cup \intv{n} \setminus \alpha} \\
        & \lesssim \frac{\varepsilon^2 R^2 \norm{g_{1}}_{2}^{2}}{\lambda} \int \varepsilon^{2n-2} \int |\widetilde{g}_{2}(v_{\com \cup \intv{n}\setminus\alpha}, \eta_{\com \cup \intv{n}\setminus\alpha})|^2 \d v_{\com \cup \intv{n}\setminus\alpha}  \d \eta_{\com \cup \intv{n} \setminus \alpha} \\
        & = \frac{\varepsilon^2 R^2}{\lambda} \norm{g}_{2}^{2} \, \, \, .
      \end{align*}
      Similarly, we also have
      \begin{equation*}
      	\int_{\{\norm{\xi} \leq R\}}|\widetilde{F}_{l}(\xi)|^2 \, \d \xi \lesssim \int_{\{\norm{\eta} \lesssim R\}} |\widetilde{F}_{l}(\Tdualone{\alpha}\eta)|^2 \d \eta \lesssim \norm{f}_{2}^{2} \, \, \, .
      \end{equation*} 
      In all, the following bound is obtained:
      \begin{align*}
      \Bigg| \int_{\{\norm{\xi} \leqslant R\}}\frac{\widetilde{F}_{l}(\xi)\conj{\widetilde{G}_{\varepsilon \alpha,m}(\xi)}}{\lambda + \fourierepsgen(\xi)} \d \xi\Bigg| 
      & \lesssim  \Bigg(\int_{\{\norm{\xi} \leq R\}}|\widetilde{F}_{l}(\xi)|^2 \, \d \xi \Bigg)^{\half} \, \Bigg(  \int_{\{\norm{\xi} \leqslant R\}} \frac{|\widetilde{G}_{\varepsilon\alpha, m} (\xi)|^{2}}{(\lambda + \fourierepsgen(\xi))^2} \, \d \xi  \Bigg)^{\half} \, \, \\
      & \lesssim \varepsilon R \norm{f}_{2} \, \norm{g}_{2}  \, .
      \end{align*}
      Since $l$ and $m$ were arbitrarily chosen from $\{0,1, \ldots n\}$, the above bound holds for every term in equation~\eqref{eqn:outgoingerror_one_a}, so that 
     \begin{equation}{\label{eqn:outgoingerror_one}}
    	| \cali{I}^{f,g}_{1}(R,\varepsilon)|  \lesssim \varepsilon R \norm{f}_{2}\norm{g}_{2} \,\, .
    \end{equation}
    We proceed to $\cali{I}^{f,g}_{2}(R,\varepsilon)$. Here too, the proof is similar to that used in bounding $\cali{I}(\varepsilon, f, g)$. Recall we used the bound
    \begin{equation*}
    	 \frac{1}{\lambda + \fourierepsgen(\xi)} \lesssim \frac{1}{\lambda + \half \norm{\xi}^2}
    \end{equation*} 
    from Lemma~\ref{lem:uniformgreensbound}. Before we perform a change of variables $\xi \mapsto \Tdualone{\alpha}\eta$, write $(\lambda + \half \norm{\xi}^2)^{-1}$ as 
    \begin{equation}{\label{eqn:outgoingerror_onepointone}}
    	\frac{1}{(\lambda + \half \norm{\xi}^2)^{\threequart}} \times \frac{1}{(\lambda + \half \norm{\xi}^2)^{\quart}}  \, \, . 
    \end{equation}
    Then, on the set $\{R < \norm{\xi} \lesssim \pi/\varepsilon\}$, we have the bound
    \begin{equation*}
    	\frac{1}{(\lambda + \half\norm{\xi}^2)^2} \lesssim \frac{1}{(\lambda + \half\norm{\xi}^2)^{\threehalf}(\lambda + \half R^2)^{\half}} \, .
    \end{equation*}
    The power $\threequart$ is arbitrary. In fact, any power $\half < p < 1$ will work. It is chosen such that the right-hand side above is integrable with respect to $\d \eta_{\rel}$ after a change of variables and goes to $0$ as $R \to \infty$. Now moving to the $\eta$-frame and using the Cauchy--Schwarz inequality, we obtain:
    \begin{align*}
    \int_{\{R < \norm{\xi} \lesssim \pi/\varepsilon\}} \frac{|\widetilde{G}_{\varepsilon\alpha, m} (\xi)|^{2}}{(\lambda + \half\norm{\xi}^2)^2} \, \d \xi 
    & \lesssim \frac{1}{(1+\half R^2)^{\half}}\int \frac{\ind{R \lesssim \norm{\eta} \lesssim \pi/\varepsilon}}{(\lambda + \norm{\eta_{\rel}}^2)^{\threehalf}} \,  \d \eta_{\rel} \,\int |\widetilde{G}_{\varepsilon\alpha, m} (\Tdualone{\alpha}\eta)|^{2} \, \d \eta_{\com \cup \intv{n} \setminus \alpha} \\
    & \lesssim \frac{\lambda^{-\half}\norm{g_{1}}_{2}^{2}}{R} \iint\varepsilon^{2n-2} |\widetilde{g}_{2}(v_{\com \cup \intv{n}\setminus\alpha}, \eta_{\com \cup \intv{n}\setminus\alpha})|^2 \d v_{\com \cup \intv{n}\setminus\alpha}  \d \eta_{\com \cup \intv{n} \setminus \alpha} \\
    & = \frac{1}{R\sqrt{\lambda}} \norm{g}_{2}^{2} \, \, \, .
    \end{align*}
    Of course, we also have
    \begin{equation*}
        	\int_{\{R <\norm{\xi} \leqslant \pi/\varepsilon\}}|\widetilde{F}_{0}(\xi)|^2 \, \d \xi \lesssim  \norm{f}_{2}^{2} \, \, \, .
    \end{equation*}
    Putting the above two estimates together, we get the error bound:
    \begin{equation}{\label{eqn:outgoingerror_two}}
       |\cali{I}^{f,g}_{2}(R, \varepsilon)| \lesssim \frac{\norm{f}_{2}\norm{g}_{2}}{\sqrt{R}} \, .
    \end{equation}
    We next consider $\cali{I}^{f,g}_{3}(R,\varepsilon)$. By Lemma~\ref{lem:uniformgreensbound}, the following inequality holds on the set $\{ \norm{\xi} \leqslant R\}$:
    \begin{equation*}
    \Bigg|\frac{e^{\im \frac{\varepsilon}{2}(\xi_i - \xi_j)\cdot y'_{\rel}} - 1}{\lambda + \fourierepsgen(\xi)}\Bigg| \lesssim \frac{\varepsilon\norm{\xi_i - \xi_j}\norm{y'_{\rel}} \wedge 2}{\lambda + \half\norm{\xi}^2} \, \, .
    \end{equation*}
    Then, taking absolute values in the definition of $\cali{I}^{f,g}_{3}(R,\varepsilon)$ and using the above bound, we obtain:
    \begin{align*}
    	|\cali{I}^{f,g}_{3}
    	& (R, \varepsilon)| \\
    	& \lesssim \int_{\R^2}\sqrrange(y'_{\rel}) |g_{1}(y'_{\rel})| \d y'_{\rel} \int_{\epshattorus{2\intv{n}}}\ind{\norm{\xi} \leq R} \frac{|\F{f}(\xi)| \big(\varepsilon\norm{\xi_i - \xi_j} \,\norm{y'_{\rel}}\wedge 2 \big)|\F{g_{2}}(\xi_i + \xi_j, \xi_{\intv{n}\setminus\alpha})}{\lambda + \half\norm{\xi}^2} \, \d \xi \, .
    \end{align*}
    To bound this integral, we first make the change of variables $\xi \mapsto \Tdualone{\alpha}\eta$. This transforms the domain from $\epshattorus{2\intv{n}}$ to $\epsalhattorus{2\intv{n}}$. Moreover, $\xi_i + \xi_j \mapsto \eta_{\com}$, $\xi_i - \xi_j \mapsto 2\eta_{\rel}$ and $\norm{\xi}^2 \mapsto \norm{\eta_{\rel}}^2 + \quart\norm{\eta_{\com}}^2 + \half\norm{\eta_{\intv{n}\setminus\alpha}}^2 \geqslant \norm{\eta_{\rel}}^2$. Furthermore, recall that, under this transformation, $\norm{\xi} \leqslant R$ implies that $\norm{\eta} \lesssim R$. The right-hand side of the preceding equation becomes
    \begin{equation}{\label{eqn:outgoingerror_three_one}}
    	\int_{\R^2}\sqrrange(y'_{\rel}) |g_{1}(y'_{\rel})| \d y'_{\rel} \int_{ \epsalhattorus{2\intv{n}}}\ind{\norm{\eta} \lesssim R} \frac{|\F{f}(\Tdualone{\alpha}\eta)| \big(\varepsilon\norm{\eta_{\rel}} \,\norm{y'_{\rel}}\wedge 2 \big)|\F{g_{2}}(\eta_{\com\cup\intv{n}\setminus\alpha})}{\lambda + \norm{\eta_{\rel}}^2 + \quart \norm{\eta_{\com}}^2 + \half\norm{\eta_{\intv{n}\setminus\alpha}}^2} \, \d \eta \, .
    \end{equation}
    Applying the Cauchy--Schwarz inequality and noting that $\{\norm{\eta} \lesssim R\} \subseteq \{\norm{\eta_{\rel}} \lesssim R\}$, the $\d \eta$ integral above is at most
    \begin{equation*}
      \norm{f}_{2}\norm{g_{2}}_{2} \, \Bigg(\int_{\{\norm{\eta_{\rel}}  \lesssim R\}}\frac{\varepsilon^2\norm{\eta_{\rel}}^2\norm{y'_{\rel}}^2  \wedge 4}{(\lambda + \norm{\eta_{\rel}}^2)^2} \, \d \eta_{\rel} \Bigg)^{\half} \, \, .
    \end{equation*}
    Substituting this bound into equation~\eqref{eqn:outgoingerror_three_one}, the inner $\d \eta_{\rel}$ integral, which depends on $y'_{\rel}$, can be paired with the factor $\sqrrange(y'_{\rel})$ appearing there, after which the Cauchy--Schwarz inequality can be applied in the $y'_{\rel}$ variable. Consequently, we obtain the following bound for the expression in equation~\eqref{eqn:outgoingerror_three_one}
    \begin{equation*}
        \sqrt{\range^{\lambda}(R,\varepsilon)}\, \norm{g_{1}}_{2} \, \norm{g_{2}}_{2} \, \norm{f}_{2} = \sqrt{\range^{\lambda}(R,\varepsilon)} \,\norm{f}_{2} \, \norm{g}_{2} \,  \, \, ,
    \end{equation*} 
    where 
    \begin{equation*}
        \range^{\lambda}(R,\varepsilon) \coloneqq \int_{\R^2} \range(y'_{\rel}) \Bigg(\int_{\{\norm{\eta_{\rel}}  \lesssim R\}}\frac{\varepsilon^2\norm{\eta_{\rel}}^2\norm{y'_{\rel}}^2  \wedge 4}{(\lambda + \norm{\eta_{\rel}}^2)^2} \, \d \eta_{\rel} \Bigg) \, \d y'_{\rel} \, \, .
    \end{equation*}
    This quantity is finite. Indeed, since $\range \in \cali{L}^{1}(\R^2)$ and the numerator of the $\d \eta_{\rel}$ integrand is at most $4$, we may bound the integral by
    \begin{equation*}
    	\norm{\range}_{1} \, \int_{\{\norm{\eta_{\rel}} \lesssim R\}} \frac{4 \, \d \eta_{\rel}}{(\lambda + \norm{\eta_{\rel}}^2)^2} \lesssim  4\pi \, \norm{\range}_{1}  \int_{0}^{\infty} \frac{2r \, \d r}{(\lambda + r^2)^2}  = \frac{4\pi}{\lambda} \norm{\range}_{1} \, \, \, .
    \end{equation*} 
    Furthermore, since this bound is independent of $y'_{\rel}$ we can use the dominated convergence theorem to pass $\lim_{\varepsilon \to 0}$ inside the integral, whence
    \begin{equation*}
    	\lim_{\varepsilon \to 0} \range^{\lambda}(R, \varepsilon) = \int_{\R^2} \range(y'_{\rel}) \bigg(\int_{\{\norm{\eta_{\rel}} \lesssim R\}}\lim_{\varepsilon \to 0} \frac{\varepsilon^2 \norm{\eta_{\rel}}^2 \norm{y'_{\rel}}^2 \wedge 4}{(\lambda + \norm{\eta_{\rel}}^2)^2} \, \d \eta_{\rel} \bigg) = 0 \, .
    \end{equation*}
    Thus, combining the estimates, we have just shown that 
    \begin{equation}{\label{eqn:outgoingerror_three}}
    	|\cali{I}^{f,g}_{3}(R, \varepsilon)| \lesssim \sqrt{\range^{\lambda}(R,\varepsilon)} \norm{f}_{2}\norm{g}_{2} \, \, .
    \end{equation}
    Now we move towards estimating $\cali{I}^{f,g}_{4}(R,\varepsilon)$. Using Lemma~\ref{lem:uniformgreensbound} along with equation~\eqref{eqn:outgoingerror_onepointone} and the crude bound $| e^{\im \frac{\varepsilon}{2} (\xi_i - \xi_j) \cdot y'_{\rel}}  - 1| \leqslant 2$, we obtain the inequality
    \begin{align*}
    	|\cali{I}_{4}(R, \varepsilon) | \lesssim  \frac{\ind{R < \norm{\xi} \lesssim \pi/\varepsilon}}{\lambda + \half\norm{\xi}^2} 
    	& \lesssim  \frac{\ind{R < \norm{\xi} \lesssim \pi/\varepsilon}}{(\lambda + \half\norm{\xi}^2)^{\threequart}(\lambda + \half\norm{\xi}^2)^{\quart}} \\
    	& \lesssim \frac{\ind{R < \norm{\xi} \lesssim \pi/\varepsilon}}{(\lambda + \half\norm{\xi}^2)^{\threequart}} \cdot \frac{1}{(\lambda + \half R^2)^{\quart}} \, \, .
    \end{align*}
    We insert this into the definition of $\cali{I}^{f,g}_{4}(R,\varepsilon)$ and take absolute values:
     \begin{equation*}
    |\cali{I}^{f,g}_{4}(R, \varepsilon)| \lesssim \int_{\R^2}\sqrrange(y'_{\rel}) |g_{1}(y'_{\rel})| \d y'_{\rel} \int_{\{R < \norm{\xi} \lesssim \pi/\varepsilon\}} \frac{|\F{f}(\xi)| \, |\F{g_{2}}(\xi_i + \xi_j, \xi_{\intv{n}\setminus\alpha})}{\sqrt{R} \, (\lambda + \half\norm{\xi}^2)^{\threequart}} \, \d \xi \, .
    \end{equation*}
    Now, make the change of variables $\xi = \Tdualone{\alpha}\eta$ along with the Cauchy--Schwarz inequality. Subsequently, we get the bound
    \begin{equation}{\label{eqn:outgoingerror_four}}
    	|\cali{I}^{f,g}_{4}(R,\varepsilon)| \lesssim \frac{\norm{f}_{2}\norm{g}_{2}}{\sqrt{R}} \, \, .
    \end{equation}
    For $\cali{I}^{f,g}_{5}(R,\varepsilon)$, recall that on the compact set $\{\norm{\xi} \leqslant R\}$, the symbol $\fourierepsgen(\xi)$ converges to  $\half\norm{\xi}^2$ uniformly. In fact, from Lemma~\ref{lem:uniformgreensbound}, we have:
    \begin{equation*}
    	\ind{\{\norm{\xi} \leqslant R\}} \Bigg|\frac{1}{\lambda + \fourierepsgen(\xi)} - \frac{1}{\lambda + \half\norm{\xi}^2}\Bigg| \leq \frac{\varepsilon^{\delta}R^{2+\delta}}{\lambda \, (\lambda + \half\norm{\xi}^2)} \, \, .
    \end{equation*}
    Substitute this into $\cali{I}^{f,g}_{5}(R,\varepsilon)$. As in the previous cases, we then perform a change of variables $\xi \mapsto \Tdualone{\alpha}\eta$ followed by an application of the Cauchy--Schwarz inequality. Then, we obtain
    \begin{equation}{\label{eqn:outgoingerror_five}}
    	|\cali{I}^{f,g}_{5}(R, \varepsilon)| \lesssim \varepsilon^{\delta}R^{2+\delta} \norm{f}_{2} \norm{g}_{2} \, \, .
    \end{equation}
    In order to bound $\cali{I}^{f,g}_{6}(R,\varepsilon)$, we again use Lemma~\ref{lem:uniformgreensbound} in the form of equation~\eqref{eqn:outgoingerror_onepointone}. By the triangle inequality,
    \begin{equation*}
    \ind{\{R < \norm{\xi} \lesssim \pi/\varepsilon\}}\Bigg|\frac{1}{\lambda + \fourierepsgen(\xi)} - \frac{1}{\lambda + \half\norm{\xi}^2}\Bigg| \lesssim \frac{1}{\lambda + \half\norm{\xi}^2} \lesssim \frac{1}{(\lambda + \half\norm{\xi}^2)^{\threequart} (\lambda + R^2)^{\quart}} \, \, .
    \end{equation*}
    This is the same bound we used in the cases of $\cali{I}^{f,g}_{2}(R,\varepsilon)$ and $\cali{I}^{f,g}_{4}(R,\varepsilon)$ and therefore, we get the same estimate, namely:
    \begin{equation}{\label{eqn:outgoingerror_six}}
    	|\cali{I}^{f,g}_{6}(R, \varepsilon)| \lesssim \frac{\norm{f}_{2}\norm{g}_{2}}{\sqrt{R}} \, \, .
    \end{equation}
    Finally, we turn to $\cali{I}^{f,g}_{7}(R,\varepsilon)$. As before, make the change of variables $\xi \mapsto  \Tdualone{\alpha}\eta$ and write $(\lambda + \half\norm{\xi}^2)$ as $(\lambda + \half\norm{\xi}^2)^{\threequart} \cdot (\lambda + \half\norm{\xi}^2)^{\quart}$. Then, on the set $\{  \norm{\xi} > \pi/\varepsilon\}$, we have:
    \begin{equation*}
    	\Bigg|\frac{1}{\lambda + \half\norm{\xi}^2} \Bigg|\lesssim \frac{1}{(\lambda + \half\norm{\xi}^2)^{\threequart}\cdot(\lambda + \half\norm{\xi}^2)^{\quart}} \lesssim \frac{\sqrt{\varepsilon}}{(\lambda + \half\norm{\xi}^2)^{\threequart}} \, \, .
    \end{equation*}
    Insert this into the expression for $\cali{I}^{f,g}_{7}(R,\varepsilon)$ and use the Cauchy--Schwarz inequality. This yields the estimate:
    \begin{equation}{\label{eqn:outgoingerror_seven}}
    	|\cali{I}^{f,g}_{7}(R,\varepsilon)| \lesssim \sqrt{\varepsilon}\norm{f}_{2}\norm{g}_{2} \, \, .
    \end{equation}
    Combining all the estimates from equations~\eqref{eqn:outgoingerror_zero},\eqref{eqn:outgoingerror_one},\eqref{eqn:outgoingerror_two},\eqref{eqn:outgoingerror_three},\eqref{eqn:outgoingerror_four},\eqref{eqn:outgoingerror_five},\eqref{eqn:outgoingerror_six} and \eqref{eqn:outgoingerror_seven}, and taking the supremum over $f,g$ such that $\norm{f}_{2}, \norm{g}_{2} \leqslant 1$, we obtain the final error estimate:
    \begin{equation*}
    	\normop{\preoutg{\alpha}(\lambda) - \outg{\alpha}(\lambda)} \lesssim \varepsilon R + \sqrt{\range^{\lambda}(R,\varepsilon)} + \varepsilon^{\delta}R^{2+\delta} + \frac{1}{\sqrt{R}} +\sqrt{\varepsilon} \, \, .
    \end{equation*}
    Letting $\varepsilon \to 0$ first and then $R \to \infty$ proves the desired norm convergence result. 
      
\subsection{Off-diagonal operators}

In this subsection, we will show the boundedness and convergence of the off-diagonal operators mentioned in equation~\eqref{eqn:offdiagonal_bound}. Fix $\alpha = ij, \alpha' = kl$ with $\alpha \neq \alpha'$. Let $f \in \cali{S}(\R^{2\intv{n}}_{\alpha}), g \in \cali{S}(\R^{2\intv{n}}_{\alpha'})$ of the form $f = f_{1}f_{2}$ and $g=g_{1}g_{2}$ with $f_{1}, g_{1} \in \cali{S}(\R^2), f_{2} \in \cali{S}(\R^{2\intv{n}\setminus\rel}_{\alpha})$ and $g_{2} \in \cali{S}(\R^{2\intv{n}\setminus\rel}_{\alpha'})$. By the definition of the kernel,
\begin{equation*}
	\langle f, \preoffdiag{\alpha \alpha'}(\lambda)g \rangle  = \int_{\R^{2\intv{n}}_{\alpha}} \int_{\R^{2\intv{n}}_{\alpha'}} \ \, \sqrrange(y_{\rel})f(y) \conj{\epsgreens(\lambda, \epsfloor{\Seps{\alpha}y} - \epsfloor {\Seps{\alpha'}y'})} \conj{\sqrrange(y'_{\rel})g(y')} \, \d y \, \d y' . \\
\end{equation*}
The first step here is to replace the lattice Green's function by its continuous interpolation along with the error incurred.  Now, we do a change of variables with $\Seps{\alpha}y \mapsto x$ (so $\d y \mapsto \varepsilon^{-2}\d x$) and similarly for  the $y'$ variable. For notational convenience, let $F_{\varepsilon \alpha}(x) \coloneqq \varepsilon^{-2}(\sqrrange f \circ \Teps{\alpha})(x)$ and $G_{\varepsilon \alpha'}(x') \coloneqq \varepsilon^{-2}(\sqrrange g \circ \Teps{\alpha'})(x')$. With this new definition, we rewrite the above inner product as:
\begin{equation}{\label{eqn:offdiagonal_position_one}}
	\langle f, \preoffdiag{\alpha \alpha'}(\lambda) g \rangle = \int_{\R^{2\intv{n}}}  \int_{\R^{2\intv{n}}}  \, F_{\varepsilon\alpha}(x) \conj{\epscontgreens(\lambda, x-x')} G_{\varepsilon\alpha'}(x') \, \d x \, \d x'  + \cali{O}(\varepsilon, f,g) \, ,
\end{equation}
where
\begin{equation*}
	\cali{O}(\varepsilon, f, g) \coloneqq \int_{\R^{2\intv{n}}} \int_{\R^{2\intv{n}}}  \, F_{\varepsilon\alpha}(x)\big[\conj{\epsgreens(\lambda, \epsfloor{x} - \epsfloor{x'}) - \epscontgreens(\lambda, x - x')}\big]G_{\varepsilon\alpha'}(x') \, \d x \, \d x' \, .
\end{equation*}
Let us first look at the term involving $\epscontgreens(\lambda)$. Using Plancherel's theorem, write the integral as
\begin{equation}{\label{eqn:preoffdiagonalcont_fourier}}
	\int_{\epshattorus{2\intv{n}}} \frac{\F{F_{\varepsilon\alpha}}(\xi)\conj{\F{G_{\varepsilon\alpha'}}(\xi)} }{\lambda + \fourierepsgen(\xi)} \, \d \xi \, .
\end{equation}
Like we did in the previous subsection, we can calculate $\F{F_{\varepsilon\alpha}}(\xi)$ and $\F{G_{\varepsilon\alpha'}}(\xi)$ in terms of the Fourier transforms of $f$ and $g$ through a change of variables $x \mapsto \Seps{\alpha}y$ and $x' \mapsto \Seps{\alpha'}y'$ and using the definition of $\Sdualeps{\alpha}$ and $\Sdualeps{\alpha'}$. The resulting expressions are:
\begin{equation*}
	\F{F_{\varepsilon \alpha}}(\xi) = \frac{\F{f_{2}}(\xi_i + \xi_j, \xi_{\intv{n}\setminus\alpha})}{2\pi} \int_{\R^{2}} \sqrrange(y_{\rel})f_{1}(y_{\rel})e^{\im \frac{\varepsilon}{2}(\xi_i - \xi_j) \cdot y_{\rel}}\, \d y_{\rel} \, .
\end{equation*}
and
\begin{equation*}
	\F{G_{\varepsilon \alpha'}}(\xi) = \frac{\F{g_{2}}(\xi_k + \xi_l,\xi_{\intv{n}\setminus\alpha'})}{2\pi}\int_{\R^{2}}\sqrrange(y'_{\rel})g_{1}(y'_{\rel})e^{\im\frac{\varepsilon}{2}(\xi_k - \xi_l)\cdot y'_{\rel}}\, \d y'_{\rel} \, . 
\end{equation*}
Insert this into equation~\eqref{eqn:preoffdiagonalcont_fourier} and do a change of variables $\xi \mapsto \Tdualone{\alpha'}\eta$. Then, using the bound $(\lambda + \fourierepsgen(\xi))^{-1} \lesssim (\lambda + \half \norm{\xi}^2)^{-1}$ along with Lemma~\ref{lem:guquasteltsailemma}, we can bound the inner product in equation~\eqref{eqn:preoffdiagonalcont_fourier} by $\norm{f}_{2}\norm{g}_{2}$ up to a constant. Now we prove the same bound for the error term $\cali{O}(\varepsilon,f,g)$. Just like the case for the outgoing operators, we do a first order Taylor expansion $e^{-\im \xi \cdot (\epsfracpart{x} - \epsfracpart{x'})} -1 = - \im\xi \cdot (\epsfracpart{x} - \epsfracpart{x'})\int_{0}^{1}e^{-\im z \xi \cdot (\epsfracpart{x} - \epsfracpart{x'})} \, \d z$. For $l, m \in \{1, 2 \ldots n\}$ and $\xi \in \epshattorus{2\intv{n}}$, define the functions $\widetilde{F}_{\varepsilon \alpha,l}$ and $\widetilde{G}_{\varepsilon\alpha', m}$ as
\begin{equation*}
\begin{split}
\widetilde{F}_{\varepsilon\alpha ,l}(\xi) & \coloneqq \frac{1}{(2\pi)^{n}}\int_{\R^{2\intv{n}}} \int_{0}^{1} F_{\varepsilon\alpha}(x) (\xi_l \cdot \epsfracpart{x_l}) e^{-\im z \xi \cdot \epsfracpart{x}+\im \xi \cdot x}\d x \, \d z \\ \widetilde{G}_{\varepsilon\alpha', m}(\xi) & \coloneqq \frac{1}{(2\pi)^{n}}\int_{\R^{2\intv{n}}} \int_{0}^{1} G_{\varepsilon \alpha'}(x') (\xi_m \cdot \epsfracpart{x'_m} )e^{-\im z \xi \cdot \epsfracpart{x'} + \im \xi \cdot x'}\d x' \, \d z \, .
\end{split}
\end{equation*} 
Let $\widetilde{F}_{\varepsilon\alpha, 0}$ and $\widetilde{G}_{\varepsilon\alpha', 0}$ denote the above functions with $\xi_l \cdot \epsfracpart{x_l}$ and $\xi_m \cdot \epsfracpart{x'_m}$ respectively omitted. With this notation, we can rewrite the error term as
\begin{equation}{\label{eqn:preoffdiagonalerror_fourier}}
	\cali{O}(\varepsilon, f,g) =  -\im\sum_{l=1}^{n}  \int_{\epshattorus{2\intv{n}}}\frac{\widetilde{F}_{\varepsilon\alpha, l}(\xi)\conj{\widetilde{G}_{\varepsilon\alpha', 0}(\xi)}}{\lambda + \fourierepsgen(\xi)} \d \xi + \im  \sum_{m=1}^{n} \int_{\epshattorus{2\intv{n}}}\frac{\widetilde{F}_{\varepsilon\alpha, 0}(\xi)\conj{\widetilde{G}_{\varepsilon \alpha',m}(\xi)}}{\lambda + \fourierepsgen(\xi)} \d \xi \, \, .
\end{equation}
We can bound this in exactly the same way as we did for $\cali{I}(\varepsilon, f,g)$ by using $F_{\varepsilon\alpha,l}(x), G_{\varepsilon\alpha', m}(x')$ for $l,m = 1, 2, \ldots n$. In this case, we use the bounds $|\widetilde{F}_{\varepsilon\alpha, l}(\xi)| \lesssim \norm{\sqrrange}_{2}\norm{f_{1}}_{2} \, |\F{f_{2}}(\xi)|$ and $|\widetilde{G}_{\varepsilon\alpha', m}(\xi)| \lesssim \norm{\sqrrange}_{2} \norm{g_{1}}_{2} \, |\F{g_{2}}(\xi)|$ along with Lemma~\ref{lem:guquasteltsailemma} to obtain the bound
\begin{equation*}
	|\cali{O}(\varepsilon, f, g)| \lesssim \norm{\range}_{1}\norm{f}_{2}\norm{g}_{2} \, \, .
\end{equation*}
To avoid repetition, we skip the details and directly conclude
\begin{equation*}
    \normop{\preoffdiag{\alpha\alpha'}(\lambda)} = \sup_{\norm{f}_{2}, \norm{g}_{2}\leqslant 1} |\langle f, \preoffdiag{\alpha\alpha'}(\lambda)g \rangle| \lesssim \sup_{\norm{f}_{2}, \norm{g}_{2}\leqslant 1} \norm{f}_{2}\norm{g}_{2} \lesssim 1 \, \, .
\end{equation*}
This proves the norm bound on the prelimiting operator. To show norm convergence, we again proceed by comparing it to its limiting counterpart. To that end, choose $R \ll \varepsilon^{-1}$ and define the errors:
 \begin{align*}
\cali{O}_{1}(R, \varepsilon) &\coloneqq   \frac{\ind{\norm{\xi} \leqslant R}(e^{\im \xi \cdot (\epsfloor{x} - \epsfloor{x'})} - e^{\im \xi \cdot (x-x')})}{\lambda + \fourierepsgen(\xi)} \\
\cali{O}_{2}(R,\varepsilon) & \coloneqq \frac{ \ind{R < \norm{\xi} \lesssim \pi/\varepsilon} (e^{\im \xi \cdot (\epsfloor{x} - \epsfloor{x'})} - e^{\im \xi \cdot (x-x')})}{\lambda + \fourierepsgen(\xi)} \\    
\cali{O}_{3}(R, \varepsilon) &\coloneqq  \ind{\norm{\xi} \leqslant R} \frac{e^{\im \frac{\varepsilon}{2}(\xi_i -\xi_j)\cdot (y_{\rel} - y'_{\rel})} - 1}{\lambda + \fourierepsgen(\xi)}  \qquad  \, \cali{O}_{4}(R,\varepsilon) \coloneqq \ind{R < \norm{\xi} \lesssim \pi/\varepsilon}   \frac{e^{\im \frac{\varepsilon}{2}(\xi_i -\xi_j)\cdot (y_{\rel} - y'_{\rel})} - 1}{\lambda + \fourierepsgen(\xi)}   \\ 
\cali{O}_{5}(R, \varepsilon) & \coloneqq  \Bigg[\frac{\ind{\norm{\xi} \leqslant R}}{\lambda + \fourierepsgen(\xi)} - \frac{\ind{\norm{\xi} \leqslant R}}{\lambda + \half\norm{\xi}^2}\Bigg] \qquad	\cali{O}_{6}(R, \varepsilon)  \coloneqq \Bigg[\frac{\ind{R < \norm{\xi} \lesssim \pi/\varepsilon}}{\lambda + \fourierepsgen(\xi)} - \frac{\ind{R < \norm{\xi} \lesssim \pi/\varepsilon}}{\lambda + \half\norm{\xi}^2}\Bigg] \\
\cali{O}_{7}(R, \varepsilon) & \coloneqq  \frac{-\ind{\norm{\xi} \gtrsim  \pi/\varepsilon}}{\lambda + \half\norm{\xi}^2}
\end{align*}
 Then from equations~\eqref{eqn:offdiagonal_fourier_one} and \eqref{eqn:offdiagonal_position_one}, we see that
 \begin{align}{\label{eqn:offdiagonalerror_zero}}  
 & \langle f,[\preoffdiag{\alpha\alpha'}(\lambda) - \offdiag{\alpha\alpha'}(\lambda)]g \rangle  \nonumber  \\
 & = \frac{1}{(2\pi)^{2n}}\sum_{r=1}^{2}\int_{\R^{2\intv{n}}} \int_{\R^{2\intv{n}}} F_{\varepsilon\alpha}(x) \cali{O}_{r}(R, \varepsilon) G_{\varepsilon\alpha'}(x') \, \d x \, \d x'+ \frac{1}{(2\pi)^{2}}\int_{\R^2} \int_{\R^2}\sqrrange(y_{\rel})\sqrrange(y'_{\rel})   \nonumber \\
 & \qquad \qquad \qquad \times f_{1}(y_{\rel})g_{1}(y'_{\rel})\sum_{r=3}^{7}\int_{\epshattorus{2\intv{n}}} \F{f_{2}}(\xi_i + \xi_j, \xi_{\intv{n}\setminus\alpha})  \cali{O}_{r}(R, \varepsilon) \conj{\F{g_{2}}(\xi_k + \xi_l, \xi_{\intv{n}\setminus\alpha'})} \, \d \xi \, \d y_{\rel} \, \d y'_{\rel} \nonumber \\
 & \coloneqq \sum_{r=1}^{7} \cali{O}^{f,g}_{r}(R, \varepsilon) \, \, .
 \end{align}
The proof of the error bounds are exactly the same as in the case for the outgoing operator, except now instead of using Cauchy--Schwarz inequality to bound the Fourier integrals, we use Lemma~\ref{lem:guquasteltsailemma}. We skip the details.
\begin{equation*}
	\normop{\preoffdiag{\alpha\alpha'}(\lambda) - \offdiag{\alpha\alpha'}(\lambda)} \xrightarrow{\varepsilon \to 0} 0 \, \, .
\end{equation*}

\subsection{Diagonal Operators}

To finish the proof of Lemma~\ref{lem:resolventconvergence}, it remains to verify the norm boundedness and convergence of the prelimiting diagonal operators, which we do in this section. Recall from equation \eqref{eqn:defn_prelimiting_operators} that for $\alpha = ij$, the prelimiting diagonal resolvent can be expressed as an operator on $\Ltwo(\R^{2\intv{n}}_{\alpha})$:
\begin{equation}{\label{eqn:prediagonal_neumannseries}}
	\prediag{\theta}{\alpha}(\lambda) = \sum_{m = 0}^{\infty}\beta_{\varepsilon}^{m+1}\preoffdiag{\alpha\alpha}^{m}(\lambda) = \beta_{\varepsilon} \sum_{m=0}^{\infty} \big( \beta_{\varepsilon} \preoffdiag{\alpha\alpha}(\lambda) \big)^m \, \, .
\end{equation}  
To estimate the norm of $\prediag{\theta}{\alpha}(\lambda)$, we need a suitable estimate for the norm of $\preoffdiag{\alpha\alpha}(\lambda)$. However, unlike the off-diagonal operators in the previous section, here, the norm of $\preoffdiag{\alpha\alpha}(\lambda)$ diverges logarithmically as $\varepsilon \to 0$. In particular, we will show that up to leading order, the norm satisfies
\begin{equation*}
	\normop{\preoffdiag{\alpha\alpha}(\lambda)} \lesssim  |\log\varepsilon| + c_{\varepsilon}(\lambda) \qquad \text{ as  } \quad \varepsilon \to 0 \, \, ,
\end{equation*}
where $c_{\varepsilon}(\lambda)$ is a constant which depends on $\lambda$ and which converges to a finite limit as $\varepsilon \to 0$. Since $\beta_{\varepsilon} \sim |\log\varepsilon|^{-1}$, this divergence is canceled after multiplication by $\beta_{\varepsilon}$ so that the resulting norm is governed by $c_{\varepsilon}(\lambda)$. We then choose $\lambda$ sufficiently large so that $c_{\varepsilon}(\lambda) < 1$, which results in
\begin{equation*}
	\normop{\beta_{\varepsilon}\preoffdiag{\alpha\alpha}(\lambda)} < 1 \, \, ,
\end{equation*}
as $\varepsilon \to 0$, ensuring that the Neumann series above converges absolutely and at the same time yielding a norm estimate. After this, we proceed to show the convergence to its limiting counterpart. The proof here will be carried out in two steps. First, we view $\prediag{\theta}{\alpha}(\lambda)$ as an operator on $\Ltwo(\R^2 \times \widehat{\R}^{2\intv{n}\setminus\rel}_{\alpha}) \cong \Ltwo(\R^{2\intv{n}}_{\alpha})$. Fixing the $\eta_{\com \cup \intv{n} \setminus \alpha}$ variable, we prove that it converges in norm to an operator on $\Ltwo(\R^2)$, at a rate which is uniformly bounded in $\eta_{\com \cup \intv{n} \setminus \alpha}$. In the second step, we use this result to establish convergence in $\Ltwo(\widehat{\R}^{2\intv{n}\setminus\rel}_{\alpha})$ with respect to the strong operator topology. This is the method used in \cite{gu2021moments} and what we present here is almost the same barring notational differences. We divide the proof into two subsections.

\subsubsection{Boundedness of the Pre-Diagonal Operator}

Fix $f, g \in \cali{S}(\R^{2\intv{n}}_{\alpha})$ of the form $f= f_{1}f_{2}$ and $g = g_{1}g_{2}$. As in the sections before, we first replace the lattice Green's function in $\langle f, \preoffdiag{\alpha\alpha}(\lambda) g\rangle$ with its continuous modification $\epscontgreens(\lambda)$. The inner product then becomes:
\begin{equation*}
		\int_{\R^{2\intv{n}}_{\alpha}} \int_{\R^{2\intv{n}}_{\alpha}} f(y) \sqrrange(y_{\rel}) \conj{\epscontgreens(\lambda, \Seps{\alpha}y - \Seps{\alpha}y')}  \conj{\sqrrange(y'_{\rel})g(y')} \, \d y \, \d y' + \langle f, \cali{D}_{1}(\varepsilon) g \rangle \, \, ,
\end{equation*}
where $ \langle f, \cali{D}_{1}(\varepsilon) g \rangle$ is defined as:
\begin{equation}{\label{eqn:prediagonal_error_one}}
	\int_{\R^{2\intv{n}}_{\alpha}} \int_{\R^{2\intv{n}}_{\alpha}} f(y) \sqrrange(y_{\rel}) [\conj{\epsgreens(\lambda, \epsfloor{\Seps{\alpha}y} - \epsfloor{\Seps{\alpha}y'})} - \conj{\epscontgreens(\lambda, \Seps{\alpha}y - \Seps{\alpha}y')}] \sqrrange(y'_{\rel})g(y') \, \d y \d y'  .
\end{equation}
We begin with the term involving just $\epscontgreens(\lambda)$. Using Plancherel's theorem, it has the following expression:
  \begin{align}{\label{eqn:prediagonal_fourier_one}}
   	 \int_{\R^2}\int_{\R^2} 
   	 &  \sqrrange(y_{\rel}) f_{1}(y_{\rel})\sqrrange(y'_{\rel})g_{1}(y'_{\rel}) \nonumber \\
   	 & \times \frac{1}{(2\pi)^{2}}\int_{\epshattorus{2\intv{n}}} \frac{\F{f_{2}}(\xi_i + \xi_j, \xi_{\intv{n}\setminus\alpha})e^{\im \frac{\varepsilon}{2}(\xi_i - \xi_j) \cdot (y_{\rel} - y'_{\rel})} \conj{\F{g_{2}}(\xi_i + \xi_j, \xi_{\intv{n}\setminus\alpha})}}{\lambda + \fourierepsgen(\xi)} \, \d \xi \,\d y_{\rel} \, \d y'_{\rel} \, .
   \end{align}
 Switch to the center-of-mass and relative frame in Fourier space by the change of variables $\xi = \Tdualone{\alpha}\eta$. The integral above then becomes:
    \begin{equation*}
   	\frac{1}{(2\pi)^{2}}\int_{\R^2}\int_{\R^2} \int_{\epsalhattorus{2\intv{n}}} \frac{ \sqrrange(y_{\rel}) f_{1}(y_{\rel})\F{f_{2}}(\eta_{\com\cup\intv{n}\setminus\alpha})e^{\im \varepsilon\eta_{\rel} \cdot (y_{\rel} - y'_{\rel})} \sqrrange(y'_{\rel})g_{1}(y'_{\rel})\conj{\F{g_{2}}(\eta_{\com\cup\intv{n}\setminus\alpha})}}{\lambda + \fourierepsgen(\Tdualone{\alpha}\eta)}  \d y_{\rel} \d y'_{\rel}  \d \eta  \, .
   \end{equation*}
    To understand where this logarithmic divergence arises from, observe that as $\varepsilon \to 0$, the denominator approaches $\lambda + \half \norm{\Tdualone{\alpha}\eta}^2$. By Lemma~\ref{lem:uniformgreensbound}, this is approximately $ \lambda + \norm{\eta_{\rel}}^2
    + \norm{\eta_{\com \cup \intv{n} \setminus \alpha}}^2$. Apart from this term, the only dependence on $\eta_{\rel}$ appears in the
    phase factor $ e^{\im \varepsilon \eta_{\rel} \cdot (y_{\rel} - y'_{\rel})}$, which tends to $1$ in the limit. Thus, as $\varepsilon \to 0$, the $\eta_{\rel}$ integral is essentially that of $\norm{\eta_{\rel}}^{-2}$ over $\widehat{\R}^2$, which diverges in a logarithmic way. In our proof, it will be more convenient to transfer this divergence from a neighborhood of infinity to near the origin, where this singular behavior can be analyzed more explicitly. We do this by further performing a change of variables $\eta_{\rel} \mapsto \eta_{\rel}/\varepsilon$. After this rescaling, the domain $\epsalhattorus{2\intv{n}}$ becomes $\mealtorus{\alpha}{2\intv{n}} \coloneqq \metorus{4} \times \epshattorus{2\intv{n}\setminus \alpha} $ with
   \begin{equation*}
   	 \metorus{4} \coloneqq \bigg\{(\eta_{\rel}, \eta_{\com}) : \eta_{\rel} = (\eta_{\rel}^{1}, \eta_{\rel}^{2}), \eta_{\com} = (\eta_{\com}^{1}, \eta_{\com}^{2}) \text{ and }  \max_{k=1,2}\Big|\eta_{\rel}^{k} - \frac{\varepsilon\eta_{\com}^{k}}{2}\Big| \leqslant \pi, \, \max_{k=1,2}\Big|\eta_{\rel}^{k} + \frac{\varepsilon \eta_{\com}^{k}}{2}\Big| \leqslant \pi \bigg\},
   \end{equation*} 
   and $\epsalhattorus{2\intv{n}\setminus\alpha} =\{(\eta_k)_{k \in \intv{n}\setminus\alpha} : \eta_k \in \epshattorus{2}\}$. With this notation, the integral takes the form
   \begin{equation}{\label{eqn:prediagonal_fourier_two}}
   	\frac{1}{(2\pi)^{2}}\int_{\R^2}\int_{\R^2} \int_{\mealtorus{\alpha}{2\intv{n}}} \frac{ \sqrrange(y_{\rel}) f_{1}(y_{\rel})\F{f_{2}}(\eta_{\com\cup\intv{n}\setminus\alpha})e^{\im \eta_{\rel} \cdot (y_{\rel} - y'_{\rel})} \sqrrange(y'_{\rel})g_{1}(y'_{\rel})\conj{\F{g_{2}}(\eta_{\com\cup\intv{n}\setminus\alpha})}}{\varepsilon^2\lambda + \fouriergen(\mealeta{\alpha})}  \d y_{\rel} \d y'_{\rel}  \d \eta  .
   \end{equation}
   Recall that $\fourierepsgen( \cdot ) = \varepsilon^{-2}\fouriergen(\varepsilon \, \cdot ) \,$. The factor $\varepsilon^{-2}$ arising from the Jacobian of the transformation has been absorbed into the denominator of the integrand, changing $\fourierepsgen$ to $\fouriergen$. Its argument $\mealeta{\alpha} \in \mealtorus{\alpha}{2\intv{n}}$ is given by:
   \begin{equation}{\label{eqn:prediagonal_modifiedeta}}
   	\mealeta{\alpha} \coloneqq \Big\{\frac{\varepsilon \eta_{\com}}{2} + \eta_{\rel}, \frac{\varepsilon \eta_{\com}}{2}-\eta_{\rel},\varepsilon \eta_{\intv{n}\setminus\alpha}\Big\} \, .
   \end{equation}    
  From its definition, observe that as $\varepsilon \to 0$, the $\eta_{\rel}$ component restricts to $\hattorus{2}$ and $\eta_{\com}$ loses its dependence on the set, so that $\lim_{\varepsilon \to 0}\metorus{4} = \hattorus{2} \times \widehat{\R}^2$. Moreover, for $(\eta_{\rel},\eta_{\com}) \in \metorus{4}$,
  \begin{equation*}
  	| 2\eta_{\rel}^{1} | \leqslant \bigg| \eta_{\rel}^{1} + \frac{\varepsilon \eta_{\com}^{1}}{2}\bigg| + \bigg| \eta_{\rel}^{1} - \frac{\varepsilon \eta_{\com}^{1}}{2}\bigg| \lesssim \pi \, ,
  \end{equation*}
  and similarly for $\eta_{\rel}^{2}$. Thus, we may choose an $R > 0$ such that $\norm{\eta_{\rel}} \lesssim R$ uniformly in $\eta_{\com}$. Similar reasoning shows that $\varepsilon\norm{\eta_{\com}} \lesssim \pi$.
  
   To isolate the leading logarithmic term together with its dependence on $\lambda$, we will decompose the expression in equation~\eqref{eqn:prediagonal_fourier_two} into a sum of four integrals, the first three of which remain bounded as $\varepsilon \to 0$. The fourth integral contains the divergence in an explicit manner, which allows us to extract it. It will be convenient to define the slices $\metorus{4, \rel}(\eta_{\com}) \coloneqq \{\eta_{\rel} \in \hattorus{2} : (\eta_{\rel}, \eta_{\com}) \in \metorus{4}\}$ and $\metorus{4, \com} \coloneqq \{\eta_{\com} : \exists \eta_{\rel} \text{ such that } (\eta_{\rel}, \eta_{\com}) \in \metorus{4}\}$. We will also use $\ballorigin{R}$ to denote a ball of radius $R$ in $\R^2_{\alpha}$ or $\widehat{\R}^2_{\alpha}$. For $\eta_{\rel} \in \metorus{4, \rel}(\eta_{\com})$, define the kernels:
   \begin{equation*}
   	\begin{split}
   	   & \cali{D}_{2}(\varepsilon, y_{\rel} ,y'_{\rel},\eta_{\com \cup \intv{n} \setminus \alpha}) \coloneqq \sqrrange(y_{\rel})\Bigg[\int_{\metorus{4, \rel}(\eta_{\com})}\Bigg(\frac{e^{\im \eta_{\rel} \cdot (y_{\rel} - y'_{\rel})}}{\varepsilon^2 \lambda + \fouriergen(\mealeta{\alpha})} - \frac{1}{\varepsilon^2 \lambda + \fouriergen(\mealeta{\alpha})}\Bigg)\, \d \eta_{\rel}\Bigg]\sqrrange(y'_{\rel}) \\
   	   & \cali{D}_{3}(\varepsilon,y_{\rel},y'_{\rel},\eta_{\com \cup \intv{n} \setminus \alpha}) \coloneqq  \int_{\metorus{4, \rel}(\eta_{\com})}\Bigg[\frac{\sqrrange(y_{\rel})\sqrrange(y'_{\rel})}{\varepsilon^2 \lambda + \fouriergen(\mealeta{\alpha})} - \frac{\sqrrange(y_{\rel})\sqrrange(y'_{\rel})}{\varepsilon^2 \lambda + \quart \norm{\varepsilon\eta_{\com}}^2 + \half \norm{\varepsilon\eta_{\intv{n}\setminus\alpha}}^2 + \norm{\eta_{\rel}}^2} \Bigg]  \d \eta_{\rel} \\
   	   & \cali{D}_{4}(\varepsilon,y_{\rel}, y'_{\rel},\eta_{\com \cup \intv{n} \setminus \alpha}) \coloneqq \sqrrange(y_{\rel})\Bigg[\int_{\metorus{4, \rel}(\eta_{\com})}\frac{\mathbf{1}_{\metorus{4,\rel}(\eta_{\com})\setminus\ballorigin{\pi}}}{\varepsilon^2 \lambda + \quart \norm{\varepsilon\eta_{\com}}^2 + \half \norm{\varepsilon\eta_{\intv{n}\setminus\alpha}}^2 + \norm{\eta_{\rel}}^2} \, \d \eta_{\rel} \Bigg]\sqrrange(y'_{\rel}) \\
   	   & \cali{D}_{5}(\varepsilon,y_{\rel}, y'_{\rel}, \eta_{\com \cup \intv{n} \setminus \alpha}) \coloneqq  \sqrrange(y_{\rel}) \Bigg[\int_{\metorus{4, \rel}(\eta_{\com})}\frac{\mathbf{1}_{\ballorigin{\pi}}}{\varepsilon^2 \lambda + \quart \norm{\varepsilon\eta_{\com}}^2 + \half \norm{\varepsilon\eta_{\intv{n}\setminus\alpha}}^2 + \norm{\eta_{\rel}}^2} \, \d \eta_{\rel} \Bigg]\sqrrange(y'_{\rel}) \, \, .
   	\end{split}
   \end{equation*}
   In terms of these definitions, equation~\eqref{eqn:prediagonal_fourier_two} takes the form:
    \begin{align*}
   \int_{\R^2 }\int_{\R^2} 
   & \int_{\metorus{4, \com}\times \epsalhattorus{2\intv{n}\setminus\alpha}}   f_{1}(y_{\rel})\F{f_{2}}(\eta_{\com \cup\intv{n}\setminus\alpha}) \\
   & \times \Bigg[  \frac{1}{(2\pi)^{2}}\sum_{r=2}^{5}\cali{D}_{r}(\varepsilon, y_{\rel}, y'_{\rel}, \eta_{\com \cup \intv{n} \setminus \alpha})\Bigg] \times g_{1}(y'_{\rel})\conj{\F{g_{2}}(\eta_{\com \cup\intv{n}\setminus\alpha})}\d \eta \, \d y_{\rel}\d y'_{\rel} \\
   & \eqqcolon \int \F{f_{2}}(\eta_{\com \cup \intv{n} \setminus \alpha})\Big[\sum_{r=2}^{5}\langle f_{1},\cali{D}_{r}(\varepsilon, \eta_{\com \cup \intv{n} \setminus \alpha}) g_{1} \rangle\Big]\conj{\F{g_{2}}(\eta_{\com \cup \intv{n} \setminus \alpha})}  \, \d \eta_{\com \cup\intv{n} \setminus \alpha} \\
   & \eqqcolon \sum_{r=2}^{5} \langle f ,\cali{D}_{r}(\varepsilon) g \rangle \, \, .
   \end{align*}
   The functions $\cali{D}_{r}(\varepsilon,y_{\rel}, y'_{\rel}, \eta_{\com\cup\intv{n}\setminus\alpha})$ act as multipliers on $\Ltwo(\widehat{\R}^{2\intv{n}\setminus \rel}_{\alpha})$, in the variables $\eta_{\com \cup\intv{n}\setminus\alpha}$, and as genuine kernels on $\Ltwo(\R^2)$ in the $y_{\rel},y'_{\rel}$ variables. Below, we derive norm bounds for each of the above operators. For this purpose, it suffices to estimate the norm of
   $\cali{D}_{r}(\varepsilon,\eta_{\com \cup \intv{n} \setminus \alpha})$
   as an operator on $\Ltwo(\R^2)$ for each fixed
   $\eta_{\com \cup \intv{n} \setminus \alpha}$, and then show that this norm is
   bounded uniformly in $\eta_{\com \cup \intv{n} \setminus \alpha}$. In addition, we will show that for $r = 2, 3,4$ and every $\eta_{\com \cup \intv{n} \setminus \alpha}$, the kernels converge in norm to an operator on $\Ltwo(\R^2)$, which may depend on $\eta_{\com \cup \intv{n} \setminus \alpha}$. Together with the uniform bounds, this implies strong convergence of the multipliers on $\Ltwo(\widehat{\R}^{2\intv{n}\setminus\rel}_{\alpha})$. The operator $\cali{D}_{5}(\varepsilon)$ encapsulates the logarithmic divergence which will be offset by the constant $\beta_{\varepsilon}$ after the summation. Let us begin with the operator $\cali{D}_{2}(\varepsilon)$. In what follows, we fix $\eta_{\com \cup \intv{n}\setminus\alpha} \neq \mathbf{0}$. By definition,
    \begin{align*}
        \langle f_{1} ,\cali{D}_{2}(\varepsilon,\eta_{\com \cup \intv{n} \setminus \alpha}) g_{1} \rangle = \int_{\R^2}  
    	& \int_{\R^2} f_{1}(y_{\rel})\sqrrange(y_{\rel})\sqrrange(y'_{\rel}) g_{1}(y'_{\rel}) \\
    	& \times   \frac{1}{(2\pi)^2}\int_{\metorus{4,\rel}(\eta_{\com})} \Bigg[\frac{e^{\im \eta_{\rel} \cdot (y_{\rel} - y'_{\rel})}}{\varepsilon^2 \lambda + \fouriergen(\mealeta{\alpha})} - \frac{1}{\varepsilon^2 \lambda + \fouriergen(\mealeta{\alpha})}\Bigg] \, \d \eta_{\rel} \, \d y_{\rel} \, \d y'_{\rel}  \, .
    \end{align*} 
    Recall from Lemma~\ref{lem:uniformgreensbound} that the discrete symbol satisfies:
    \begin{equation}{\label{eqn:prediagonal_symbolbound}}
    	\fouriergen(\mealeta{\alpha}) = \sum_{k=1}^{n}\sum_{m_k \in \Z^{2}}\Big(1-\cos\big((\mealeta{\alpha})_k \cdot m_k\big)\Big) \nu(m_k) \gtrsim \norm{\mealeta{\alpha}}^2 \gtrsim \norm{\eta_{\rel}}^2 \, .
    \end{equation} 
Furthermore, due to the continuity of the symbol and the definition of $\mealeta{\alpha}$ in equation~\eqref{eqn:prediagonal_modifiedeta}, we have 
\begin{equation*}
	\lim_{\varepsilon \to 0} \fouriergen(\mealeta{\alpha})   = \fouriergen(\mzaleta{\alpha}) \, \, \, ,
\end{equation*}
with $\mzaleta{\alpha} \coloneqq \{\eta_{\rel}, -\eta_{\rel}, \mathbf{0}_{\intv{n}\setminus\alpha}\}$. More explicitly, 
  \begin{align*}
   	\fouriergen(\mzaleta{\alpha}) 
   	& = \sum_{m_1 \in \Z^2}\big(1-\cos(\eta_{\rel} \cdot m_1)\big)\nu(m_1) + \sum_{m_2 \in \Z^2}\big(1-\cos(\eta_{\rel} \cdot m_2)\big)\nu(m_2) + 0 \, \\
   	& = 2\sum_{m_1 \in \Z^2}\sin^2\bigg(\frac{\eta_{\rel} \cdot m_1}{2}\bigg)\nu(m_1) + 2\sum_{m_2 \in \Z^2}\sin^2\bigg(\frac{\eta_{\rel} \cdot m_2}{2}\bigg)\nu(m_2) \, \, .
  \end{align*}
  Near the origin, the limiting symbol behaves like $\norm{\eta_{\rel}}^2$. This asymptotic quadratic behavior of the denominator induces a logarithmic divergence in both the $\d \eta_{\rel}$ integrals in $\langle f_{1} ,\cali{D}_{2}(\varepsilon, \eta_{\com \cup \intv{n} \setminus \alpha}) g_{1} \rangle$, individually, in the limit. But since we are taking a difference, it ends up being canceled. To make this precise, choose a ball $\ballorigin{R}$ such that $\metorus{4,\rel}(\eta_{\com}) \subset \ballorigin{R}$ uniformly for all $\eta_{\com} \in \metorus{4,\com}$.  In this ball, using Lemma~\ref{lem:uniformgreensbound}, we have the following inequality:
   \begin{equation*}{\label{eqn:diagonal_error_two}}
   \Bigg| \frac{e^{\im \eta_{\rel} \cdot (y_{\rel} - y'_{\rel})}-1}{\varepsilon^2 \lambda + \fouriergen(\mealeta{\alpha})} \Bigg| \leqslant \frac{\norm{\eta_{\rel}}\norm{y_{\rel} - y'_{\rel}} \wedge 2}{\norm{\eta_{\rel}}^2} \, \, .
   \end{equation*} 
   The right-hand side of the above equation is locally integrable in two dimensions. We show this separately according to the value the minimum attains. First, assume $y_{\rel}, y'_{\rel}$ is such that $\norm{y_{\rel} - y'_{\rel}} \geqslant 2 R^{-1}$. In this region, if $\norm{\eta_{\rel}} \leqslant 2\norm{y_{\rel} - y'_{\rel}}^{-1}$, then the above minimum is $\norm{\eta_{\rel}} \norm{y_{\rel} - y'_{\rel}}$ and in the complementary region $\{2 \norm{y_{\rel} - y'_{\rel}}^{-1} < \norm{\eta_{\rel}} < R\}$, it equals $2$. Thus, switching to polar coordinates, the integral is
   \begin{equation*}
   	2\pi \int_{0}^{2\norm{y_{\rel} - y'_{\rel}}^{-1}} \frac{r^2 \norm{y_{\rel} - y'_{\rel}}}{r^2} \, \d r + 2\pi \int_{2\norm{y_{\rel} -y'_{\rel}}^{-1}}^{R} \frac{2r}{r^2} \, \d r 
    = 4\pi + 4\pi \log\frac{R}{2} + 4\pi \log\norm{y_{\rel} - y'_{\rel}} \, \, .
   \end{equation*}    
   In the set $\norm{y_{\rel} - y'_{\rel}} < 2R^{-1}$, we employ a similar strategy. However, in this case, the minimum cannot be attained at $2$ as $ \norm{\eta_{\rel}} < R < 2\norm{y_{\rel} - y'_{\rel}}^{-1}$ always. Hence, in this region, the integral evaluates to
   \begin{equation*}
   	2\pi \int_{0}^{R} \frac{r^2 \norm{y_{\rel} - y'_{\rel}}}{r^2} \, \d r = 2\pi \norm{y_{\rel} - y'_{\rel}} R < 4\pi \, \, .
   \end{equation*}
   Therefore, we can use dominated convergence theorem to pass the limit inside the integral. Furthermore, since the bounds obtained above do not depend on the chosen $\eta_{\com \cup \intv{n}\setminus\alpha}$, we have that
   \begin{align}{\label{eqn:prediagonal_D2_bound}}
   |\langle f, \cali{D}_{2}(\varepsilon) g\rangle| 
   & \lesssim \norm{f_{2}}_{2} \norm{g_{2}}_{2} \Bigg|\int_{\R^2} \int_{\R^2}  f_{1}(y_{\rel}) \sqrrange(y_{\rel})(1 + |\log\norm{y_{\rel} - y'_{\rel}}|) \sqrrange(y'_{\rel})g_{1}(y'_{\rel}) \, \d y_{\rel} \, \d y'_{\rel} \,\Bigg|\nonumber \\
   	& \lesssim  \norm{f_{2}}_{2} \norm{g_{2}}_{2} \,  \norm{f_{1}}_{2} \norm{g_{1}}_{2} \Bigg(\int_{\R^2} \int_{\R^2} \range(y_{\rel})\,  | \log\norm{y_{\rel} - y'_{\rel}} \, |^2 \range(y'_{\rel}) \, \d y_{\rel} \, \d y'_{\rel}\Bigg)^{\half}  \nonumber \\
   	&  \lesssim \sqrt{\quadrange} \norm{f}_{2}\norm{g}_{2} \, \, \, ,
   \end{align} 
   and the limit
   \begin{align*}
   	\lim_{\varepsilon \to 0} \langle f,\cali{D}_{2}(\varepsilon) g \rangle 
   	& = \langle f, \Theta_{2} \otimes \fid{n-1} \, g \rangle \\ 
   	& =\int_{\widehat{\R}^{2\intv{n}\setminus\rel}_{\alpha}}\int_{\R^2}\int_{\R^2} \F{f_{2}}(\eta_{\com \cup \intv{n} \setminus \alpha})f_{1}(y_{\rel})\Theta_{2}(y_{\rel}, y'_{\rel})g_{1}(y'_{\rel}) \conj{\F{g_{2}}(\eta_{\com \cup \intv{n} \setminus \alpha})}\d y_{\rel} \d y'_{\rel} \d \eta_{\com \cup \intv{n} \setminus \alpha}  ,
   \end{align*}
   where
   \begin{equation}{\label{eqn:diagonal_theta2}}
    \Theta_{2}(y_{\rel},y'_{\rel}) \coloneqq   \sqrrange(y_{\rel})\, \Bigg[\frac{1}{(2\pi)^2}\int_{\hattorus{2}}\frac{e^{\im \eta_{\rel} \cdot (y_{\rel} - y'_{\rel})}-1}{\fouriergen(\mzaleta{\alpha})} \, \d \eta_{\rel}\Bigg] \, \sqrrange(y'_{\rel}) \, \, ,
   \end{equation} 
   and $\fid{n-1}$ is the identity operator on $\Ltwo(\widehat{\R}^{2\intv{n}\setminus\rel}_{\alpha})$. Technically, this operator depends on the chosen $\alpha$ but we shall ignore this dependence as this does not play a serious role in our proof.
   
   Next, we proceed to show that the operator $\cali{D}_{2}(\varepsilon, \eta_{\com \cup \intv{n} \setminus \alpha})$ converges to $\Theta_{2}$ in norm on the space $\Ltwo(\R^{2})$. To this end, express the difference of the kernels $\cali{D}_{2}(\varepsilon, y_{\rel}, y'_{\rel},\eta_{\com \cup \intv{n} \setminus \alpha})$ and $\Theta_{2}(y_{\rel}, y'_{\rel})$ as $\sqrrange(y_{\rel})\sqrrange(y'_{\rel})$ times
   \begin{equation*}
   	\int_{\metorus{4,\rel}(\eta_{\com})}\Bigg[\frac{e^{\im \eta_{\rel} \cdot (y_{\rel} - y'_{\rel})} -1}{\varepsilon^2 \lambda + \fouriergen(\mealeta{\alpha})} -\frac{e^{\im \eta_{\rel} \cdot (y_{\rel} - y'_{\rel})}-1}{\fouriergen(\mzaleta{\alpha})}\Bigg]\d \eta_{\rel} + \int \mathbf{1}_{\hattorus{2}\setminus\metorus{4,\rel}(\eta_{\com})}\Bigg[\frac{e^{\im \eta_{\rel} \cdot (y_{\rel} - y'_{\rel})}-1}{\fouriergen(\mzaleta{\alpha})} \Bigg] \d \eta_{\rel} ,
   \end{equation*}
   and let $\langle f_{1},\cali{D}_{2,1}(\varepsilon, \eta_{\com \cup \intv{n} \setminus \alpha}) g_{1} \rangle$ and $\langle f_{1},\cali{D}_{2,2}(\varepsilon, \eta_{\com \cup \intv{n} \setminus \alpha}) g_{1} \rangle$ denote the inner products of the above kernels tested against $f_{1},g_{1}$, respectively. We need to show that these vanish uniformly with respect to $f_{1}, g_{1}$ as $\varepsilon \to 0$. Let us begin with the former. First, observe that, due to the basic identity $\cos(x+y) + \cos(x- y) = 2\cos x \cos y$, we have
   \begin{align*}
   	|\fouriergen(\mealeta{\alpha}) 
   	&- \fouriergen(\mzaleta{\alpha}) | \\
   	& \leqslant \sum_{m_1 \in \Z^2}\Big|\cos\Big(\Big(\frac{\varepsilon\eta_{\com}}{2} + \eta_{\rel}\Big) \cdot m_1\Big)  + \cos\Big(\Big(\frac{\varepsilon\eta_{\com}}{2} - \eta_{\rel}\Big) \cdot m_1\Big) - 2\cos(\eta_{\rel} \cdot m_1 ) \Big| \nu(m_1) \\
   	& \qquad \qquad \qquad \qquad \qquad+ \sum_{k=3}^{n} \sum_{m_k \in \Z^2}|1 - \cos(\eta_k \cdot m_k)|\nu(m_k)  \\
   	& \leqslant \sum_{m_1 \in \Z^2}2|\cos(\eta_{\rel})| |1 - \cos(\varepsilon \eta_{\com} \cdot m_1)| \nu(m_1) + \sum_{k=3}^{n}\sum_{m_k \in \Z^2}|1 - \cos(\eta_k \cdot m_k)|\nu(m_k) \\
   	& \lesssim \varepsilon^2 \norm{\eta_{\com \cup \intv{n} \setminus \alpha}}^2 \, .
   \end{align*}

   Thus, substituting this in, we get the bound:
   \begin{align}{\label{eqn:prediagonal_D21_bound}}
   	\Bigg|\frac{e^{\im \eta_{\rel} \cdot (y_{\rel} - y'_{\rel})} -1}{\varepsilon^2 \lambda + \fouriergen(\mealeta{\alpha})}  
    -\frac{e^{\im \eta_{\rel} \cdot (y_{\rel} - y'_{\rel})}-1}{\fouriergen(\mzaleta{\alpha})}\Bigg| 
   	 \lesssim \frac{(\norm{\eta_{\rel}}\norm{y_{\rel} - y'_{\rel}}\wedge 2)(\varepsilon^2\lambda + \varepsilon^2\norm{\eta_{\com \cup \intv{n} \setminus \alpha}}^2)}{\norm{\eta_{\rel}}^2(\norm{\eta_{\rel}}^2 + \varepsilon^2\norm{\eta_{\com \cup \intv{n} \setminus \alpha}}^2 + \varepsilon^2\lambda)}  \, .
   \end{align}
   The integral of this bound with respect to $\d\eta_{\rel}$ on the ball $\ballorigin{R}$ can be evaluated as before by splitting the integral according to which term attains the minimum. First, assume $\norm{y_{\rel} - y'_{\rel}} \geqslant 2R^{-1}$. In this region, if $\norm{\eta_{\rel}} \leqslant 2\norm{y_{\rel}-y'_{\rel}}^{-1}$, the minimum attained is $\norm{\eta_{\rel}}\norm{y_{\rel}-y'_{\rel}}$. Hence, in this case, the integral becomes
   \begin{align}{\label{eqn:prediagonal_D2_bound1}}
   	2\pi\norm{y_{\rel} - y'_{\rel}}
   	&(\varepsilon^2\lambda + \varepsilon^2\norm{\eta_{\com \cup \intv{n} \setminus \alpha}}^2)\int_{0}^{2\norm{y_{\rel}-y'_{\rel}}^{-1}}\frac{\d r}{\varepsilon^2\lambda +\varepsilon^2\norm{\eta_{\com\cup\intv{n}\setminus\alpha}}^2 + r^2} \nonumber \\
   	& = 2\pi\norm{y_{\rel} - y'_{\rel}}(\varepsilon^2\lambda + \varepsilon^2\norm{\eta_{\com \cup \intv{n} \setminus \alpha}}^2)^{\half}\tan^{-1}\Bigg(\frac{2\norm{y_{\rel}-y'_{\rel}}^{-1}}{(\varepsilon^2\lambda + \varepsilon^2\norm{\eta_{\com \cup \intv{n} \setminus \alpha}}^2)^{\half}}\Bigg) \, \, .
   \end{align} 
   Since $\tan^{-1}x \lesssim x$ and $\varepsilon\norm{\eta_{\com \cup \intv{n} \setminus \alpha}} \lesssim \pi$ in $\metorus{4,\com} \times \epsalhattorus{2\intv{n}\setminus\alpha}$, the right-hand side is uniformly bounded with respect to the $y_{\rel}, y'_{\rel}$ and $\eta_{\com \cup \intv{n} \setminus \alpha}$ variables. Unfortunately, this bound is too crude for it to vanish uniformly. For our purpose, we require a better bound, preferably something logarithmic because of our integrability condition on $\range$. To this end, consider the function $h(x) = x\tan^{-1}(1/x) - C \log(1 + x), \, x >0$, where $C$ is a positive constant yet to be determined. Clearly, $\lim_{x \to \infty}h(x) = -\infty$ and $h'(x)$ is equal to 
   	\begin{equation*}
   		\tan^{-1}\bigg(\frac{1}{x}\bigg) - \frac{x}{x^2 + 1} - \frac{C}{1+x} \, \, .
   	\end{equation*}
   	For large $x$, we can do a Taylor series expansion so that this derivative is approximated by
   	\begin{equation*}
   		\frac{1}{x} - \frac{1}{3x^3} + \frac{1}{5x^5} - \frac{1}{x} + \frac{1}{x^3} - \frac{1}{x^5} - \frac{C}{1 + x} \, \, .
   	\end{equation*}
   	As $x \to \infty$, the $1/x$ and $1/(1+x)$ terms dominate from which we can infer the existence of a large enough $C$ such that $h'(x) < 0$. This then implies $h(x) \leqslant 0$ for large values of $x$ which is equivalent to the condition $ x\tan^{-1}(1/x) \lesssim C\log(1+x)$ for $x$ in this range. We stress that this is purely heuristic and a rigorous proof of this bound is possible. Using this functional inequality, the bound ultimately becomes
   	\begin{equation}
   		\log\Big(1 + \half(\varepsilon^2\lambda + \varepsilon^2\norm{\eta_{\com \cup \intv{n} \setminus \alpha}}^2)^{\half}\norm{y_{\rel}-y'_{\rel}}\Big) \, \, ,
   	\end{equation}
   	modulo a constant. 

   Later, when we finally compile all the norm estimates to bound $\prediag{\theta}{\alpha}(\lambda)$, we will restrict our $\lambda$ to the range $\lambda  \geqslant c$ for some constant $c$ depending on $\theta$ and $\quadrange$ only, and $\varepsilon$ to $\varepsilon \lesssim 1/\sqrt{\lambda}$. In particular, these restrictions imply that $\varepsilon^2 \lambda \lesssim 1$ for all our estimates. Thus, the expression above is $\lesssim 1$, uniformly with respect to $ \varepsilon$ and $\lambda$. 
   
   Next, we consider the region $\{ 2\norm{y_{\rel} - y'_{\rel}}^{-1} < \norm{\eta_{\rel}} \leqslant R\}$, where the term $\norm{\eta_{\rel}}\norm{y_{\rel} - y'_{\rel}} \wedge 2$ equals $2$. In this case, the integral of the bound with respect to $\eta_{\rel}$ is, in polar coordinates, equal to
   	\begin{align*}
   		\int_{\norm{y_{\rel} - y'_{\rel}}^{-1}}^{R}\frac{4\pi r(\varepsilon^2 \lambda + \varepsilon^2 \norm{\eta_{\com \cup \intv{n} \setminus \alpha}}^2 )  \d r}{r^2(r^2 + \varepsilon^2\lambda+ \varepsilon^2 \norm{\eta_{\com \cup \intv{n} \setminus \alpha}}^2)}
   		& = 4\pi\int_{4\norm{y_{\rel} - y'_{\rel}}^{-2}}^{R^2}\Bigg( \frac{1}{u} - \frac{1}{u + \varepsilon^2\lambda + \varepsilon^2 \norm{\eta_{\com \cup \intv{n} \setminus \alpha}}^2}  \Bigg) \, \d u  \\
   		& = 4\pi\log\Bigg(\frac{4R^2 + (\varepsilon^2\lambda + \varepsilon^2\norm{\eta_{\com \cup \intv{n} \setminus \alpha}}^2)\norm{y_{\rel} - y_{\rel}}^2}{4R^2 + 4\varepsilon^2\lambda + 4\varepsilon^2 \norm{\eta_{\com \cup \intv{n} \setminus \alpha}}^2}  \Bigg) \\
   		& \lesssim 4\pi\log\Big(1 + \half(\varepsilon^2\lambda + \varepsilon^2\norm{\eta_{\com \cup \intv{n} \setminus \alpha}}^2)^{\half}\norm{y_{\rel}-y'_{\rel}}\Big) \, .
   	\end{align*} 
   	
   Finally, we move to  $\{\norm{y_{\rel} - y'_{\rel}} < 2R^{-1}\}$. As reasoned previously, the only relevant $\eta_{\rel}$-region here is the set $\{\norm{\eta_{\rel}} < R\}$, on which the minimum is $\norm{\eta_{\rel}}\norm{y_{\rel} - y'_{\rel}}$. A simple calculation shows that the integral then evaluates to 
   \begin{equation*}
   	2\pi\norm{y_{\rel} - y'_{\rel}}(\varepsilon^2\lambda + \varepsilon^2\norm{\eta_{\com \cup \intv{n} \setminus \alpha}}^2 )^{\half}\tan^{-1}\Bigg(\frac{R}{(\varepsilon^2\lambda + \varepsilon^2\norm{\eta_{\com \cup \intv{n} \setminus \alpha}}^2)^{\half}}\Bigg) \, .
   \end{equation*}
   Because $\norm{y_{\rel} - y'_{\rel}}R \leqslant 2$, applying the earlier derived inequality $x \tan^{-1}(1/x) \lesssim \log(1+x)$ shows that the preceding expression is bounded by
   \begin{equation*}
  4\pi\log\Big( 1 + R^{-1}(\varepsilon^2\lambda + \varepsilon^2\norm{\eta_{\com \cup \intv{n} \setminus \alpha}}^2)^{\half}\Big)   	\, .
   \end{equation*}
  
   Thus, combining these bounds and using the Cauchy--Schwarz inequality in the $y_{\rel}, y'_{\rel}$ variables, we find that the inner product $|\langle f_{1} , \cali{D}_{2,1}(\varepsilon, \eta_{\com \cup \intv{n} \setminus \alpha}) g_{1} \rangle|$ is, up to a constant, at most
   \begin{align*}
   	\Bigg(\int_{\R^2}\int_{\R^2}\range(y_{\rel})\range(y'_{\rel})  \Bigg[
   	\log 
   	&\Big(1 +\half(\varepsilon^2\lambda + \varepsilon^2\norm{\eta_{\com \cup \intv{n} \setminus \alpha}}^2)^{\half}\norm{y_{\rel}-y'_{\rel}}\Big)  \\
   	 & + \log\Big( 1 + R^{-1}(\varepsilon^2\lambda + \varepsilon^2\norm{\eta_{\com \cup \intv{n} \setminus \alpha}}^2)^{\half}\Big)  \Bigg]^{2}  \d y_{\rel}  \d y'_{\rel}\Bigg)^{\half} \norm{f_{1}}_{2} \norm{g_{1}}_{2}  .
   \end{align*}
   In lieu of our integrability assumption on the function $\range$, we can use the dominated convergence theorem to interchange the integrals and limits whence we obtain
   \begin{equation*}
   	\lim_{\varepsilon \to 0} \sup_{\norm{f_{1}}_{2}, \norm{g_{1}}_{2} \leqslant 1} |\langle f_{1} , \cali{D}_{2,1}(\varepsilon, \eta_{\com \cup \intv{n} \setminus \alpha}) g_{1} \rangle| = 0 \, \, ,
   \end{equation*}
   which proves the required norm convergence of $\cali{D}_{2,1}(\varepsilon, \eta_{\com \cup \intv{n} \setminus \alpha})$ for fixed $\eta_{\com\cup\intv{n}\setminus\alpha}$.
   
   We move to the term $\langle f_{1} , \cali{D}_{2,2}(\varepsilon, \eta_{\com \cup \intv{n} \setminus \alpha}) g_{1} \rangle $. Estimating this is much easier because the $ \eta_{\rel}$ integrand is non-singular and continuous in the region under consideration. Furthermore, it is uniformly bounded in $y_{\rel}$ and $y'_{\rel}$. We state the bound directly:
   \begin{equation}
   	 | \langle f_{1}, \cali{D}_{2,2}(\varepsilon, \eta_{\com \cup \intv{n} \setminus \alpha}) g_{1} \rangle | \lesssim \norm{f_{1}}_{2} \norm{g_{1}}_{2} \, \big|\hattorus{2}\setminus\metorus{4,\rel}(\eta_{\com})\big| \, \, .
   \end{equation}
   From this, it is clear that for fixed $\eta_{\com \cup \intv{n} \setminus \alpha}$, the inner product vanishes as $\varepsilon$ goes to $0$, uniformly with respect to $f_{1}, g_{1}$. Thus, combining the results obtained so far, we find that
   \begin{equation*}
   	 \lim_{\varepsilon \to 0}\normop{\cali{D}_{2}(\varepsilon, \eta_{\com \cup \intv{n} \setminus \alpha}) - \Theta_{2}} = \lim_{\varepsilon \to 0} \sup_{\norm{f_{1}}_{2}, \norm{g_{1}}_{2} \leqslant 1} | \langle f_{1}, \big(\cali{D}_{2}(\varepsilon, \eta_{\com \cup \intv{n} \setminus \alpha}) - \Theta_{2} \big) g_{1} \rangle | = 0 \, \, .
   \end{equation*}
   In addition, since $\norm{\varepsilon \eta_{\com \cup \intv{n} \setminus \alpha}} \lesssim \pi$ and $\varepsilon^2\lambda \lesssim 1$, all the bounds obtained above are uniformly bounded in the $\eta_{\com \cup \intv{n} \setminus \alpha}$ variable, from which we conclude that the operator $\cali{D}_{2}(\varepsilon)$ converges to $\Theta_{2} \otimes \fid{n-1}$ strongly on the space $\Ltwo(\R^2 \times \widehat{\R}^{2\intv{n}\setminus\rel}_{\alpha})$.   
  
   Let us analyze some properties of the function $\Theta_{2}(y_{\rel}, y'_{\rel})$ in equation~\eqref{eqn:diagonal_theta2}. To evaluate the integral in the middle, note that, as the domain is symmetric, only the real part contributes. Write $\hattorus{2}$ as the union of $\ballorigin{\pi}$ and $\hattorus{2}\setminus\ballorigin{\pi}$. On $\ballorigin{\pi}$, we make a switch to polar coordinates and write the integral as:
   \begin{equation*}
   \int_{\ballorigin{\pi}}\frac{\cos(\eta_{\rel}\cdot (y_{\rel}-y'_{\rel}))-1}{\norm{\eta_{\rel}}^2} \, \d \eta_{\rel}  = \int_{0}^{2\pi} \int_{0}^{\pi} \frac{\cos(r\norm{y_{\rel}-y'_{\rel}}\cos\varphi)-1}{r} \, \d r \, \d \varphi \, \, .
   \end{equation*}
   The right-hand side above can be re-expressed in terms of $J_{0}$, the Bessel function of the first kind of order $0$. In fact, it has the following integral representation~\cite{gradshteyn2007}:
   \begin{equation*}
   	 J_{0}(z) = \frac{1}{2\pi}\int_{0}^{2\pi} \cos(z \cos\varphi) \, \d \varphi \, \, .
   \end{equation*}
   Substituting this expression and then performing the change of variables $r \mapsto r \norm{y_{\rel}-y'_{\rel}}$, the integral becomes:  
   \begin{equation*}
   	2\pi \int_{0}^{\pi\norm{y_{\rel} - y'_{\rel}}}  \frac{J_{0}(r)-1}{r} \, \d r \, \, .
   \end{equation*}
   First, we consider the case when $\norm{y_{\rel} - y'_{\rel}} \ll 1$ . In this case, it is known, again from \cite{gradshteyn2007}, that $J_{0}(r) \sim 1  - \frac{r^2}{4}$ and so,
   \begin{equation*}
   2\pi \int_{0}^{\pi\norm{y_{\rel}-y'_{\rel}}} \frac{J_{0}(r)-1}{r}  \, \d r \sim -\frac{\pi}{2} \int_{0}^{\pi\norm{y_{\rel}-y'_{\rel}}}  r \, \d r = -\frac{\pi^3}{4} \norm{y_{\rel}-y'_{\rel}}^2 \, .
   \end{equation*}
   For large $\norm{y_{\rel}-y'_{\rel}}$, split the interval of integration to $(0,1)$ and $(1, \pi \norm{y_{\rel}-y'_{\rel}})$. The integral over $(0,1)$ is finite for the same reason as before. For the second one, we use the fact that for large arguments, $|J_{0}(r)| \lesssim r^{-\half}$. Thus,
   \begin{equation*}
   2\pi\int_{1}^{\pi\norm{y_{\rel}-y'_{\rel}}} \frac{J_{0}(r)-1}{r} \, \d r = 2\pi \int_{1}^{\pi\norm{y_{\rel}-y'_{\rel}}} \frac{J_{0}(r)}{r} \, \d r - 2\pi\int_{1}^{\pi\norm{y_{\rel}-y'_{\rel}}} \frac{\d r}{r} \,  \, .
   \end{equation*}
   By the aforementioned tail bound of $J_{0}(z)$, the first term above is bounded uniformly with respect to $\norm{y_{\rel}-y'_{\rel}}$. The second term just evaluates to $2\pi \log(\pi \norm{y_{\rel}-y'_{\rel}})$. In the region $\hattorus{2} \setminus\ballorigin{\pi}$, the integrand has no singularities and is uniformly bounded with respect to $\norm{y_{\rel} - y'_{\rel}}$. To summarize, we have just shown that 
   \begin{equation}{\label{eqn:diagonal_theta2_bound}}
   \Theta_{2}(y_{\rel},y'_{\rel}) \lesssim
   \begin{cases}
   1 -\norm{y_{\rel}-y'_{\rel}}^2 & \text{ if } \norm{y_{\rel}-y'_{\rel}} \text{ is small } \\
   1 + \log\norm{y_{\rel}-y'_{\rel}} & \text{ if } \norm{y_{\rel}-y'_{\rel}} \text{ is large } \qquad .
   \end{cases}
   \end{equation}
   
   Now we turn to the operator $\cali{D}_{3}(\varepsilon)$. Consider $\langle f_{1} ,\cali{D}_{3}(\varepsilon,\eta_{\com \cup \intv{n} \setminus \alpha}) g_{1} \rangle$, which is given by the integral:
   \begin{align*}
   	 \int_{\R^2} 
   	 & \int_{\R^2} f_{1}(y_{\rel})\sqrrange(y_{\rel})\sqrrange(y'_{\rel}) g_{1}(y'_{\rel})  \\
   	 & \times  \frac{1}{(2\pi)^{2}}\int_{\metorus{4,\rel}(\eta_{\com})} \Bigg[\frac{1}{\varepsilon^2 \lambda + \fouriergen(\mealeta{\alpha})} - \frac{1}{\varepsilon^2 \lambda + \quart \norm{\varepsilon\eta_{\com}}^2 + \half \norm{\varepsilon\eta_{\intv{n}\setminus\alpha}}^2 + \norm{\eta_{\rel}}^2} \Bigg] \, \d \eta_{\rel} \, \d y_{\rel} \, \d y'_{\rel} \, \, .
   \end{align*}
   Since the term in the square brackets does not depend on $y_{\rel}, y'_{\rel}$, the operator is just a constant multiple of the projection operator $\proj: \Ltwo(\R^2) \to \Ltwo(\R^2)$ onto the subspace generated by $\sqrrange$, introduced earlier.
   This implies that to prove norm boundedness and convergence of $\cali{D}_{3}(\varepsilon, \eta_{\com \cup \intv{n} \setminus \alpha})$, it suffices to do the same for its coefficient, the $\eta_{\rel}$ integral. We do this now. Similar to before, choose an $R>0$ such that $ \metorus{4,\rel}(\eta_{\com}) \subset  \ballorigin{R}$. Inside this ball, using Assumptions~\ref{assumption2},\ref{assumption3} and Lemma~\ref{lem:uniformgreensbound}, the symbol $\fouriergen(\mealeta{\alpha})$ can be expanded around the origin as:
   \begin{equation*}
   	\fouriergen(\mealeta{\alpha}) = \norm{\eta_{\rel}}^2 + \quart \norm{\varepsilon\eta_{\com}}^2 + \half \norm{\varepsilon \eta_{\intv{n} \setminus \alpha}}^2 + r_{\varepsilon}(\eta) \, ,
   \end{equation*} 
   where the remainder term satisfies $|r_{\varepsilon}(\eta)| \lesssim \norm{\eta_{\rel}}^{2+\delta} +\varepsilon^{2+\delta}\norm{\eta_{\com \cup\intv{n} \setminus \alpha}}^{2+\delta}$. Using this expansion, we can bound the integrand the following way:
   \begin{equation}\label{eqn:diagonal_error_three}
   	\Bigg|\frac{1}{\varepsilon^2\lambda + \fouriergen(\mealeta{\alpha})} - \frac{1}{\varepsilon^2\lambda + \norm{\eta_{\rel}}^2 + \quart\norm{\varepsilon\eta_{\com}}^2 + \half\norm{\varepsilon\eta_{\intv{n}\setminus\alpha}}^2 } \Bigg| 
    \lesssim \frac{\norm{\eta_{\rel}}^{2+\delta} + \varepsilon^{2+\delta}\norm{\eta_{\com \cup \intv{n}\setminus\alpha}}^{2+\delta}}{(\norm{\eta_{\rel}}^2 + \varepsilon^2\norm{\eta_{\com \cup \intv{n} \setminus \alpha}}^2)^2} \, . 
   \end{equation}
   In order to show the integrability of the right hand side, decompose it as 
   \begin{equation*}
   	\int_{\ballorigin{R}}\frac{\norm{\eta_{\rel}}^{2+\delta}}{(\varepsilon^2\norm{\eta_{\com \cup \intv{n} \setminus \alpha}}^2 + \norm{\eta_{\rel}}^2)^2}\, \d\eta_{\rel} + \int_{\ballorigin{R}} \frac{\varepsilon^{2+\delta} \norm{\eta_{\com \cup \intv{n}\setminus\alpha}}^{2+\delta}}{(\varepsilon^2\norm{\eta_{\com \cup \intv{n} \setminus \alpha}}^2 + \norm{\eta_{\rel}}^2)^2} \, \d\eta_{\rel}\, .
   \end{equation*}
   After moving to polar coordinates, the above become
   \begin{equation*}
   	2\pi \int_{0}^{R} \frac{r^{3+\delta}}{(\varepsilon^2\norm{\eta_{\com \cup \intv{n} \setminus \alpha}}^2 + r^2)^2} \, \d r \quad \text{ and } \quad  2\pi \varepsilon^{2+\delta} \norm{\eta_{\com \cup\intv{n}\setminus \alpha}}^{2+\delta}\int_{0}^{R}\frac{r}{(\varepsilon^2\norm{\eta_{\com \cup \intv{n} \setminus \alpha}}^2 + r^2)^2} \, \d r \, ,
   \end{equation*}
   respectively. For the first term, we can use the inequality $(\varepsilon^2\norm{\eta_{\com \cup \intv{n} \setminus \alpha}}^2 + r^2)^2 \geqslant r^4$ to bound the integrand by $r^{3+\delta}r^{-4} = r^{\delta -1}$, which is integrable in the interval. For the second, make a substitution with $u= \varepsilon^2 \norm{\eta_{\com \cup \intv{n} \setminus \alpha}}^2 + r^2$ and so $\half \d u = r \d r$. The limits then change from $(0,R)$ to $(\varepsilon^2\norm{\eta_{\com \cup \intv{n} \setminus \alpha}}^2, \varepsilon^2\norm{\eta_{\com \cup \intv{n} \setminus \alpha}}^2 + R^2)$. Hence,
   \begin{align*}
    \pi \varepsilon^{2+\delta} \norm{\eta_{\com \cup\intv{n}\setminus \alpha}}^{2+\delta}\int_{\varepsilon^2\norm{\eta_{\com\cup\intv{n}\setminus\alpha}}^2}^{\varepsilon^2\norm{\eta_{\com \cup \intv{n}\setminus\alpha}}^2 + R^2} \frac{\d u}{u^2} 
   	& =\frac{R^2 \varepsilon^{2+\delta}\norm{\eta_{\com \cup \intv{n} \setminus \alpha}}^{2+\delta}}{(\varepsilon^{2}\norm{\eta_{\com \cup \intv{n} \setminus \alpha}}^{2} + R^2)\varepsilon^{2}\norm{\eta_{\com \cup \intv{n} \setminus \alpha}}^{2}} \\
   	& \leqslant \varepsilon^{\delta}\norm{\eta_{\com \cup \intv{n} \setminus \alpha}}^{\delta} \, \, .
   \end{align*}
   Since $\norm{\eta_{\com \cup \intv{n} \setminus \alpha}} \lesssim \varepsilon^{-1}$, the bounds obtained above are uniform in the $\eta_{\com \cup \intv{n} \setminus \alpha}$ variable.  Therefore, using dominated convergence theorem, we get the limit
   \begin{align*}
   	\lim_{\varepsilon \to 0} \frac{1}{(2\pi)^{2}}\int_{\metorus{4,\rel}(\eta_{\com})}\Bigg[\frac{1}{\varepsilon^2\lambda + \fouriergen(\mealeta{\alpha})} - 
   	& \frac{1}{\varepsilon^2\lambda + \norm{\eta_{\rel}}^2 + \quart \norm{\varepsilon \eta_{\com}}^2 + \half \norm{\varepsilon \eta_{\intv{n}\setminus \alpha}}^2}\Bigg] \, \d\eta_{\rel} \\
   	& = \frac{1}{(2\pi)^{2}}\int_{\hattorus{2}}\Bigg[\frac{1}{\fouriergen(\mzaleta{\alpha})} -\frac{1}{\norm{\eta_{\rel}}^2}\Bigg] \, \d\eta_{\rel} < \infty \, .
   \end{align*}
   Moreover, due to the uniform bounds, we obtain the norm bound
   \begin{equation}{\label{eqn:prediagonal_D3_bound}}
   	|\langle f ,\cali{D}_{3}(\varepsilon) g \rangle| \lesssim  \norm{\range}_{1} \norm{f}_{2}\norm{g}_{2} \, \, .
   \end{equation}
   Hence, 
   \begin{equation}{\label{eqn:diagonal_theta3}}
   	\lim_{\varepsilon \to 0}\langle f ,\cali{D}_{3}(\varepsilon) g \rangle =  \langle f, \theta_{3} \proj \otimes \fid{n-1} \, g \rangle \, \, ,
   \end{equation}
   where
   \begin{equation*}
   	\theta_{3} \coloneqq \frac{1}{(2\pi)^2}\int_{\hattorus{2}} \Bigg(\frac{1}{\fouriergen(\mzaleta{\alpha})} - \frac{1}{\norm{\eta_{\rel}}^2}\Bigg) \, \d\eta_{\rel}\, \, .
   \end{equation*}
   This proves that the operator $\cali{D}_{3}(\varepsilon)$ converges strongly to $\theta_{3}\proj \otimes \fid{n-1}$ on $\Ltwo(\R^2 \times \widehat{\R}^{2\intv{n}\setminus\rel}_{\alpha})$.
   
   We next analyze $\cali{D}_{4}(\varepsilon)$. Here also, the kernel is a multiple of the projection $\proj$ and so it suffices to study its coefficient. It is clear from the definition of $\cali{D}_{4}(\varepsilon, y_{\rel}, y'_{\rel}, \eta_{\com \cup \intv{n} \setminus \alpha})$ that the kernel is bounded above in $\norm{\eta_{\rel}}^{-2}$, which is uniformly bounded in $\eta_{\com \cup \intv{n} \setminus \alpha}$ and integrable in the region $\metorus{4,\rel}(\eta_{\com})\setminus\ballorigin{\pi}$. Thus, an application of the Cauchy--Schwarz inequality yields the bound
   \begin{equation}{\label{eqn:prediagonal_D4_bound}}
  |\langle f ,\cali{D}_{4}(\varepsilon) g \rangle| \lesssim  \norm{\range}_{1} \norm{f}_{2}\norm{g}_{2} \, \, , 
   \end{equation}
   and the limit
   \begin{equation}{\label{eqn:diagonal_theta4}}
   	\lim_{\varepsilon \to 0}\langle f,\cali{D}_{4}(\varepsilon) g \rangle =  \langle f, \theta_{4}\proj \otimes \fid{n-1} \, g\rangle \, \, ,
   \end{equation}
   where
   \begin{equation*}
   	 \theta_{4} \coloneqq   \frac{1}{(2\pi)^2}\int_{\hattorus{2}\setminus\ballorigin{\pi}} \frac{1}{\norm{\eta_{\rel}}^2}\, \d\eta_{\rel}  \, .
   \end{equation*}
   Hence, $\cali{D}_{4}(\varepsilon)$ converges to $\theta_{4}\proj \otimes \fid{n-1}$ in the strong operator topology.
   
   Let us evaluate the constant $\theta_{4}$. Moving to polar coordinates $(r, \varphi)$, observe that due to symmetry, it suffices to evaluate the integral in the region $0 \leqslant \varphi \leqslant \pi/4$ and then multiply the result by $8$. In this region, we compute the integral as:
  \begin{equation*}
  \frac{1}{(2\pi)^{2}} \int_{0}^{\frac{\pi}{4}}  \int_{\pi}^{\pi\sec\varphi} \frac{1}{r} \, \d r\,  \d \varphi =  \frac{1}{(2\pi)^{2}} \int_{0}^{\frac{\pi}{4}} \log\sec\varphi \, \d \varphi = -\frac{1}{(2\pi)^{2}} \int_{0}^{\frac{\pi}{4}} \log\cos\varphi \, \d \varphi \, .
  \end{equation*}
  The integral on the right hand side above is equal to $-\tfrac{\pi}{4}\log 2 + \tfrac{1}{2}G$, where $G$ is Catalan's constant, as mentioned in \cite[p.~531]{gradshteyn2007}. Consequently, after multiplication by $8$, we obtain:
  \begin{equation*}
  \theta_{4} = \frac{1}{(2\pi)^2}\int_{\hattorus{2}\setminus\ballorigin{\pi}} \frac{1}{\norm{\eta_{\rel}}^2} \, \d\eta_{\rel} = \frac{\log 2}{2\pi} - \frac{G}{\pi^2} \, .
  \end{equation*}
   
   Finally, we come to the operator $\cali{D}_{5}(\varepsilon)$, which constitutes the divergence in the operator $\preoffdiag{\alpha\alpha}(\lambda)$. For the $\d \eta_{\rel}$ integral, first, transfer to polar coordinates and then do a substitution with $u = r^2 + \varepsilon^2\lambda + \quart\norm{\varepsilon\eta_{\com}}^2 + \half\norm{\varepsilon\eta_{\intv{n}\setminus\alpha}}^2$ so that $\d u = 2r\, \d r$. The integral then evaluates to:
   \begin{align*}
    \int_{0}^{\pi}\frac{(4\pi)^{-1}2r \, \d r}{\varepsilon^2\lambda + \quart\norm{\varepsilon\eta_{\com}}^2 + \half\norm{\varepsilon\eta_{\intv{n}\setminus\alpha}}^2 + r^2}  
   & = \frac{1}{4\pi}\log\Bigg(\frac{\varepsilon^2\lambda + \quart\norm{\varepsilon\eta_{\com}}^2 + \half\norm{\varepsilon\eta_{\intv{n}\setminus\alpha}}^2 + \pi^2}{\varepsilon^2\lambda + \quart\norm{\varepsilon\eta_{\com}}^2 + \half\norm{\varepsilon\eta_{\intv{n}\setminus\alpha}}^2 }\Bigg) \\
   & = \frac{1}{4\pi}\log\Bigg(\frac{(\lambda + \quart\norm{\eta_{\com}}^2 + \half\norm{\eta_{\intv{n} \setminus \alpha}}^2)\varepsilon^2 + \pi^2}{\lambda + \quart\norm{\eta_{\com}}^2 + \half\norm{\eta_{\intv{n} \setminus \alpha}}^2}\Bigg) - \frac{\log\varepsilon}{2\pi}  \\
   & = \frac{|\log\varepsilon|}{2\pi} + \frac{1}{4\pi}\log\Bigg(\varepsilon^2 + \frac{\pi^2}{\lambda + \quart\norm{\eta_{\com}}^2 + \half\norm{\eta_{\intv{n} \setminus \alpha}}^2}\Bigg) \, \, . 
   \end{align*}
   Thus,
   \begin{equation}{\label{eqn:prediagonal_D5_bound}}
   | \langle f ,\cali{D}_{5}(\varepsilon) g \rangle| \lesssim  \Bigg(\frac{|\log\varepsilon|}{2\pi} + \frac{1}{4\pi}\log\bigg(\varepsilon^2 + \frac{\pi^2}{\lambda}\bigg)\Bigg) \norm{f}_{2} \norm{g}_{2}\, \, .
   \end{equation}
   From the definition of $\beta_{\varepsilon}$, as $\varepsilon \to 0$, its inverse can be expressed as
   \begin{equation*}
   	\beta_{\varepsilon}^{-1} = \frac{|\log\varepsilon|}{2\pi}\Bigg(1 + \frac{\theta}{|\log\varepsilon|}\Bigg)^{-1} = \frac{|\log\varepsilon|}{2\pi} - \frac{\theta}{2\pi} + o(1) \, \, .
   \end{equation*}
   With the above, it is easily verified that for fixed $\eta_{\com \cup \intv{n} \setminus \alpha}$,
   \begin{align}{\label{eqn:diagonal_theta5}}
      \lim_{\varepsilon \to 0} \big[\beta_{\varepsilon}^{-1}
      &\langle f_{1}, \proj g_{1} \rangle
      - \langle f_{1} ,\cali{D}_{5}(\varepsilon,\eta_{\com \cup \intv{n} \setminus \alpha}) g_{1} \rangle\big] \nonumber \\
     & =  \lim_{\varepsilon \to 0} \bigg\langle f_{1},\Bigg[ \beta_{\varepsilon}^{-1}\proj - \frac{1}{(2\pi)^{2}} \int_{\ballorigin{\pi}} 
     \frac{\d\eta_{\rel}}{\varepsilon^2 \lambda + \quart \norm{\varepsilon\eta_{\com}}^2 + \half \norm{\varepsilon\eta_{\intv{n}\setminus\alpha}}^2 + \norm{\eta_{\rel}}^2} \Bigg] g_{1} \bigg\rangle \nonumber  \\  
    &  =  \bigg\langle f_{1}, \Bigg[\frac{\log\big(\lambda + \quart\norm{\eta_{\com}}^2 + \half\norm{\eta_{\intv{n}\setminus\alpha}}^2\big) - 2\log\pi - 2\theta}{4\pi}\Bigg] \proj g_{1} \bigg\rangle \, . 
   \end{align} 
   We will use this fact in our convergence result.
   
   It remains to analyze the lattice correction term $\cali{D}_{1}(\varepsilon)$ in equation~\eqref{eqn:prediagonal_error_one}. After a change of variables $x = \Seps{\alpha}y$ and $x' = \Seps{\alpha}y'$, it becomes
  \begin{equation*}
  	\int_{\R^{2\intv{n}}} \int_{\R^{2\intv{n}}} F_{\varepsilon\alpha}(x) \big[\conj{\epsgreens(\lambda, \epsfloor{x} - \epsfloor{x'})} - \conj{\epscontgreens(\lambda, x - x')}\big] G_{\varepsilon\alpha}(x') \, \d x \, \d x' \, \, ,
  \end{equation*}
  where $F_{\varepsilon\alpha}(x) \coloneqq \varepsilon^{-2}(\sqrrange f)(\Teps{\alpha}x)$ and $G_{\varepsilon\alpha}(x') \coloneqq \varepsilon^{-2}(\sqrrange g)(\Teps{\alpha}x')$. The presence of the fractional parts makes estimating this term less conducive to the Fourier methods used before. We thus follow a more direct approach, like we did for the terms $\cali{I}(\varepsilon, f, g)$ and $\cali{O}(\varepsilon, f, g)$ in the previous sections. That is, we decompose the integral into a sum over the lattice $\Z_{\varepsilon}^{2\intv{n}}$, together with an integral over the box $\epsrect{2\intv{n}} = [0,\varepsilon)^{2\intv{n}}$ around each lattice point, and then estimate these using the Fourier analysis techniques.
 	Inserting the representation
 	\begin{equation*}
 		\epsgreens(\lambda, \epsfloor{x} - \epsfloor{x'}) - \epscontgreens(\lambda, x - x')
 		= \frac{1}{(2\pi)^{2n}}\int_{\epshattorus{2\intv{n}}}\,\Bigg[ \frac{e^{-\im \xi \cdot (\epsfloor{x} - \epsfloor{x'})} - e^{-\im \xi \cdot (x - x')}}{\lambda + \fourierepsgen(\xi)}\Bigg] \,  \d \xi \, ,
 	\end{equation*}
 	into the integral, we can write it as 
 	\begin{align}{\label{eqn:diagonal_error_zero}}
 		\frac{1}{(2\pi)^{2n}}
 		& \int_{\epshattorus{2\intv{n}}}\int_{\R^{2\intv{n}}}\int_{\R^{2\intv{n}}} F_{\varepsilon\alpha}(x)\Bigg[ \frac{e^{\im \xi \cdot (\epsfloor{x} - \epsfloor{x'})} - e^{\im \xi \cdot (x - x')}}{\lambda + \fourierepsgen(\xi)}\Bigg] G_{\varepsilon\alpha}(x') \, \d x \, \d x' \, \d \xi \nonumber  \\
 		& = \frac{1}{(2\pi)^{2n}}\int_{\epshattorus{2\intv{n}}}\sum_{p \in \Z_{\varepsilon}^{2\intv{n}}}\sum_{p' \in \Z_{\varepsilon}^{2\intv{n}}} \int_{\epsrect{2\intv{n}}}\int_{\epsrect{2\intv{n}}} \nonumber  \\ 
 		& \qquad \qquad \times  F_{\varepsilon\alpha}(p+u)\Bigg[ \frac{e^{\im \xi \cdot (p - p')} - e^{\im \xi \cdot (p - p')}e^{\im \xi \cdot (u - u')}}{\lambda + \fourierepsgen(\xi)}\Bigg] G_{\varepsilon\alpha}(p'+u') \, \d u  \, \d u' \, \d \xi \, \, .
 	\end{align}
 	Now, we revert back to the center-of-mass and relative coordinate frames using the change of variables $\xi = \Tdualeps{\alpha}\eta$,  $p = \Seps{\alpha}q, \, p' = \Seps{\alpha}q', \, u = \Seps{\alpha}v, \, u' = \Seps{\alpha}v' $. This transformation cancels the $\varepsilon^{-4}$ in the definitions of $F_{\varepsilon\alpha}$ and $G_{\varepsilon\alpha}$ and changes $\lambda, \fourierepsgen(\xi)$ to $\varepsilon^2\lambda$ and $\fouriergen(\mealeta{\alpha})$ respectively. Furthermore, the summations now run over $\Z_{\varepsilon\alpha}^{2\intv{n}}$, $\epsrect{2\intv{n}}$ changes to $\epsalrect{2\intv{n}} = \Teps{\alpha}\epsrect{2\intv{n}}$ and the product $\xi \cdot (u-u')$ becomes $\eta \cdot (v- v')$.
 	Note that as $\varepsilon \to 0$, the $v_{\rel}$ region in $\epsalrect{2\intv{n}}$ approaches the box $(-1, 1)^2$. In particular, it stays bounded, independently of $\varepsilon$. For $y= q+ v$ and $y'=q'+v'$, define the kernel $	\cali{D}_{1}(\varepsilon, y, y', \eta_{\com \cup \intv{n} \setminus \alpha})$ as
 	\begin{equation*}
 	 \sqrrange(y_{\rel})\Bigg[\frac{1}{(2\pi)^2}\int_{\metorus{4,\rel}(\eta_{\com})} \frac{e^{\im \eta_{\rel}\cdot (q_{\rel}- q'_{\rel})} - e^{\im \eta_{\rel} \cdot (y_{\rel} - y'_{\rel})  + \im \eta_{\com \cup \intv{n} \setminus \alpha} \cdot (v_{\com \cup \intv{n}\setminus\alpha}-v'_{\com \cup \intv{n}\setminus\alpha})} }{\varepsilon^2 \lambda + \fouriergen(\mealeta{\alpha})} \, \d \eta_{\rel}\Bigg]\sqrrange(y'_{\rel}) \, \, .
 	\end{equation*}
 	and let 
 	\begin{align*}
 		\langle f_{1}, \cali{D}_{1}(\varepsilon, \eta_{\com \cup \intv{n} \setminus \alpha}) g_{1} \rangle \coloneqq\sum_{q_{\rel}} \sum_{q'_{\rel}} \int \int f_{1}(q_{\rel} + v_{\rel})\cali{D}_{1}(\varepsilon, y, y', \eta_{\com \cup \intv{n} \setminus \alpha})g_{1}(q'_{\rel} + v'_{\rel}) \, \d v_{\rel} \, \d v'_{\rel}  \, .
 	\end{align*}
 	This kernel of course depends on $v_{\com\cup\intv{n}\setminus\alpha}$ and $v'_{\com\cup\intv{n}\setminus\alpha}$ as well, but we omit this to reduce the notational burden. We make a further simplification by defining
 	\begin{equation*}
 		\widetilde{f_{2}}(v_{\com \cup \intv{n}\setminus\alpha}, \eta_{\com \cup \intv{n} \setminus \alpha}) \coloneqq \frac{1}{(2\pi)^{n-1}}\sum_{q_{\com\cup\intv{n}\setminus\alpha}}f_{2}(q_{\com \cup \intv{n}\setminus\alpha} + v_{\com \cup \intv{n}\setminus\alpha})e^{\im \eta_{\com \cup \intv{n} \setminus \alpha} \cdot q_{\com \cup \intv{n}\setminus\alpha}} \, ,
 	\end{equation*}
 	and similarly $\widetilde{g_{2}}(v'_{\com \cup \intv{n}\setminus\alpha}, \eta_{\com \cup \intv{n} \setminus \alpha})$. With these notations, the summation in equation~\eqref{eqn:diagonal_error_zero} can be re-expressed in a more compact form:
 	\begin{align}{\label{eqn:diagonal_D1_discreteplancherel}}
 		\int_{\metorus{4,\com} \times \epsalhattorus{2\intv{n}\setminus\alpha}} \int_{\epsalrect{2\intv{n}\setminus\rel}} 
 		 \int_{\epsalrect{2\intv{n}\setminus\rel}}  
 		& \widetilde{f_{2}}(v_{\com \cup \intv{n}\setminus\alpha}, \eta_{\com \cup \intv{n} \setminus \alpha}) 
 		\conj{\widetilde{g_{2}}(v'_{\com \cup \intv{n}\setminus\alpha}, \eta_{\com \cup \intv{n} \setminus \alpha})} \nonumber \\
 		& \times \langle f_{1},\cali{D}_{1}(\varepsilon, \eta_{\com \cup \intv{n} \setminus \alpha})  g_{1}\rangle \, \d v_{\com \cup \intv{n}\setminus\alpha} \,  \d v'_{\com \cup \intv{n}\setminus\alpha} \, \d \eta_{\com \cup \intv{n}\setminus\alpha} \, .
 	\end{align}

 	First, we bound $\langle f_{1} ,\cali{D}_{1}(\varepsilon, \eta_{\com \cup \intv{n} \setminus \alpha}) g_{1} \rangle $. It will be convenient to analyze the behavior of the Fourier integrand as $\varepsilon \to 0$ separately for $\eta_{\rel}$ and $\eta_{\com \cup \intv{n}\setminus\alpha}$. To this end, write the term as
 	\begin{equation}{\label{eqn:diagonal_D1_factorsplit}}
 	\frac{e^{\im \eta_{\rel} \cdot (q_{\rel} - q'_{\rel})}\big[1- e^{\im \eta_{\rel} \cdot (v_{\rel} - v'_{\rel})} \big]}{\varepsilon^2 \lambda + \fouriergen(\mealeta{\alpha})}
 	+ \frac{e^{\im \eta_{\rel} \cdot (q_{\rel} - q'_{\rel}) + \im \eta_{\rel} \cdot (v_{\rel} - v'_{\rel})}\big[ 1- e^{\im \eta_{\com \cup\intv{n}\setminus\alpha} \cdot (v_{\com\cup\intv{n}\setminus\alpha} - v'_{\com\cup\intv{n}\setminus\alpha}) }  \big]}{\varepsilon^2 \lambda + \fouriergen(\mealeta{\alpha})} \, \, .
 	\end{equation}
 	Using this split, decompose $\cali{D}_{1}(\varepsilon, \eta_{\com \cup \intv{n} \setminus \alpha})$ as $\cali{D}_{1, 1}(\varepsilon, \eta_{\com \cup \intv{n} \setminus \alpha}) + \cali{D}_{1, 2}(\varepsilon, \eta_{\com \cup \intv{n} \setminus \alpha})$, where the term $\cali{D}_{1,1}(\varepsilon, \eta_{\com \cup \intv{n} \setminus \alpha})$ corresponds to the first term in equation~\eqref{eqn:diagonal_D1_factorsplit}, and $\cali{D}_{1, 2}(\varepsilon, \eta_{\com \cup \intv{n} \setminus \alpha})$ to the other.    
 	Let us first look at $\cali{D}_{1, 1}(\varepsilon, \eta_{\com \cup \intv{n} \setminus \alpha})$. As $\varepsilon \to 0$, the divergence in the $\d \eta_{\rel}$ integral occurs near the origin. Similar to what we did before, we choose an $R>0$ such that $\metorus{4,\rel}(\eta_{\com}) \subset \ballorigin{R}$ to isolate this singularity. Inside this ball, we use the inequality: 
 	\begin{equation*}
 	\Bigg| \frac{e^{\im \eta_{\rel} \cdot (q_{\rel} - q'_{\rel})}\big[1- e^{\im \eta_{\rel} \cdot (v_{\rel} - v'_{\rel})} \big]}{\varepsilon^2 \lambda + \fouriergen(\mealeta{\alpha})}\Bigg| \lesssim \frac{\norm{v_{\rel} - v'_{\rel}}\norm{\eta_{\rel}}}{\varepsilon^2 \lambda + \norm{\eta_{\rel}}^2} \lesssim \frac{1}{\norm{\eta_{\rel}}} \, \, .
 	\end{equation*} 
 	The integral of the right-hand side above over the region $\ballorigin{R}$ evaluates to $2\pi R$. As for $\cali{D}_{1, 2}(\varepsilon, \eta_{\com \cup \intv{n} \setminus \alpha})$, we use a similar bound;
 	\begin{equation*}
 	\Bigg|\frac{e^{\im \eta_{\rel} \cdot (q_{\rel} - q'_{\rel}) + \im \eta_{\rel} \cdot (v_{\rel} - v'_{\rel})}\big[ 1- e^{\im \eta_{\com \cup\intv{n}\setminus\alpha} \cdot (v_{\com\cup\intv{n}\setminus\alpha} - v'_{\com\cup\intv{n}\setminus\alpha}) }  \big]}{\varepsilon^2 \lambda + \fouriergen(\mealeta{\alpha})} \Bigg| \lesssim \frac{\varepsilon \norm{\eta_{\com \cup \intv{n}\setminus\alpha}}}{\norm{\eta_{\rel}}^2 + \varepsilon^2\norm{\eta_{\com \cup \intv{n} \setminus \alpha}}^2} \, \, .
 	\end{equation*}
 	 We can compute the integral of the right-hand side over $\ballorigin{R}$ in polar coordinates as:
 	\begin{equation*}
 	\int_{0}^{R} \frac{2\pi \varepsilon\norm{\eta_{\com \cup \intv{n} \setminus \alpha}}\, r \,\d r}{r^2 + \varepsilon^2\norm{\eta_{\com \cup \intv{n} \setminus \alpha}}^2} = \pi \varepsilon\norm{\eta_{\com \cup \intv{n} \setminus \alpha}}\big[\log(R^2 + \varepsilon^2\norm{\eta_{\com \cup \intv{n} \setminus \alpha}}^2) - 2\log(\varepsilon\norm{\eta_{\com \cup \intv{n} \setminus \alpha}})\big] \, \,.
 	\end{equation*}
 	Since $\varepsilon\norm{\eta_{\com \cup \intv{n} \setminus \alpha}} \lesssim \pi$ in $\metorus{4,\com} \times \epsalhattorus{2\intv{n}\setminus\alpha}$ and the function $x \mapsto x \log x$ is bounded on compact intervals, the above obtained bounds are uniform in $\eta_{\com \cup \intv{n} \setminus \alpha}$. Thus, using the Cauchy--Schwarz inequality for sums, we get the norm bound
 	\begin{equation}{\label{eqn:diagonal_D1_bound}}
 	 \sup_{\eta_{\com \cup \intv{n} \setminus \alpha}}|\langle f_{1} ,\cali{D}_{1}(\varepsilon, \eta_{\com \cup \intv{n} \setminus \alpha})g_{1} \rangle | \lesssim \pi^2\big[\log(R^2 + \pi^2) + \log(\pi^2)\big] \norm{\range}_{1}\norm{f_{1}}_{2}\norm{g_{1}}_{2} \, .
 	\end{equation}
 	For equation~\eqref{eqn:diagonal_D1_discreteplancherel}, we can again use the Cauchy--Schwarz inequality in the $\eta_{\com \cup \intv{n} \setminus \alpha}$ variable in conjunction with the bound above. This bounds the same from above by the term
 	\begin{align*}
 	 \norm{\range}_{1} \norm{f_{1}}_{2} \, 
 	 &  \norm{g_{1}}_{2} \Bigg( \int_{\metorus{4,\com}\times \epsalhattorus{2\intv{n}\setminus\alpha}}  \int_{\epsalrect{2\intv{n}\setminus\rel}}  \varepsilon^{2n-2} \,  | \widetilde{f_{2}}(v_{\com \cup \intv{n}\setminus\alpha}, \eta_{\com \cup \intv{n} \setminus \alpha}) |^2 \, \d v_{\com\cup\intv{n}\setminus\alpha} \, \d \eta_{\com \cup \intv{n} \setminus \alpha}\Bigg)^{\half}  \\
 	 & \times \Bigg( \int_{\metorus{4,\com}\times \epsalhattorus{2\intv{n}\setminus\alpha}} \int_{\epsalrect{2\intv{n}\setminus\rel}}  \varepsilon^{2n-2} \, | \widetilde{g_{2}}(v'_{\com \cup \intv{n}\setminus\alpha}, \eta_{\com \cup \intv{n} \setminus \alpha}) |^2 \, \d v'_{\com\cup\intv{n}\setminus\alpha} \, \d \eta_{\com \cup \intv{n} \setminus \alpha}\Bigg)^{\half} .
 	\end{align*}
    By Plancherel's theorem for Fourier series, the terms in the parentheses are just $\Ltwo$ norms of $f_{2}$ and $g_{2}$ respectively, whence we obtain the bound
    \begin{equation}{\label{eqn:prediagonal_D1_bound}}
    	| \langle f, \cali{D}_{1}(\varepsilon) g \rangle | \lesssim  \norm{\range}_{1} \norm{f}_{2} \norm{g}_{2} \, \, .
    \end{equation}
 	Applying the dominated convergence theorem, we have that $\lim_{\varepsilon \to 0} \langle f , \cali{D}_{1}(\varepsilon) g \rangle$ is equal to
    \begin{equation*}
    	\int_{\R^2} \int_{\R^2} \int_{\widehat{\R}^{2\intv{n}\setminus\rel}_{\alpha}}f_{1}(y_{\rel})\F{f_{2}}(\eta_{\com \cup \intv{n} \setminus \alpha}) \Theta_{1}(y_{\rel}, y'_{\rel}) g_{1}(y'_{\rel})\conj{\F{g_{2}}(\eta_{\com \cup \intv{n} \setminus \alpha})} \, \d y_{\rel} \, \d y'_{\rel} \, \d \eta_{\com \cup \intv{n} \setminus \alpha} \, \, ,
    \end{equation*} 
 	where
 	\begin{equation}{\label{eqn:diagonal_theta1}}
 	\Theta_{1}(y_{\rel}, y'_{\rel}) \coloneqq \sqrrange(y_{\rel})\, \Bigg[\frac{1}{(2\pi)^2}\int_{\hattorus{2}} \frac{e^{\im \eta_{\rel} \cdot (\floor{y_{\rel}} - \floor{y'_{\rel}}) } -e^{\im \eta_{\rel} \cdot (y_{\rel} - y'_{\rel})} }{\fouriergen(\mzaleta{\alpha})} \, \d \eta_{\rel} \Bigg]\, \sqrrange(y'_{\rel}) \, \, .
 	\end{equation}
 	
 	Next, we show that  $\cali{D}_{1}(\varepsilon)$ converges strongly to the operator $\Theta_{1} \otimes \fid{n-1}$ on $\Ltwo(\R^2) \times \Ltwo(\widehat{\R}^{2\intv{n}\setminus\rel}_{\alpha})$. For this reason, it will be convenient to express the action of $\Theta_{1} \otimes \fid{n-1}$ in a way which resembles that in equation~\eqref{eqn:diagonal_D1_discreteplancherel}. Using the same notations defined above, the inner product $\langle f, \Theta_{1} \otimes \fid{n-1}  g \rangle$ can be equivalently expressed as
 	\begin{align*}
 		\int_{\widehat{\R}^{2\intv{n}\setminus\rel}_{\alpha}}
 		\int_{\epsalrect{2\intv{n}\setminus\rel}} &\int_{\epsalrect{2\intv{n}\setminus\rel}}
 		 \widetilde{f_{2}}(v_{\com \cup \intv{n}\setminus\alpha}, \eta_{\com \cup \intv{n} \setminus \alpha}) \\
 		& \qquad \times \langle f_{1} , \Theta_{1} g_{1}\rangle \, \conj{\widetilde{g_{2}}(v'_{\com \cup \intv{n}\setminus\alpha}, \eta_{\com \cup \intv{n} \setminus \alpha})} \, \d v_{\com\cup\intv{n}\setminus\alpha} \, \d v'_{\com\cup\intv{n}\setminus\alpha} \, \d \eta_{\com \cup \intv{n} \setminus \alpha} \, \, ,
 	\end{align*}
 	where $\langle f_{1} , \Theta_{1} \, g_{1} \rangle$ is equal to
 	\begin{align*}
 		\sum_{q_{\rel}}\sum_{q'_{\rel}}
 		 &\int_{\epsrect{2}} \int_{\epsrect{2}} \,
 		 f_{1}(q_{\rel}+v_{\rel})\sqrrange(q_{\rel} + v_{\rel})\sqrrange(q'_{\rel} + v'_{\rel})g_{1}(q'_{\rel}+v'_{\rel}) \\
 		& \times \Bigg[\frac{e^{\im \eta_{\com \cup \intv{n} \setminus \alpha} \cdot (v_{\com\cup\intv{n}\setminus\alpha} - v'_{\com\cup\intv{n}\setminus\alpha})}}{(2\pi)^2}\int_{\hattorus{2}} \frac{e^{\im \eta_{\rel} \cdot (q_{\rel} - q'_{\rel})}-e^{\im \eta_{\rel} \cdot (q_{\rel} - q'_{\rel}) + \im \eta_{\rel} \cdot (v_{\rel} - v'_{\rel})} }{\fouriergen(\mzaleta{\alpha})} \, \d \eta_{\rel}\Bigg]  \, \d v_{\rel} \, \d v'_{\rel} \, .
 	\end{align*}
 	The difference $\langle f , (\cali{D}_{1}(\varepsilon) - \Theta_{1} \otimes \fid{n-1}) \,g \rangle$ consists of terms with the kernels:
 	\begin{align*}
 		& \mathbf{1}_{\mealtorus{\alpha}{2\intv{n}}}(\eta) \, \Bigg[\frac{ e^{\im \eta_{\rel} \cdot (v_{\rel} - v'_{\rel})}\big(1 - e^{\im \eta_{\com} \cdot (v_{\com} - v'_{\com}) + \im \eta_{\intv{n}\setminus\alpha} \cdot (v_{\intv{n}\setminus\alpha} - v'_{\intv{n}\setminus\alpha}) } \, \big) }{\varepsilon^2 \lambda + \fouriergen(\mealeta{\alpha})} \Bigg] \\
 		& + \mathbf{1}_{\mealtorus{\alpha}{2\intv{n}}}(\eta) \, e^{\im \eta_{\com \cup \intv{n} \setminus \alpha} \cdot (v_{\com\cup\intv{n}\setminus\alpha} - v'_{\com\cup\intv{n}\setminus\alpha})}\Bigg[\frac{ e^{\im \eta_{\rel}\cdot (q_{\rel}- q'_{\rel})} - e^{\im \eta_{\rel} \cdot (y_{\rel} - y'_{\rel})}}{\varepsilon^2 \lambda + \fouriergen(\mealeta{\alpha})} - \frac{ e^{\im \eta_{\rel}\cdot (q_{\rel}- q'_{\rel})} - e^{\im \eta_{\rel} \cdot (y_{\rel} - y'_{\rel})}}{\fouriergen(\mzaleta{\alpha})} \Bigg] \\
 		&+ \mathbf{1}_{\metorus{4,\com} \times \epsalhattorus{2\intv{n}\setminus\alpha}}(\eta_{\com \cup \intv{n} \setminus \alpha})\mathbf{1}_{\hattorus{2}\setminus\metorus{4,\rel}(\eta_{\com})}(\eta_{\rel}) \,
 		e^{\im \eta_{\com \cup \intv{n} \setminus \alpha} \cdot (v_{\com\cup\intv{n}\setminus\alpha} - v'_{\com\cup\intv{n}\setminus\alpha})}\Bigg[\frac{e^{\im \eta_{\rel}\cdot (q_{\rel}- q'_{\rel})} - e^{\im \eta_{\rel} \cdot (y_{\rel} - y'_{\rel})}}{\fouriergen(\mzaleta{\alpha})}\Bigg] \\
 		& + [1 - \mathbf{1}_{\metorus{4,\com} \times \epsalhattorus{2\intv{n}\setminus\alpha}}(\eta_{\com \cup \intv{n} \setminus \alpha})]\Theta_{1}(y_{\rel}, y'_{\rel}) \fid{n-1}(\eta_{\com \cup \intv{n} \setminus \alpha}) \, \, .
 	\end{align*}
 	Label these terms $\langle f ,\cali{D}_{1,1}(\varepsilon) g\rangle , \langle f ,\cali{D}_{1,2}(\varepsilon) g\rangle, \langle f, \cali{D}_{1,3}(\varepsilon) g \rangle$ and $\langle f ,\cali{D}_{1,4}(\varepsilon) g\rangle $ respectively. We estimate each term individually. Let us start with $\langle f ,\cali{D}_{1,1}(\varepsilon) g\rangle$. For the  $\d \eta_{\rel}$ integral, using Lemma~\ref{lem:uniformgreensbound} and the fact that $\norm{v_{\com\cup\intv{n}\setminus\alpha} - v'_{\com\cup\intv{n}\setminus\alpha}} \lesssim \varepsilon$, we have the bound
 	\begin{equation*}
 		\Bigg|\frac{e^{\im\eta_{\rel}\cdot(v_{\rel} - v'_{\rel})}[e^{\im \eta_{\com \cup \intv{n} \setminus \alpha} \cdot (v_{\com \cup \intv{n}\setminus\alpha}-v'_{\com \cup \intv{n}\setminus\alpha})} - 1]}{\varepsilon^2 \lambda + \fouriergen(\mealeta{\alpha})} \Bigg| \lesssim \frac{\varepsilon \norm{\eta_{\com \cup \intv{n}\setminus\alpha}}}{\norm{\eta_{\rel}}^2 + \varepsilon^2\norm{\eta_{\com \cup \intv{n} \setminus \alpha}}^2} \, \, .
 	\end{equation*}
 		We evaluated the integral of this bound as part of our proof of the norm bound of $\cali{D}_{1}(\varepsilon)$. This yielded
 	\begin{equation*}
 		\int_{0}^{R} \frac{2\pi \varepsilon\norm{\eta_{\com \cup \intv{n} \setminus \alpha}}\, r \,\d r}{r^2 + \varepsilon^2\norm{\eta_{\com \cup \intv{n} \setminus \alpha}}^2} = \pi \varepsilon\norm{\eta_{\com \cup \intv{n} \setminus \alpha}}\big[\log(R^2 + \varepsilon^2\norm{\eta_{\com \cup \intv{n} \setminus \alpha}}^2) - 2\log(\varepsilon\norm{\eta_{\com \cup \intv{n} \setminus \alpha}})\big] \, \,.
 	\end{equation*}
 	Using this bound along with Cauchy--Schwarz inequality and Plancherel's theorem for Fourier series, $|\langle f, \cali{D}_{1,1}(\varepsilon) g \rangle|$ is bounded by $\norm{f}_{2}\norm{g_{1}}_{2}$ times 
 	\begin{align*}
 		\Bigg(\int_{\epsalrect{2\intv{n}\setminus\rel}} \int_{\metorus{4,\com} \times \epsalhattorus{2\intv{n}\setminus\alpha}} \varepsilon^2\norm{\eta_{\com \cup \intv{n} \setminus \alpha}}^2
 		 &\big|\log(R^2 + \varepsilon^2\norm{\eta_{\com \cup \intv{n} \setminus \alpha}}^2) + 2\log(\varepsilon\norm{\eta_{\com \cup \intv{n} \setminus \alpha}})\big|^2 \\
 		 &\times  |\widetilde{g_{2}}(v'_{\com \cup \intv{n}\setminus\alpha}, \eta_{\com \cup \intv{n} \setminus \alpha})|^2 \,  \d \eta_{\com \cup \intv{n} \setminus \alpha} \,  \d v'_{\com \cup\intv{n}\setminus\alpha}\Bigg)^{\half} \,  ,
 	\end{align*}
 	up to a constant. The function inside the integral goes to zero point-wise as $\varepsilon \to 0$ and is bounded by the $\Ltwo$ norm of $g_{2}$. This implies we can use the dominated convergence theorem to pass the limit inside. Then, taking a supremum over $f$, we obtain,
 	\begin{equation*}
 		\lim_{\varepsilon \to 0} \sup_{\norm{f}_{2} \leqslant 1}|\langle f, \cali{D}_{1,1}(\varepsilon) g \rangle|  = 0 \, \, .
 	\end{equation*} 
 	
 	We turn to $\langle f , \cali{D}_{1,2}(\varepsilon)g \rangle$. By Lemma~\ref{lem:uniformgreensbound},
 	\begin{align*}
 		\Bigg|\frac{e^{\im \eta_{\rel}\cdot (q_{\rel}- q'_{\rel})} - e^{\im \eta_{\rel} \cdot (y_{\rel} - y'_{\rel})}}{\varepsilon^2 \lambda + \fouriergen(\mealeta{\alpha})} 
 		& - \frac{e^{\im \eta_{\rel}\cdot (q_{\rel}- q'_{\rel})} - e^{\im \eta_{\rel} \cdot (y_{\rel} - y'_{\rel})}}{\fouriergen(\mzaleta{\alpha})} \Bigg| \\
 		& \lesssim \frac{(\norm{\eta_{\rel}}\norm{v_{\rel} - v'_{\rel}} \wedge 2)(\varepsilon^2\lambda +\varepsilon^2 \norm{\eta_{\com\cup\intv{n}\setminus\alpha}}^2}{(\norm{\eta_{\rel}}^2 + \varepsilon^2\lambda + \varepsilon^2\norm{\eta_{\com \cup \intv{n} \setminus \alpha}}^2) \norm{\eta_{\rel}}^2} \, .
 	\end{align*}
    This was the same bound we obtained in equation~\eqref{eqn:prediagonal_D21_bound} from our proof for $\cali{D}_{2}(\varepsilon, \eta_{\com \cup \intv{n} \setminus \alpha})$. Hence, we can bound this term in the same way. After this, we again use Cauchy--Schwarz inequality and Plancherel's theorem to conclude that $|\langle f, \cali{D}_{1,2}(\varepsilon) g \rangle|$ is $\lesssim \norm{f}_{2} \norm{g}_{2}$. Using dominated convergence theorem, we also obtain
 	\begin{equation*}
 		\lim_{\varepsilon \to 0} \sup_{\norm{f}_{2} \leqslant 1} | \langle f, \cali{D}_{1,2}(\varepsilon) g \rangle| = 0 \, \, .
 	\end{equation*}
 	
   The situation for $\langle f ,\cali{D}_{1,3}(\varepsilon) g \rangle$ is much simpler as the $\eta_{\rel}$ integral is devoid of any singularities in this region. Since it is also continuous, modulo a constant, it is bounded by $	\norm{\range}_{1}\norm{f}_{2}\norm{g}_{1}$ times 
   \begin{align*}
   \Bigg(\int_{\epsalrect{2\intv{n}\setminus\rel}} \int_{\metorus{4,\com} \times \epsalhattorus{2\intv{n}\setminus\alpha}}|\hattorus{2}\setminus\metorus{4,\rel}(\eta_{\com})| \,|\widetilde{g_{2}}(v'_{\com \cup \intv{n}\setminus\alpha}, \eta_{\com \cup \intv{n} \setminus \alpha})|^2 \,  \d \eta_{\com \cup \intv{n} \setminus \alpha} \,  \d v'_{\com \cup\intv{n}\setminus\alpha}\Bigg)^{\half} \,  .
   \end{align*}
   Again, an application of the dominated convergence theorem shows that
   \begin{equation*}
   	\lim_{\varepsilon \to 0} \sup_{\norm{f}_{2} \leqslant 1}|\langle f, \cali{D}_{1,3}(\varepsilon) g \rangle | = 0 \, \, .
   \end{equation*}
   Finally, we analyze $\langle f, \cali{D}_{1,4}(\varepsilon) g \rangle$. Since the operator $\Theta_{1} \otimes \fid{n-1}$ is bounded on $\Ltwo(\R^2) \times \Ltwo(\widehat{\R}^{2\intv{n}\setminus\rel}_{\alpha})$, a straightforward application of dominated convergence lets us conclude
   \begin{equation*}
   	\lim_{\varepsilon \to 0} \sup_{\norm{f}_{2} \leqslant 1} |\langle f, \cali{D}_{1,4}(\varepsilon) g \rangle| = 0 \, \, .
   \end{equation*}
   
   Putting together the above bounds, we obtain
   \begin{equation*}
   	\lim_{\varepsilon \to 0}\sup_{\norm{f}_{2}\leqslant 1} |\langle f, (\cali{D}_{1}(\varepsilon) - \Theta_{1} \otimes \fid{n-1} ) \,g \rangle | \lesssim \lim_{\varepsilon \to 0} \sup_{\norm{f}_{2} \leqslant 1} \sum_{r=1}^{4}|\langle f, \cali{D}_{1,r}(\varepsilon) g \rangle|= 0 \, \, ,
   \end{equation*}
   which proves that $\cali{D}_{1}(\varepsilon)$ converges strongly to $\Theta_{1} \otimes \fid{n-1}$.
   
   Now we are in a position to collect all the operator norm bounds we derived so far to obtain an estimate for $\preoffdiag{\alpha\alpha}(\lambda)$. Recall that 
   \begin{equation*}
   	\langle f, \preoffdiag{\alpha\alpha}(\lambda) g \rangle = \langle f, \cali{D}_{1}(\varepsilon)g \rangle + \langle f, \cali{D}_{2}(\varepsilon) g \rangle + \langle f, \cali{D}_{3}(\varepsilon) g \rangle + \langle f, \cali{D}_{4}(\varepsilon)g \rangle + \langle f, \cali{D}_{5}(\varepsilon) g \rangle \, \, .
   \end{equation*}
   Hence, using equations~\eqref{eqn:prediagonal_D1_bound},\eqref{eqn:prediagonal_D2_bound},\eqref{eqn:prediagonal_D3_bound},\eqref{eqn:prediagonal_D4_bound} and \eqref{eqn:prediagonal_D5_bound}, we get
   \begin{align*}
   	\normop{\beta_{\varepsilon}\preoffdiag{\alpha\alpha}(\lambda)} \lesssim \beta_{\varepsilon}\Bigg(\quadrange + 1 + \frac{|\log\varepsilon|}{2\pi} + \frac{1}{4\pi}\log\bigg(\varepsilon^2 + \frac{\pi^2}{\lambda}\bigg) \Bigg) \, \, .
   \end{align*}
   Fix $\lambda > 0$. For small enough $\varepsilon$, namely for $\varepsilon <  \pi/\sqrt{\lambda}$, the above norm is less than
   \begin{equation*}
   	\beta_{\varepsilon}\bigg(\quadrange + 1 + \frac{|\log\varepsilon|}{2\pi} + \frac{\log2\pi^2}{4\pi} - \frac{\log\lambda}{4\pi} \bigg) \, .
   \end{equation*}
   Using the definition of the $\beta_{\varepsilon}$, the right-hand side then becomes
   \begin{equation}{\label{eqn:prediagonal_normupperbound}}
   	   1 + \frac{\theta + 2\pi\quadrange + 2\pi + \log\sqrt{2}\pi - \log\sqrt{\lambda}}{|\log\varepsilon|} + \frac{\theta\big(2\pi\quadrange + 2\pi + \log\sqrt{2}\pi - \log\sqrt{\lambda}\,\big)}{|\log\varepsilon|^2} \, \, .
   \end{equation}
   As $\varepsilon \to 0$, the final term is of lower order than the second term. Consequently, if the $\lambda$ we originally chose was such that 
   \begin{equation*}
   	 \theta + 2\pi\quadrange + 2\pi + \log\sqrt{2}\pi - \log\sqrt{\lambda} < 0 ,
   \end{equation*}
   then norm bound in equation~\eqref{eqn:prediagonal_normupperbound} will be strictly less than one, uniformly for $\varepsilon < \varepsilon_{0}(\lambda)$ of order $1/\sqrt{\lambda}$. We therefore assume henceforth that $\lambda$ satisfies
   \begin{equation*}
   	 \lambda \gtrsim e^{2\theta + 4\pi\quadrange} + 1 \, \, \, .
   \end{equation*}
   Thus, for all $\lambda$ in this region, the Neumann series in equation~\eqref{eqn:prediagonal_neumannseries} converges absolutely in norm. That is,
   \begin{equation*}
   	\prediag{\theta}{\alpha}(\lambda)  = \beta_{\varepsilon} \sum_{m=0}^{\infty} \big( \beta_{\varepsilon} \preoffdiag{\alpha\alpha}(\lambda) \big)^m = \beta_{\varepsilon}(\id{n}- \beta_{\varepsilon} \preoffdiag{\alpha\alpha}(\lambda))^{-1} = (\beta_{\varepsilon}^{-1}\id{n} - \preoffdiag{\alpha\alpha}(\lambda))^{-1} \, .
   \end{equation*}
    Moreover, $\normop{\prediag{\theta}{\alpha}(\lambda)}$ is
   \begin{align}{\label{eqn:prediagonal_final_bound}}
   	&\lesssim \beta_{\varepsilon}\sum_{m =0}^{\infty} 
   	\Bigg(1 + \frac{ \theta + 2\pi(\quadrange + 1) + \log\sqrt{2}\pi - \log\sqrt{\lambda}}{|\log\varepsilon|} + \frac{\theta\big(2\pi(\quadrange + 1) + \log\sqrt{2}\pi - \log\sqrt{\lambda}\,\big)}{|\log\varepsilon|^2}\Bigg)^{m} \nonumber \\
   	& = \frac{\frac{2\pi}{|\log\varepsilon|}(1+ \frac{\theta}{|\log\varepsilon|})}{\frac{ \log\sqrt{\lambda} -\theta - 2\pi(\quadrange + 1) - \log\sqrt{2}\pi }{|\log\varepsilon|} + \frac{\theta (\log\sqrt{\lambda}-2\pi(\quadrange + 1) - \log\sqrt{2}\pi \,)}{|\log\varepsilon|^2}} \nonumber \\
   	 & \lesssim (\log\lambda - \theta)^{-1} \, \, .
   \end{align}
   This completes the proof of the norm bounds in Lemma~\ref{lem:resolventconvergence}.
   \subsubsection{Convergence to the limit}
    Now we are ready to prove the convergence of the prediagonal operator to its limiting counterpart. As we mentioned at the beginning of this section, we first fix $\eta_{\com \cup \intv{n} \setminus \alpha}$ and show the convergence in norm as an operator on $\Ltwo(\R^2)$ followed by strong convergence in the $\eta_{\com \cup \intv{n} \setminus \alpha}$ variable. However, unlike $\cali{D}_{2}(\varepsilon), \cali{D}_{3}(\varepsilon), \cali{D}_{4}(\varepsilon)$ and $\cali{D}_{5}(\varepsilon)$, the lattice correction operator $\cali{D}_{1}(\varepsilon)$ does not lend itself to this method easily. The kernels of the former operators treat the $y_{\rel}, y'_{\rel}$ and $\eta_{\com\cup\intv{n}\setminus\alpha}$ variables independently, whereas the latter has these coupled. In order to avoid this issue, we first omit $\cali{D}_{1}(\varepsilon)$ from $\preoffdiag{\alpha\alpha}(\lambda)$ and establish the convergence of the resulting operator. Then, we treat $\cali{D}_{1}(\varepsilon)$ as a perturbation of this operator, and using suitable resolvent identities, we deduce the strong convergence of the full prelimiting diagonal operator. Let us begin.  
    
    Observe $\preoffdiag{\alpha\alpha}(\lambda) - \cali{D}_{1}(\varepsilon) =\sum_{r=2}^{5}\cali{D}_{r}(\varepsilon)$. Fixing $\eta_{\com \cup \intv{n} \setminus \alpha} \in \metorus{4,\com} \times \epsalhattorus{2\intv{n}\setminus\alpha}$, we want to compute the limit of  
    \begin{equation*}
    \big(\beta_{\varepsilon}^{-1} \id{n} - \sum_{r=2}^{5}\cali{D}_{r}(\varepsilon, \eta_{\com \cup \intv{n} \setminus \alpha})\big)^{-1}
    \end{equation*} 
    as $\varepsilon \to 0$, on $\Ltwo(\R^2)$. strictly speaking, the operator $\id{n}$ should be replaced by restriction to the $y_{\rel}, y'_{\rel}$ variables since we are fixing $\eta_{\com \cup \intv{n} \setminus \alpha}$. We ignore this for now and assume it is implicit.
    Denote by $\projperp$ the orthogonal complement of the projection operator $\proj$. To pass to the limit, we use the operator identity
   \begin{equation}{\label{eqn:operator_identity}}
   	(A - B)^{-1} = \sum_{m=0}^{\infty}A^{-1} (BA^{-1})^{m} \, \, ,
   \end{equation}
   which holds whenever $\normop{BA^{-1}} < 1$. Since $\id{n} = \projperp + \proj$, we can write the operator under consideration as
   \begin{equation*}
   	\big(\, \beta_{\varepsilon}^{-1}\projperp + \beta_{\varepsilon}^{-1} \proj - \sum_{r=2}^{5}\cali{D}_{r}(\varepsilon, \eta_{\com \cup \intv{n} \setminus \alpha})\big)^{-1} \, \, .
   \end{equation*}
   Substitute this into the identity in equation \eqref{eqn:operator_identity} with 
  \begin{equation*}
  	A = \beta_{\varepsilon}^{-1}\projperp + (\beta_{\varepsilon}^{-1}\proj - \sum_{r=3}^{5}\cali{D}_{r}(\varepsilon, \eta_{\com \cup \intv{n} \setminus \alpha})) \quad \text{and} \quad 	B = \cali{D}_{2}(\varepsilon, \eta_{\com \cup \intv{n} \setminus \alpha}) \, \, .
  \end{equation*} 
  The reason for this choice is that the operators $\cali{D}_{r}(\varepsilon, \eta_{\com \cup \intv{n} \setminus \alpha})$ for $r=3,4,5$ are just multiples of the projection operator, whereas the kernel of $\cali{D}_{2}(\varepsilon, \eta_{\com \cup \intv{n} \setminus \alpha})$ depends explicitly on $y_{\rel}, y'_{\rel}$ other than through $\sqrrange$ also. This yields the series:
\begin{align}{\label{eqn:diagonal_neumann_R2}}
	\sum_{m=0}^{\infty} \bigg[\beta_{\varepsilon}^{-1}\projperp + 
	& \big(\beta_{\varepsilon}^{-1}\proj - \sum_{r=3}^{5}\cali{D}_{r}(\varepsilon, \eta_{\com \cup \intv{n} \setminus \alpha}) \, \big)\bigg]^{-1}  \nonumber \\
	& \times \bigg[ \cali{D}_{2}(\varepsilon, \eta_{\com \cup \intv{n} \setminus \alpha}) \, \big(\beta_{\varepsilon}^{-1}\projperp + \big(\beta_{\varepsilon}^{-1}\proj - \sum_{r=3}^{5}\cali{D}_{r}(\varepsilon, \eta_{\com \cup \intv{n} \setminus \alpha}) \, \big)\big)^{-1} \bigg]^{m} \, \, .
\end{align}
Since $\beta_{\varepsilon} \to 0$, from equations~\eqref{eqn:diagonal_theta3}, \eqref{eqn:diagonal_theta4} and \eqref{eqn:diagonal_theta5}, we see that
\begin{align*}
\lim_{\varepsilon \to 0}\big(\beta_{\varepsilon}^{-1}\projperp + \big(\beta_{\varepsilon}^{-1}
& \proj - \sum_{r=3}^{5}\cali{D}_{r}(\varepsilon, \eta_{\com \cup \intv{n} \setminus \alpha}) \, \big)\big)^{-1} \\ 
& =\frac{4\pi}{\log\big(\lambda + \quart\norm{\eta_{\com}}^2 + \half\norm{\eta_{\intv{n}\setminus\alpha}}^2\big) -2\log\pi - 2\theta - 4\pi\theta_{3} - 4\pi\theta_{4}} \, \proj \, ,
\end{align*}
in norm. The other terms in the series involve products of the operators $\cali{D}_{2}(\varepsilon, \eta_{\com \cup \intv{n} \setminus \alpha})$, $\proj$ and $\projperp$. The terms $\proj^m, (\projperp)^m$ and $\proj \projperp$ reduce to $\proj, \projperp$ and $0$ respectively, as they are orthogonal projections and hence idempotent. The product of $\projperp$ and $\cali{D}_{2}(\varepsilon, \eta_{\com \cup \intv{n} \setminus \alpha})$ has coefficients involving $\beta_{\varepsilon}$ and so it goes to $0$ since $\cali{D}_{2}(\varepsilon, \eta_{\com \cup \intv{n} \setminus \alpha})$ is uniformly bounded in norm as $\varepsilon \to 0$. Finally, for the term $\proj \cali{D}_{2}(\varepsilon, \eta_{\com \cup \intv{n} \setminus \alpha})\proj$, we may calculate it as 
\begin{align*}
	\langle f_{1}, \proj&\cali{D}_{2}(\varepsilon, \eta_{\com \cup \intv{n} \setminus \alpha})
	\proj \, g_{1} \rangle \\
	&= \frac{1}{(2\pi)^{2}}\int_{\R^2} \int_{\R^2}\int_{\R^2}\int_{\R^2}  \, \d y^{(1)}_{\rel} \, \d y^{(2)}_{\rel} \, \d y^{(3)}_{\rel} \, \d y^{(4)}_{\rel} \, f_{1}(y^{(1)}_{\rel})\sqrrange(y^{(1)}_{\rel})   \\
	& \qquad \qquad \qquad \qquad \times \sqrrange(y^{(2)}_{\rel})\cali{D}_{2}(\varepsilon, y^{(2)}_{\rel} ,y^{(3)}_{\rel}, \eta_{\com \cup \intv{n} \setminus \alpha})\sqrrange(y^{(3)}_{\rel}) \sqrrange(y^{(4)}_{\rel})g_{1}(y^{(4)}_{\rel}) \\
	& = \bigg(\frac{1}{(2\pi)^{2}}\int_{\R^2} \int_{\R^2} \sqrrange(y_{\rel})\cali{D}_{2}(\varepsilon, y_{\rel} ,y'_{\rel}, \eta_{\com \cup \intv{n} \setminus \alpha}) \sqrrange(y'_{\rel}) \, \d y_{\rel} \, \d y'_{\rel}\bigg) \langle f_{1}, \proj g_{1} \rangle \, .
\end{align*}
The term in the parentheses is finite due to the bounds on $\cali{D}_{2}(\varepsilon, y_{\rel}, y'_{\rel}, \eta_{\com \cup \intv{n} \setminus \alpha})$ derived earlier and our integrability assumption on $\range$. Let
\begin{align*}
	\theta_{2}  
	& \coloneqq \lim_{\varepsilon \to 0} \frac{1}{(2\pi)^{2}}\int_{\R^2} \int_{\R^2} \sqrrange(y_{\rel})\cali{D}_{2}(\varepsilon, y_{\rel} ,y'_{\rel}, \eta_{\com \cup \intv{n} \setminus \alpha}) \sqrrange(y'_{\rel}) \, \d y_{\rel} \, \d y'_{\rel} \\ 
	&= \int_{\R^2} \int_{\R^2} \sqrrange(y_{\rel}) \Theta_{2}(y_{\rel}, y'_{\rel}) \sqrrange(y'_{\rel}) \, \d y_{\rel} \, \d y'_{\rel}\, \, .
\end{align*}
The limit is of course independent of the fixed variable $\eta_{\com \cup \intv{n} \setminus \alpha}$. Thus, after taking limit $\varepsilon \to 0$ in equation~\eqref{eqn:diagonal_neumann_R2}, the operator $\big(\, \beta_{\varepsilon}^{-1}\id{n} - \sum_{r=2}^{5}\cali{D}_{r}(\varepsilon, \eta_{\com \cup \intv{n} \setminus \alpha})\big)^{-1}$ converges in norm to 
\begin{align*}
  & \frac{4\pi }{\log\big(\lambda + \quart\norm{\eta_{\com}}^2 + \half\norm{\eta_{\intv{n}\setminus\alpha}}^2\big) -2\log\pi -2\theta -4\pi\theta_{3} - 4\pi\theta_{4}}\proj \\
  & \qquad \qquad \qquad \quad \times \sum_{m=0}^{\infty}\frac{( 4\pi\theta_{2})^{m}}{\big( \log\big(\lambda + \quart\norm{\eta_{\com}}^2 + \half\norm{\eta_{\intv{n}\setminus\alpha}}^2\big) -2\log\pi -2\theta -4\pi\theta_{3} - 4\pi\theta_{4}\big)^{m}}\proj \\
 &  \qquad \qquad \qquad \quad =\frac{4\pi}{\log\big(\lambda + \quart\norm{\eta_{\com}}^2 + \half\norm{\eta_{\intv{n}\setminus\alpha}}^2\big) -2\log\pi -2\theta-4\pi(\theta_{2}+\theta_{3}+\theta_{4})}\proj \, .
\end{align*}
This is the kernel of $\F{\diag{\theta'}{\alpha}(\lambda)}(\eta_{\com \cup \intv{n} \setminus \alpha})$, with 
\begin{equation*}
	\theta' \coloneqq 2\theta + 4\pi\theta_{2} + 4\pi\theta_{3} + 4\pi\theta_{4}+2\log\pi \, \, .
\end{equation*} 
Thus, for fixed $\eta_{\com \cup \intv{n} \setminus \alpha}$, we have just shown that:
\begin{equation*}
\NOrmop{	\big(\, \beta_{\varepsilon}^{-1}\projperp + \beta_{\varepsilon}^{-1} \proj - \sum_{r=2}^{5}\cali{D}_{r}(\varepsilon, \eta_{\com \cup \intv{n} \setminus \alpha})\big)^{-1} - \F{\diag{\theta'}{\alpha}(\lambda)}(\eta_{\com \cup \intv{n} \setminus \alpha})} \to 0 \, \, .
\end{equation*}
Now let us show the strong convergence. Using Plancherel's theorem in the $y_{\com \cup \intv{n}\setminus\alpha}$ variables, the difference $\big|\langle f ,\big[	\big(\, \beta_{\varepsilon}^{-1}\projperp + \beta_{\varepsilon}^{-1} \proj - \sum_{r=2}^{5}\cali{D}_{r}(\varepsilon)\big)^{-1} - \diag{\theta'}{\alpha}(\lambda)\big] g \rangle \big| $ is
\begin{align*}
	&\lesssim
	\int_{\metorus{4,\com} \times \epsalhattorus{2\intv{n}\setminus\alpha}} \int_{\R^2} \int_{\R^2}
	|\F{f_{2}}(\eta_{\com \cup \intv{n} \setminus \alpha})|
	|\F{g_{2}}(\eta_{\com \cup \intv{n} \setminus \alpha})|
	|f_{1}(y_{\rel})|	|g_{1}(y'_{\rel})| \\
	& \times \big|\big( \beta_{\varepsilon}^{-1}\projperp + \beta_{\varepsilon}^{-1} \proj - \sum_{r=2}^{5}\cali{D}_{r}(\varepsilon,y_{\rel}, y'_{\rel}, \eta_{\com \cup \intv{n} \setminus \alpha})\big)^{-1} 
	- \F{\diag{\theta'}{\alpha}(\lambda)}(y_{\rel}, y'_{\rel}, \eta_{\com \cup \intv{n} \setminus \alpha})\big| \, \d y_{\rel}  \d y'_{\rel}  \d \eta_{\com \cup \intv{n} \setminus \alpha}
	\\
	& \qquad \qquad +\int_{\widehat{\R}^{2\intv{n}}_{\alpha}\setminus\metorus{4,\com} \times \epsalhattorus{2\intv{n}\setminus\alpha}}
	|\F{f_{2}}(\eta_{\com \cup \intv{n} \setminus \alpha})|
	|\F{g_{2}}(\eta_{\com \cup \intv{n} \setminus \alpha})|
  \\
	&\qquad \qquad \qquad \qquad \qquad \times  \int_{\R^2} \int_{\R^2}
	|f_{1}(y_{\rel})|
	|\F{\diag{\theta'}{\alpha}(\lambda)} (y_{\rel}, y'_{\rel}, \eta_{\com \cup \intv{n} \setminus \alpha})|
	|g_{1}(y'_{\rel})| \, \d y_{\rel} \, \d y'_{\rel} 	\, \d \eta_{\com \cup \intv{n} \setminus \alpha} \,  .      
\end{align*}
By the previously proven result, the $\d y_{\rel} \, \d y'_{\rel}$ integral in the first term above is bounded by $\normop{\big(\, \beta_{\varepsilon}^{-1}\projperp + \beta_{\varepsilon}^{-1} \proj - \sum_{r=2}^{5}\cali{D}_{r}(\varepsilon, \eta_{\com \cup \intv{n} \setminus \alpha})\big)^{-1} - \F{\diag{\theta'}{\alpha}(\lambda)}(\eta_{\com \cup \intv{n} \setminus \alpha})} \, \norm{f_{1}}_{2}\norm{g_{1}}_{2}$, which vanishes as $\varepsilon \to 0$ for each $\eta_{\com \cup \intv{n} \setminus \alpha}$. Therefore, we can use Cauchy--Schwarz inequality in the $\eta_{\com \cup \intv{n} \setminus \alpha}$ variables to bound the above integrals by
\begin{multline}
 \norm{f}_{2}\norm{g_{1}}_{2}
  \times \Bigg(\int_{\metorus{4,\com} \times \epsalhattorus{2\intv{n}\setminus\alpha}}\NOrmop{\big(\, \beta_{\varepsilon}^{-1}\projperp + \beta_{\varepsilon}^{-1} \proj - \sum_{r=2}^{5}\cali{D}_{r}(\varepsilon, \eta_{\com \cup \intv{n} \setminus \alpha})\big)^{-1} - \F{\diag{\theta'}{\alpha}(\lambda)}(\eta_{\com \cup \intv{n} \setminus \alpha})}^2 \\
  \qquad \qquad \qquad \qquad \qquad \qquad \times |\F{g_{2}}(\eta_{\com \cup \intv{n} \setminus \alpha})|^2 \d \eta_{\com \cup \intv{n} \setminus \alpha} \Bigg)^{\half} \\
  + \norm{f}_{2}\norm{g_{1}}_{2} \times\Bigg(\int_{\R^{2\intv{n}}\setminus\metorus{4,\com} \times \epsalhattorus{2\intv{n}\setminus\alpha}}\normop{\F{\diag{\theta'}{\alpha}(\lambda)}(\eta_{\com \cup \intv{n} \setminus \alpha})}^2 \, |\F{g_{2}}(\eta_{\com \cup \intv{n} \setminus \alpha})|^2 \Bigg)^{\half} \, \, .
\end{multline}
From equations~\eqref{eqn:prediagonal_final_bound} and \eqref{eqn:diagonal_partial_fourier_two}, the norms $\normop{\big(\, \beta_{\varepsilon}^{-1}\projperp + \beta_{\varepsilon}^{-1} \proj - \sum_{r=2}^{5}\cali{D}_{r}(\varepsilon, \eta_{\com \cup \intv{n} \setminus \alpha})\big)^{-1}}$ and $\normop{\F{\diag{\theta'}{\alpha}(\lambda)}(\eta_{\com \cup \intv{n} \setminus \alpha})}$ are uniformly bounded with respect to $\eta_{\com \cup \intv{n} \setminus \alpha}$ and hence, the $\eta_{\com \cup \intv{n} \setminus \alpha}$ integrand in both cases is bounded by $|\F{g_{2}}(\eta_{\com \cup \intv{n} \setminus \alpha})|^2$, which is integrable. Thus, by the dominated convergence theorem, 
\begin{equation*}
	\lim_{\varepsilon \to 0} \sup_{\norm{f}_{2} \leqslant 1}	\big|\langle f ,\big[	\big(\, \beta_{\varepsilon}^{-1}\projperp + \beta_{\varepsilon}^{-1} \proj - \sum_{r=2}^{5}\cali{D}_{r}(\varepsilon)\big)^{-1} - \diag{\theta'}{\alpha}(\lambda)\big] g \rangle \big|   = 0 \, \, ,
\end{equation*}
which yields the desired strong convergence result. 

It remains to show the convergence of the full prelimiting diagonal operator, which includes the lattice correction term $\cali{D}_{1}(\varepsilon)$. Since 
\[
\preoffdiag{\alpha\alpha}(\lambda) = \cali{D}_{1}(\varepsilon) + \sum_{r=2}^{5}\cali{D}_{r}(\varepsilon) \, \, ,
\]
by a resolvent identity, $	\prediag{\theta}{\alpha}(\lambda) = (\beta_{\varepsilon}^{-1}\id{n} - \preoffdiag{\alpha\alpha}(\lambda))^{-1}$ satisfies the equation
\begin{equation*}
	  \big(\beta_{\varepsilon}^{-1}\id{n} - \sum_{r=2}^{5}\cali{D}_{r}(\varepsilon)\big)^{-1} + \big(\beta_{\varepsilon}^{-1}\id{n} - \preoffdiag{\alpha\alpha}(\lambda)\big)^{-1} \, \cali{D}_{1}(\varepsilon)\, \big(\beta_{\varepsilon}^{-1}\id{n} - \sum_{r=2}^{5}\cali{D}_{r}(\varepsilon)\big)^{-1} \, \, .
\end{equation*}
Rearranging the above equation, we obtain
\begin{equation*}
	\prediag{\theta}{\alpha}(\lambda)  =\big(\beta_{\varepsilon}^{-1}\id{n} - \textstyle\sum_{r=2}^{5}\cali{D}_{r}(\varepsilon)\big)^{-1} \, \Big[ \,\id{n} - \cali{D}_{1}(\varepsilon) \big(\beta_{\varepsilon}^{-1}\id{n} - \textstyle\sum_{r=2}^{5}\cali{D}_{r}(\varepsilon) \, \, \big)^{-1} \,\Big]^{-1} \, \, .
\end{equation*}
Since 
\begin{equation*}
	\sup_{\varepsilon > 0} \normop{\cali{D}_{1}(\varepsilon)} \, \NOrmop{\big(\beta_{\varepsilon}^{-1}\id{n} - \sum_{r=2}^{5}\cali{D}_{r}(\varepsilon)\big)^{-1}} \lesssim (\log\lambda - \theta)^{-1} \ll 1 \, \, ,
\end{equation*}
the operator above is uniformly bounded in $\varepsilon > 0$. Thus, we can take a term by term limit in the strong operator topology to obtain
\begin{equation*}
	\lim_{\varepsilon \to 0}\prediag{\theta}{\alpha}(\lambda) = \diag{\theta'}{\alpha}(\lambda) \big(\id{n} - (\Theta_{1} \otimes \fid{n-1} )\, \diag{\theta'}{\alpha}(\lambda)\big)^{-1} \, \, .
\end{equation*}
To determine what this operator is, fix $\eta_{\com \cup \intv{n} \setminus \alpha} \in \widehat{\R}^{2\intv{n}\setminus\rel}_{\alpha}$. For this choice, the restriction of the limiting operator above on $\Ltwo(\R^2)$ takes the series form:
\begin{equation*}
	\F{\diag{\theta'}{\alpha}(\lambda)} 
	 (\eta_{\com \cup \intv{n} \setminus \alpha})\big(\id{n}  - \Theta_{1}	\F{\diag{\theta'}{\alpha}(\lambda)}(\eta_{\com \cup \intv{n} \setminus \alpha})\big)^{-1} = \sum_{m=0}^{\infty}	\F{\diag{\theta'}{\alpha}(\lambda)}(\eta_{\com \cup \intv{n} \setminus \alpha}) \big[ \Theta_{1}	\F{\diag{\theta'}{\alpha}(\lambda)}(\eta_{\com \cup \intv{n} \setminus \alpha}) \big]^{m}  .
\end{equation*}
From equation~\eqref{eqn:diagonal_projection_form}, recall
\begin{equation*}
	\F{\diag{\theta'}{\alpha}(\lambda)}(\eta_{\com \cup \intv{n} \setminus \alpha}) = 4\pi\big(\log\big( \lambda +\quart \norm{\eta_{\com}}^2 + \half\norm{\eta_{\intv{n}\setminus\alpha}}^2\big) - \theta'\big)^{-1}\proj. 
\end{equation*}
Inserting this into the series above, it becomes
\begin{align*}
\frac{4\pi}{\big(\log\big( \lambda +\quart \norm{\eta_{\com}}^2 + \half\norm{\eta_{\intv{n}\setminus\alpha}}^2\big) - \theta'\big)}	\sum_{m=0}^{\infty}\frac{(4\pi)^{m}\,\proj [\Theta_{1}\proj]^{m}}{ \big(\log\big( \lambda +\quart \norm{\eta_{\com}}^2 + \half\norm{\eta_{\intv{n}\setminus\alpha}}^2\big) - \theta'\big)^{m}} \, \, .
\end{align*}
Because $\proj$ is a projection, $\proj [\Theta_{1}\proj]^{m} = [\proj \Theta_{1} \proj]^{m}$ for $m \geqslant 1$. Similar to how we calculated $\theta_{2}$, we infer that 
\begin{equation*}
	\proj \Theta_{1} \proj = \theta_{1} \proj \, \, ,
\end{equation*}
where 
\begin{equation*}
	\theta_{1} \coloneqq  \int_{\R^2} \int_{\R^2} \sqrrange(y_{\rel}) \Theta_{1}(y_{\rel}, y'_{\rel}) \sqrrange(y'_{\rel}) \, \d y_{\rel} \, \d y'_{\rel} \, \, .
\end{equation*}
Plugging this into the series, we obtain
\begin{align*}
	\frac{4\pi}{\big(\log\big( \lambda +\quart \norm{\eta_{\com}}^2 + \half\norm{\eta_{\intv{n}\setminus\alpha}}^2\big) - \theta'\big)}	&\sum_{m=0}^{\infty}\frac{(4\pi \theta_{1})^{m}}{ \big(\log\big( \lambda +\quart \norm{\eta_{\com}}^2 + \half\norm{\eta_{\intv{n}\setminus\alpha}}^2\big) - \theta'\big)^{m}} \proj \\
	&= \frac{4\pi}{\log\big( \lambda +\quart \norm{\eta_{\com}}^2 + \half\norm{\eta_{\intv{n}\setminus\alpha}}^2\big) - \theta' - 4\pi\theta_{1}} \proj \, .
\end{align*}
We recognize this as the operator $\F{\diag{\shfparameter}{\alpha}(\lambda)}(\eta_{\com \cup \intv{n} \setminus \alpha})$ with $\shfparameter = \theta' + 4\pi\theta_{1}$ or equivalently,
\begin{equation*}
	\shfparameter =  2\theta + 4\pi\sum_{r=1}^{4}\theta_{r}+ 2\log\pi \, \, .
\end{equation*} 
Thus, $\prediag{\theta}{\alpha}(\lambda)$ converges to $\diag{\shfparameter}{\alpha}(\lambda)$ strongly. This completes the proof of Lemma~\ref{lem:resolventconvergence}, and hence also the proof of Theorem~\ref{thm:maintheorem}, namely that 
\begin{equation*}
  \predeltatrans{\theta}(t) \to \deltatrans{\shfparameter}(t) \qquad \text{strongly} \, ,
\end{equation*}
uniformly on compact sets.

\section{Convergence to the SHF}{\label{section:convergencetoshf}}

In the final section, we prove Proposition~\ref{cor:shfconvergence}. The first step in this direction is to prove the tightness of the piecewise constant extensions of the family $\{\epsshe_{s,t}\}_{0<s<t}$ in the space $C(\R_{\leqslant}^2,\measspace{\R^2 \times \R^2})$. The approach we employ is the same as that in~\cite{tsai25stochasticheatflow} but adapted to our discrete setting. This involves establishing weighted norm bounds on the prelimiting semigroup operators, which is the content of the first lemma below. Before that, we need to introduce more notation. For a function $f$ on $\Omega = \R^{2\intv{n}}$ or $\R^{2\intv{n}}_{\alpha}$ and $a \in \R$, we define its \textit{prelimiting weighted norm} as
\begin{equation*}
	\norm{f}_{p, a}^{p} \coloneqq
	\begin{cases}
		\int_{\R^{2\intv{n}}} \big| f(x) e^{a |x|} \big|^{p} \, \d x \quad  & \Omega = \R^{2\intv{n}} \\
		\int_{\R^{2\intv{n}}_{\alpha}} \big|f(y) \, e^{ a | y_{\com} + \frac{\varepsilon y_{\rel}}{2} | + a | y_{\com} - \frac{\varepsilon y_{\rel}}{2}| + a |y_{\intv{n}\setminus\alpha}|} \big|^{p}\, \d y \quad & \Omega = \R^{2\intv{n}}_{\alpha}	 \, ,
		\end{cases} 
\end{equation*}
 and let $\Lpee{p}_{a}(\Omega) \coloneqq \{ g: \Omega \to \R \, : \, \norm{g}_{p, a} < \infty \}$. The norm of course depends on $\varepsilon$ but we suppress it from the notation for brevity. When $a = 0$, the definition just reduces to the usual $\Lpee{p}$ norm in which case we will stick to the standard notation. For a non-negative operator $\cali{T}$ that maps between $\Lpee{p}_{a}$ and $\Lpee{p'}_{a'}$, we set its norm as
 \begin{equation*}
 	\Norm{\cali{T}}_{ p,a \to p', a'} \coloneqq \sup\{ | \langle f, \cali{T} g \rangle| \, : \, \norm{f}_{ q', a'} \leqslant 1 , \norm{g}_{ p, a} \leqslant 1\} \, ,
 \end{equation*}
 where $q'$ is such that $1/p' + 1/q' = 1$. For the case when $p=p' = 2$ and $a = a' =0$, this is the same as $\normop{ \cdot}$ that we have been using so far. We slightly abuse notation and use the same symbol for piecewise extensions of $\epsshe$.

\begin{lemma}{\label{lem:prelimitingtransbounds}}
	Fix $\alpha,  \alpha' \in \pair\intv{n}$ such that $\alpha \neq \alpha'$. For $a > 0, p>2$ and $q$ such that $p^{-1} + q^{-1} =1$ and $t > 0$, the following bounds hold true: 
	\begin{subequations}{\label{eqn:prelimitingtransbounds}}
		\begin{align}
			\normop{\epstrans(t)} & \leqslant 1 \label{eqn:pretransbound_0} \\
			\norm{\epstwotrans(t)}_{q} & \leqslant ct^{-\frac{1}{p}} \label{eqn:pretransbound_1} \\
			\normop{\preoutgtrans{\alpha}(t)} = \normop{\preinctrans{\alpha}(t)} & \leqslant ct^{-\tfrac{1}{2}} \label{eqn:pretransbound_2}\\
			\normop{\preoffdiagtrans{\alpha\alpha'}(t)} & \leqslant ct^{-1} \label{eqn:pretransbound_3} \\
			\NORmop{\int_{0}^{\infty}e^{- t}\preoffdiagtrans{\alpha \alpha'}(t) \, \d t \,} & \leqslant c \label{eqn:pretransbound_4} \\
			\normop{\prediagtrans{\theta}{\alpha}(t)} & \leqslant ct^{-1}|\log\big( t \wedge \tfrac{1}{2} \big) |^{-2}e^{ct} \label{eqn:pretransbound_5} \\
		\Norm{\preoutgtrans{\alpha}(t)}_{2,0 \to q,a}	=\Norm{\preinctrans{\alpha}(t)}_{p,a \to 2,0} & \leqslant ct^{-\frac{1}{p}} \, . \label{eqn:pretransbound_6} 
		\end{align}
	\end{subequations}
	Consequently, there exists $c > 0$ such that
	\begin{equation*}
		\sup_{0 <\varepsilon <1/c}\normop{\predeltatrans{\theta}(t)} \leqslant ce^{ct} \, .
	\end{equation*}
\end{lemma}

Using this lemma, for every $m \geqslant 1$, $h,h' \in C_{c}(\R^2)$ and $s < t, s' < t'$ in a compact interval $I$, the partition functions satisfy
\begin{equation*}
	\Big(\E\Big| \big\langle h , \, (\epsshe_{s,t} - \epsshe_{s',t'})\, h' \big\rangle \Big|^{2m} \Big)^{1/2m} \leqslant c \big( |s-s'|^{\quart(1- \frac{2}{p})} + |t-t'|^{\quart(1-\frac{2}{p})} \big)  \, ,
\end{equation*}
for a constant $c = c(m, \range, p, h, h', I)$. This verifies the requirements for Kolmogorov's criterion, yielding the tightness of the process $\{\langle h , \epsshe_{s,t} h' \rangle \}_{s< t}$ for every $h, h' \in C_{c}(\R^2)$. Using Mitoma's theorem~\cite{mitoma83}, this implies the tightness of the partition functions in $C(\R_{\leqslant}^{2}, \measspace{\R^2 \times \R^2})$. The proof for the above inequality is the same as in~\cite{tsai25stochasticheatflow}. We skip this part and only prove Lemma~\ref{lem:prelimitingtransbounds}.

\begin{proof}
	Assuming the bounds hold true, the bound on $\predeltatrans{\theta}(t)$ follows from Lemma~2.1 of \cite{sudhtsai25deltabose}. Therefore, it suffices to show the bounds \eqref{eqn:pretransbound_0}--\eqref{eqn:pretransbound_6}. We begin with equation~\eqref{eqn:pretransbound_0}. Fix $f, g \in \Ltwo(\R^{2\intv{n}})$. Then,
	\begin{align*}
		\langle f, \epstrans(t) g \rangle 
		&= \int_{\R^{2\intv{n}}} \int_{\R^{2\intv{n}}} f(x) \conj{\epstrans(t, x - x')} g(x') \, \d x \, \d x' \\
		& = \frac{1}{(2\pi)^{2n}} \int_{\R^{2\intv{n}}} \int_{\R^{2\intv{n}}} \int_{\epshattorus{2\intv{n}}}f(x) e^{-t\fourierepsgen(\xi)}e^{\im \xi \cdot (\epsfloor{x} - \epsfloor{x'})} g(x') \, \d x \, \d x' \, \d \xi \, .
	\end{align*}
	As done in Section~\ref{section:convergenceofresolvent}, to bound this integral, we decompose it as a sum over the lattice and estimate it there using Fourier methods. Define 
	\begin{equation*}
		\widetilde{f}(u, \xi) \coloneqq \frac{1}{(2\pi)^{n}} \sum_{p \in \Z^{2\intv{n}}_{\varepsilon}} f(p + u)e^{\im \xi \cdot p} \, \, ,
	\end{equation*}
	and $\widetilde{g}$ similarly. Using this notation, the inner product above takes the form
	\begin{equation*}
		\int_{\epshattorus{2\intv{n}}} \int_{\epsrect{2\intv{n}}} \int_{\epsrect{2\intv{n}}} \widetilde{f}(u, \xi) e^{-t \fourierepsgen(\xi)} \conj{\widetilde{g}(u',\xi)} \, \d u \, \d u' \, \d \xi \, .
	\end{equation*} 
	Taking absolute values, using the trivial bound $|e^{-t\fourierepsgen(\xi)}| \leqslant 1$, followed by the Cauchy--Schwarz inequality and the discrete Plancherel's theorem, we obtain the bound
	\begin{equation*}
		| \langle f, \epstrans(t) g \rangle | \leqslant \bigg( \int_{\epsrect{2\intv{n}}} \sum_{p \in \Z^{2\intv{n}}_{\varepsilon}} |f(p+u)|^{2} \, \d u \bigg)^{\half} \cdot \bigg( \int_{\epsrect{2\intv{n}}} \sum_{p' \in \Z^{2\intv{n}}_{\varepsilon}} |g(p'+u')|^{2} \, \d u' \bigg)^{\half} = \norm{f}_{2} \norm{g}_{2} \, \, .
	\end{equation*}   
	Taking supremum over $\norm{f}_{2} , \norm{g}_{2} \leqslant 1$ proves equation~\eqref{eqn:pretransbound_0}. Next, we proceed to proving equation~\eqref{eqn:pretransbound_1}. Clearly, $\epstwotrans(t) \in \Lpee{1}(\R^2)$ as it is just the piecewise constant extension of a probability transition function. Recall, by the Fourier inversion formula,
	\begin{equation*}
		\epstwotrans(t,x) = \frac{1}{(2\pi)^{2}} \int_{\epshattorus{2}} e^{- t \fourierepstwogen(\xi)} e^{- \im \xi \cdot x} \, \d \xi \, .
	\end{equation*}
	Using Lemma~\ref{lem:uniformgreensbound} in this expression, we obtain the bound
	\begin{equation*}
		\norm{\epstwotrans(t)}_{\infty} 
		\leqslant c \int_{\widehat{\R}^2}e^{- ct \norm{\xi}^{2}} \, \d \xi 
		\leqslant ct^{-1} \, .
	\end{equation*}
	Therefore, 
	\begin{equation*}
		\int_{\R^2}|\epstwotrans(t,x)|^{q} \, \d x = \int_{\R^2} |\epstwotrans(t,x)|^{q-1} |\epstwotrans(t,x)| \, \d x \leqslant c t^{1-q} \int_{\R^2} \epstwotrans(t,x) \, \d x \, .
	\end{equation*}
	Raising both sides to the $1/q$-th power yields the desired bound. Let us move to proving equation~\eqref{eqn:pretransbound_2}. Equality follows by adjointness, so it suffices to show the bound for $\preoutgtrans{\alpha}(t)$. Fix $f \in \Ltwo(\R^{2\intv{n}})$ and $g = g_1 g_2 \in \Ltwo(\R^{2\intv{n}}_{\alpha})$. Similar to what we did for the corresponding resolvent, replace the operator by its continuous interpolation $\epsconttrans(t)$ along with the error incurred. Then, moving to the Fourier domain, we see that $\langle f, \preoutgtrans{\alpha}(t) g \rangle$ equals 
	\begin{align*}
		&\frac{1}{2\pi}
		\int_{\epshattorus{2\intv{n}}} \F{f}(\xi)e^{-t\fourierepsgen(\xi)}\conj{\F{g_{2}}(\xi_i + \xi_j, \xi_{\intv{n}\setminus\alpha})} \, \d \xi \, \int_{\R^2}g_{1}(y'_{\rel}) \sqrrange(y'_{\rel}) \, \d y'_{\rel} \\
		& + \frac{1}{(2\pi)^{2n}}\int_{\R^{2\intv{n}}} \int_{\R^{2\intv{n}}}\int_{\epshattorus{2\intv{n}}}f(x)e^{-t\fourierepsgen(\xi)}\Big[e^{\im \xi \cdot (\epsfloor{x}-\epsfloor{x'})} - e^{\im \xi \cdot (x - x')}\Big] \conj{G_{\varepsilon\alpha}(x')}  \, \d x \, \d x' \, \d \xi \, ,
	\end{align*}
	where $G_{\varepsilon\alpha}(x') = \varepsilon^{-2}(\sqrrange g \circ \Teps{\alpha})(x')$. Using Lemma~\ref{lem:uniformgreensbound}, the first term can be bounded by $t^{-\tfrac{1}{2}}$. Methods similar to those used in the proof of the norm bound for $\preoutg{\alpha}(\lambda)$ can be used to show the same bound for the second term as well.  In fact, the bound in equation~\eqref{eqn:pretransbound_3} also follows along the same lines as its resolvent bound, using Lemma~\ref{lem:guquasteltsailemma}. 
	
	Let us move to showing~\eqref{eqn:pretransbound_4}. For this case assume $f \in \Ltwo(\R^{2\intv{n}}_{\alpha})$ and $g \in \Ltwo(\R^{2\intv{n}}_{\alpha'})$ instead. As the kernel is non-negative, by Tonelli's theorem, we have
	\begin{equation*}
		\Big| \Big\langle   f, \int_{0}^{\infty}e^{-t} \preoffdiagtrans{\alpha\alpha'}(t) \, \d t \, g  \Big\rangle  \Big| \leqslant \big\langle \, |f| , \preoffdiag{\alpha\alpha'}(1) |g| \,\big\rangle \leqslant c \norm{f}_{2} \norm{g}_{2} \, ,
	\end{equation*}
	by the results in Lemma~\ref{lem:resolventconvergence}. Although $1$ lies outside the assumed range of $\lambda$ in that Lemma, the bound for $\preoffdiag{\alpha\alpha'}(\lambda)$ holds for any $\lambda > 0$.  This proves equation~\eqref{eqn:pretransbound_4}. Next, we turn to $\prediagtrans{\theta}{\alpha}(t)$. Recall it is given by the expression:
	\begin{equation*}
		\beta_{\varepsilon} \delta_{0}(t) \id{n} + \sum_{m=1}^{\infty}\beta_{\varepsilon}^{m+1} \int_{\Sigma(t)} \preoffdiagtrans{\alpha\alpha}(\tau_1)\preoffdiagtrans{\alpha\alpha}(\tau_2) \cdots \preoffdiagtrans{\alpha\alpha}(\tau_m) \, \d \vec{\tau} \, .
	\end{equation*}
	The first term has norm $\beta_{\varepsilon}$ which vanishes as $\varepsilon \to 0.$ Thus, it suffices to show the desired bound for the second term. First let us estimate the norm of $\preoffdiagtrans{\alpha\alpha}(t)$. Using 
	\begin{align*}
		\langle f, \preoffdiagtrans{\alpha\alpha}(t) g \rangle
		& = \int_{\R^{2\intv{n}}} \int_{\R^{2\intv{n}}} F_{\varepsilon\alpha}(x) \conj{\epstrans(t,\epsfloor{x} - \epsfloor{x'})} \conj{G_{\varepsilon\alpha}(x')} \, \d x \, \d x' \\
		&  =\frac{1}{(2\pi)^{2n}}\int_{\R^{2\intv{n}}}\int_{\R^{2\intv{n}}}\int_{\epshattorus{2\intv{n}}}F_{\varepsilon\alpha}(x) e^{-t\fourierepsgen(\xi)} e^{\im \xi \cdot(\epsfloor{x} - \epsfloor{x'})}G_{\varepsilon\alpha}(x') \, \d x \, \d x' \, \d \xi \, .
	\end{align*}
	After making the necessary change of variables as in how we estimated $\langle f ,\cali{D}_{1}(\varepsilon) g \rangle$, we obtain
	\begin{equation}{\label{eqn:diagonaltrans_bound_1}}
		|\langle f, \preoffdiagtrans{\alpha\alpha}(t) g \rangle | \lesssim \Bigg(\frac{\varepsilon^{-2}}{(2\pi)^{2}} \int_{\ballorigin{R}} e^{-\frac{ct}{\varepsilon^2}\norm{\eta_{\rel}}^2} \, \d \eta_{\rel} \Bigg) \norm{f}_{2} \norm{g}_{2}  \, ,
	\end{equation}
	where $R > 0$ is chosen such that $\metorus{4,\rel}(\eta_{\com}) \subset \ballorigin{R}$, uniformly with respect to $\eta_{\com \cup \intv{n}\setminus\alpha} \in \metorus{4,\com} \times \epsalhattorus{2\intv{n}\setminus\alpha}$. Thus, to estimate the norm of $\prediagtrans{\theta}{\alpha}(t)$, it suffices to evaluate the above $\eta_{\rel}$ integral, plug it into its definition, and evaluate the sum after integrating over the simplex. Switching to polar coordinates, the $\d \eta_{\rel}$ integral becomes:
	\begin{align*}
		\varepsilon^{-2} \int_{0}^{R} e^{-\frac{ct}{\varepsilon^2}r^2} r \, \d r = \frac{1}{4\pi}\frac{1-e^{-\frac{cR^2 t}{\varepsilon^2}}}{2t} \, .
	\end{align*}
	To bound this function, first consider the product
	\begin{equation*}
		\Big(1- e^{-\frac{cR^2 t}{\varepsilon^2}}\Big) \times \sqrt{1+ \frac{\varepsilon^2}{cR^2t}}  = \Big(1 -e^{-\frac{cR^2 t}{\varepsilon^2}}\Big) \times \frac{\sqrt{ cR^2 t + \varepsilon^2}}{\sqrt{cR^2 t}}\, \, .
	\end{equation*}
	Consider the case when $cR^2 t \leqslant \varepsilon^2$. In this case, use the inequality $|1 - e^{-x}| \leqslant \sqrt{x}$ for $x > 0$, to obtain the bound
	\begin{equation*}
		\frac{\sqrt{cR^2 t}}{\sqrt{\varepsilon^2}} \times \frac{\sqrt{cR^2 t + \varepsilon^2}}{\sqrt{cR^2 t}} = \sqrt{\frac{cR^2 t}{\varepsilon^2} + 1} \leqslant 2 \, .
	\end{equation*}
	In the complementary case, when $\varepsilon^2 < R^2 t$, we use the trivial inequality $|1 - e^{-x} | \leqslant 2$ for $x > 0$, to get the upper bound
	\begin{equation*}
		2 \frac{\sqrt{cR^2 t + cR^2 t}}{\sqrt{cR^2 t}} = 2\sqrt{2} \, .
	\end{equation*}
	Combining these followed by some rearrangements, we obtain
	\begin{equation*}
		\frac{1 - e^{-\frac{cR^2 t}{\varepsilon^2}}}{2t} \leqslant \frac{1}{2t} \times \frac{1}{\sqrt{1 + \frac{\varepsilon^2}{cR^2 t}}} = \frac{1}{2\sqrt{t^2 + \varepsilon^2c^{-1}R^{-2}}t} \, . 
	\end{equation*}
    Using this in the definition of $\prediagtrans{\theta}{\alpha}(t)$ in conjunction with the bound in equation~\eqref{eqn:diagonaltrans_bound_1}, we get the estimate
    \begin{equation*}
     \normop{\diagtrans{\theta}{\alpha}(t)} \leqslant \beta_{\varepsilon}\sum_{m=1}^{\infty} \int_{\Sigma(t)} \prod_{l=1}^{m} \frac{\beta_{\varepsilon}}{4\pi\sqrt{\tau_{k}^{2} + c^{-1}R^{-2}\varepsilon^2 \tau_{l}}} \, \d \vec{\tau} \, .
    \end{equation*}
    This is the same estimate obtained in \cite[eq. (3.15)]{sudhtsai25deltabose} with $c_{0} = c^{-1}R^{-2}$ as part of the proof of diagonal operator bound in the weighted space $\Lpee{2}_{a}(\R^{2\intv{n}}_{\alpha})$. Thus, the rest of the proof can be continued along those lines, to yield the required bound in equation~\eqref{eqn:pretransbound_5}.
    
It remains to prove the weighted norm bound for $\preoutgtrans{\alpha}(t)$ and $\preinctrans{\alpha}(t)$. The fact that they are equal just follows from the definition of the operators and the weighted norm. We show the bound for the former. Fix $f \in \Ltwo(\R^{2\intv{n}}) $ and $g \in \Ltwo(\R^{2\intv{n}}_{\alpha})$. By definition, for $\alpha = ij \in \pair\intv{n}$, the inner product $\langle f, \preoutgtrans{\alpha}(t) g \rangle$ is equal to
\begin{align*}
 &\int_{\R^{2\intv{n}}} \int_{\R^{2\intv{n}}_{\alpha}} f(x) \conj{\epstrans(t, \epsfloor{x} - \epsfloor{\Seps{\alpha}y})} \sqrrange(y_{\rel}) g(y) \, \d y \, \d x \\
	& = \int_{\R^2} \int_{\R^2}f(x) \, \epstwotrans(t, \epsfloor{x_{i}} -\epsfloor{y_{\com} + \varepsilon y_{\rel}/2})\epstwotrans(t, \epsfloor{x_{j}} -\epsfloor{y_{\com} - \varepsilon y_{\rel}/2}) \\
	& \quad \quad  \times  \prod_{k \in \intv{n}\setminus\alpha}\epstwotrans(t, \epsfloor{x_{k}} - \epsfloor{y_{k}})  \times \sqrrange(y_{\rel}) \, g_{1}(y_{\rel}) g_{2}(y_{\com\cup\intv{n}\setminus\alpha}) \, \d y_{\rel} \, \d y_{\com\cup\intv{n}\setminus\alpha} \, \d x_{i} \, \d x_{j} \, \d x_{\intv{n}\setminus\{ij\}} .
\end{align*}
Using H\"{o}lder's inequality in the $\d x_{i}$ integral with indices $q$ and $p$ yields 
\begin{align*}
	\int_{\R^2}|f(x)|\epstwotrans(t, \epsfloor{x_{i}} - 
	& \epsfloor{y_{\com} + \varepsilon y_{\rel}/2}) \, \d x_{i} \\
	& \leqslant \Big(\int_{\R^2} |f(x)|^{p} \, \d x_{i}\Big)^{1/p} \cdot \Big(\int_{\R^2}|\epstwotrans(t, \epsfloor{x_{i}} - \epsfloor{y_{\com} + \varepsilon y_{\rel}/2})|^{q} \, \d x_{i}\Big)^{1/q}\\
	& \leqslant ct^{-\frac{1}{p}}\Big(\int_{\R^2}|f(x)|^{p} \, \d x_{i}\Big)^{1/p} \, ,
\end{align*}
where we used \eqref{eqn:pretransbound_1}. Insert this bound into that of $\langle f, \preoutgtrans{\alpha}(t) g \rangle$ from above, make a change of variables $y_{\com} \mapsto y_{\com} + \varepsilon y_{\rel}/2$, followed by the bound $\normop{\epstwotrans(t)} \leqslant 1$ and Cauchy--Schwarz inequality to obtain the bound
\begin{align}{\label{eqn:pretransweighted_bound_1}}
    \leqslant \Bigg(\int_{\R^{2\intv{n}\setminus\{i\}}} \bigg( 
    &\int_{\R^{2}}|f(x)|^{p} \, \d x_{i} \bigg)^{2/p} \, \d x_{\intv{n}\setminus\{i\}} \Bigg)^{1/2} \nonumber \\
    & \times \Bigg( \int_{\R^{2\intv{n}\setminus\rel}} \bigg( \int_{\R^2}\sqrrange(y_{\rel})g_{1}(y_{\rel})g_{2}(y_{\com} + \varepsilon y_{\rel}/2, y_{\intv{n}\setminus\alpha}) \, \d y_{\rel} \bigg)^{2} \, \d y_{\com\cup\intv{n}\setminus\alpha} \Bigg)^{1/2} \, . 
\end{align}
Because $a > 0$, the $f$ integral on the right hand side above can be further bounded by the $\Lpee{p}_{a}(\R^{2\intv{n}})$ norm of $f$. That is, 
\begin{align*}
	\Bigg(\int_{\R^{2\intv{n}\setminus\{i\}}}
	& \bigg(\int_{\R^2} |f(x)|^{p} \, \d x_{i} \bigg)^{2/p} \, \d x_{\intv{n}\setminus\{i\}}\Bigg)^{1/2} \\
	& \leqslant \Bigg(\int_{\R^{2\intv{n}}}e^{p a |x_{\intv{n}\setminus\{i\}}|}|f(x)|^{p} \, \d x \Bigg)^{1/p} \cdot \Bigg(\int_{\R^{2\intv{n}\setminus\{i\}}} e^{-\frac{2pa}{p-2}|x_{\intv{n}\setminus\{i\}}|} \, \d x_{\intv{n}\setminus\{i\}} \Bigg)^{1/2 - 1/p} \\
	& \leqslant c\norm{f}_{p,a} \, ,
\end{align*} 
where, in the first inequality, we used H\"{o}lder's inequality with $p/2$ and $p/(p-2)$ and then in the second, we used the fact that $1 \leqslant e^{pa |x_{i}|}$. As for the second factor in equation~\eqref{eqn:pretransweighted_bound_1}, use Cauchy--Schwarz inequality first with respect to $y_{\rel}$ and then in the $y_{\com\cup\intv{n}\setminus\alpha}$ variables to get the upper bound $\sqrt{\norm{\range}_{1}}\norm{g}$. Taking a supremum over $f$ and $g$ over the unit balls in their respective spaces proves the desired bound. This concludes the proof of the lemma.
\end{proof}

Due to tightness, the rescaled partition functions have limit points in $C(\R_{\leqslant}^{2}, \measspace{\R^2 \times \R^2})$. The next lemma will help us deduce the Chapman--Kolmogorov property of these limits. Fix a family of functions $\{u_{\ell}\} \subset \Ltwo(\R^2)$ such that the operator $\psi \mapsto u_{\ell} \ast \psi$ converges strongly to $\id{1}$ on $\Ltwo(\R^2)$.
 
\begin{lemma}{\label{lem:chapman_kolmogorov}}
	For $h, h' \in C_{c}(\R^2),m \geqslant 1$ and $s < t < u$, the following holds:
	\begin{equation*}
		\lim_{\ell \to \infty}\lim_{\varepsilon \to 0}\E\Big[\Big\langle h, \Big(\epsshe_{s,t} \bullet_{\ell} \epsshe_{t,u} - \epsshe_{s,u}\Big) h' \Big\rangle^{2m} \Big] = 0 \, .
	\end{equation*}
\end{lemma}

Together with Theorem~\ref{thm:maintheorem}, this implies that every limit point of $\{\epsshe_{s,t}\}$ satisfies the Chapman--Kolmogorov property in Definition~\ref{defn:axiomatic_shf_two}.

\begin{proof}
Define the approximate delta-function $\mathbf{1}_{\varepsilon}:\R^2 \to \R$ as $z \mapsto \varepsilon^{-2}\mathbf{1}_{[0,\varepsilon)^{2}}(z)$. Let $U_{\ell}, \id{1,\varepsilon}:\Ltwo(\R^2) \to \Ltwo(\R^2)$ denote the operators with convolution kernels given by $u_{\ell}$ and $\mathbf{1}_{\varepsilon}$ respectively. Clearly, $\id{1,\varepsilon} \to \id{1}$ strongly as $\varepsilon \to 0$. Now, observe that the partition functions, which are piecewise constant extensions of their discrete counterparts, satisfy
\begin{align*}
	\int_{\R^2}\int_{\R^2}\epsshe_{s,t}(\epsfloor{x},\epsfloor{z})
	 \mathbf{1}_{\varepsilon}(z -z') 
	&\epsshe_{t,u}(\epsfloor{z'},\epsfloor{x'}) \, \d z \, \d z' \\
	&= \varepsilon^{-2} \int_{\R^2}\epsshe_{s,t}(\epsfloor{x},\epsfloor{z}) \epsshe_{t,u}(\epsfloor{z},\epsfloor{x'}) \, \d z \\ 
	&= \varepsilon^{2} \sum_{z \in \Z^2} \epsshe_{s,t}(\epsfloor{x},z)\epsshe_{t,u}(z,\epsfloor{x'}) \\
&	= \epsshe_{s,u}(\epsfloor{x'},\epsfloor{x'}) \, .
\end{align*}
Therefore, the inner product $\big\langle h, \big(\epsshe_{s,t} \bullet_{\ell} \epsshe_{t,u} - \epsshe_{s,u}\big) h' \big\rangle$ is equal to
\begin{align*}
\int_{\R^2}\int_{\R^2}  
& 	h(x)h'(x') \\
& \times \Bigg(\int_{\R^2 \times \R^2}\epsshe_{s,t}(\epsfloor{x}, \epsfloor{z}) \big[u_{\ell}(z -z') - \mathbf{1}_{\varepsilon}(z- z')\big]\epsshe_{t,u}(\epsfloor{z'}, \epsfloor{x'}) \, \d z \, \d z' \Bigg) \, \d x \, \d x' \, \\
& = \langle h, \big(\epsshe_{s,t} \,[U_{\ell} - \id{1,\varepsilon}]\,\epsshe_{t,u}\big) \, h' \rangle \, \, .
\end{align*}
As $\epsshe_{s,t}$ and $\epsshe_{t,u}$ are independent, raising the final inner product to the power $2m$ and taking expectations gives
\begin{equation*}
\E \Big(\big\langle h, \big(\epsshe_{s,t} \bullet_{\ell} \epsshe_{t,u} - \epsshe_{s,u}\big) h' \big\rangle \Big)^{2m} = \big\langle h^{\otimes 2m} , \, \modpredeltatrans{\theta}(t-s) \,[U_{\ell} - \id{1,\varepsilon}]^{\otimes 2m} \, \modpredeltatrans{\theta}(u-t) \, h'^{\otimes 2m} \big\rangle \, \, .
\end{equation*} 
By our assumption on the family $\{u_{\ell}\}$, the operator $U_{\ell}$ converges to $\id{1}$ strongly, which implies that $U_{\ell}^{\otimes 2m} \to \id{1}^{\otimes  2m} = \id{2m}$ strongly. Thus, 
\begin{equation*}
	\lim_{\ell \to \infty} \lim_{\varepsilon \to 0} \, \big[U_{\ell} - \id{1, \varepsilon}\big]^{\otimes 2m} = \lim_{\ell \to \infty}  \, \big[U_{\ell} - \id{1} \big]^{\otimes 2m}  = 0 \, ,
\end{equation*}
in the strong operator topology. Then, using the uniform $\varepsilon-$bound, taking limits in the expectation above proves the lemma. 
\end{proof}

We now have all the requisite results to finish the proof of Proposition~\ref{cor:shfconvergence}. Due to the tightness result, the family $\{\epsshe_{s,t}\}$ admits subsequential limits, all of which satisfy Definition~\ref{defn:axiomatic_shf_one}. By Lemma~\ref{lem:chapman_kolmogorov}, every limit satisfies Definition~\ref{defn:axiomatic_shf_two}, the Chapman--Kolmogorov property. Also, since the Brownian environments over disjoint time intervals are independent, so are the partition functions and therefore their subsequential limits. This verifies Definition~\ref{defn:axiomatic_shf_three}.  Definition~\ref{defn:axiomatic_shf_four} holds due to Theorem~\ref{thm:maintheorem}. Thus, all subsequential limits fulfill the requirements of Definition~\ref{defn:axiomatic_shf}. By Tsai's characterization~\cite{tsai25stochasticheatflow}, all these limit points must have the same law as $\SHF(\shfparameter)$ which implies that the whole sequence converges to $\SHF(\shfparameter)$. This completes the proof.


\printbibliography[title=References]

@article {alberts14continuum,
	AUTHOR = {Alberts, Tom and Khanin, Konstantin and Quastel, Jeremy},
	TITLE = {The continuum directed random polymer},
	JOURNAL = {J. Stat. Phys.},
	FJOURNAL = {Journal of Statistical Physics},
	VOLUME = {154},
	YEAR = {2014},
	NUMBER = {1-2},
	PAGES = {305--326},
	ISSN = {0022-4715,1572-9613},
	MRCLASS = {82D60},
	MRNUMBER = {3162542},
	DOI = {10.1007/s10955-013-0872-z},
	URL = {https://doi.org/10.1007/s10955-013-0872-z},
}

@article {alberts14intermediate,
	AUTHOR = {Alberts, Tom and Khanin, Konstantin and Quastel, Jeremy},
	TITLE = {The intermediate disorder regime for directed polymers in
	dimension {$1+1$}},
	JOURNAL = {Ann. Probab.},
	FJOURNAL = {The Annals of Probability},
	VOLUME = {42},
	YEAR = {2014},
	NUMBER = {3},
	PAGES = {1212--1256},
	ISSN = {0091-1798,2168-894X},
	MRCLASS = {60F05 (60K37 82D10)},
	MRNUMBER = {3189070},
	MRREVIEWER = {Patr\'icia\ Gon\c calves},
	DOI = {10.1214/13-AOP858},
	URL = {https://doi.org/10.1214/13-AOP858},
}

@book {albeverio88solvable,
	AUTHOR = {Albeverio, Sergio and Gesztesy, Friedrich and H\o egh-Krohn,
	Raphael and Holden, Helge},
	TITLE = {Solvable models in quantum mechanics},
	SERIES = {Texts and Monographs in Physics},
	PUBLISHER = {Springer-Verlag, New York},
	YEAR = {1988},
	PAGES = {xiv+452},
	ISBN = {0-387-17841-4},
	MRCLASS = {81C05 (03H10 81-02 81C10 81H20)},
	MRNUMBER = {926273},
	MRREVIEWER = {Lawrence\ E.\ Thomas},
	DOI = {10.1007/978-3-642-88201-2},
	URL = {https://doi.org/10.1007/978-3-642-88201-2},
}

@article {bertini1998two,
	AUTHOR = {Bertini, Lorenzo and Cancrini, Nicoletta},
	TITLE = {The two-dimensional stochastic heat equation: renormalizing a
	multiplicative noise},
	JOURNAL = {J. Phys. A},
	FJOURNAL = {Journal of Physics. A. Mathematical and General},
	VOLUME = {31},
	YEAR = {1998},
	NUMBER = {2},
	PAGES = {615--622},
	ISSN = {0305-4470,1751-8121},
	MRCLASS = {82C31 (60H15)},
	MRNUMBER = {1629198},
	DOI = {10.1088/0305-4470/31/2/019},
	URL = {https://doi.org/10.1088/0305-4470/31/2/019},
}

@article {bolthausen89polymer,
	AUTHOR = {Bolthausen, Erwin},
	TITLE = {A note on the diffusion of directed polymers in a random
	environment},
	JOURNAL = {Comm. Math. Phys.},
	FJOURNAL = {Communications in Mathematical Physics},
	VOLUME = {123},
	YEAR = {1989},
	NUMBER = {4},
	PAGES = {529--534},
	ISSN = {0010-3616,1432-0916},
	MRCLASS = {60K35 (60F05 60G42 82D60)},
	MRNUMBER = {1006293},
	MRREVIEWER = {Cheng\ Xun\ Wu},
	URL = {http://projecteuclid.org/euclid.cmp/1104178982},
}

@article {carmonahu02,
	AUTHOR = {Carmona, Philippe and Hu, Yueyun},
	TITLE = {On the partition function of a directed polymer in a
	{G}aussian random environment},
	JOURNAL = {Probab. Theory Related Fields},
	FJOURNAL = {Probability Theory and Related Fields},
	VOLUME = {124},
	YEAR = {2002},
	NUMBER = {3},
	PAGES = {431--457},
	ISSN = {0178-8051,1432-2064},
	MRCLASS = {60K37 (82D30)},
	MRNUMBER = {1939654},
	MRREVIEWER = {Ofer\ Zeitouni},
	DOI = {10.1007/s004400200213},
	URL = {https://doi.org/10.1007/s004400200213},
}

@misc{chen1stochastic,
	title={Stochastic motions of the two-dimensional many-body delta-Bose gas, I: One-$\delta$ motions}, 
	author={Yu-Ting Chen},
	year={2025},
	eprint={2505.01703},
	archivePrefix={arXiv},
	primaryClass={math.PR},
	url={https://arxiv.org/abs/2505.01703}, 
}

@article {chen25skew,
	AUTHOR = {Chen, Yu-Ting},
	TITLE = {Two-dimensional delta-{B}ose gas: skew-product relative
	motions},
	JOURNAL = {Ann. Appl. Probab.},
	FJOURNAL = {The Annals of Applied Probability},
	VOLUME = {35},
	YEAR = {2025},
	NUMBER = {5},
	PAGES = {3150--3214},
	ISSN = {1050-5164,2168-8737},
	MRCLASS = {60J55 (60H10 60H15 60H30 60J60)},
	MRNUMBER = {4975045},
	MRREVIEWER = {Mario\ Abundo},
	DOI = {10.1214/25-aap2179},
	URL = {https://doi.org/10.1214/25-aap2179},
}

@misc{chen2stochastic,
	title={Stochastic motions of the two-dimensional many-body delta-Bose gas, II: Many-$\delta$ motions}, 
	author={Yu-Ting Chen},
	year={2025},
	eprint={2505.01704},
	archivePrefix={arXiv},
	primaryClass={math.PR},
	url={https://arxiv.org/abs/2505.01704}, 
}

@misc{chen3stochastic,
	title={Stochastic motions of the two-dimensional many-body delta-Bose gas, III: Path integrals}, 
	author={Yu-Ting Chen},
	year={2025},
	eprint={2505.03006},
	archivePrefix={arXiv},
	primaryClass={math.PR},
	url={https://arxiv.org/abs/2505.03006}, 
}

@article {chen4stochastic,
	AUTHOR = {Chen, Yu-Ting},
	TITLE = {Stochastic motions of the two-dimensional many-body
	delta-{B}ose gas, {IV}: {T}ransformations of relative motions},
	JOURNAL = {Electron. Commun. Probab.},
	FJOURNAL = {Electronic Communications in Probability},
	VOLUME = {31},
	YEAR = {2026},
	PAGES = {Paper No. 25, 9},
	ISSN = {1083-589X},
	MRCLASS = {60H10 (35Q40 35R60 60K35)},
	MRNUMBER = {5072225},
	DOI = {10.1214/26-ecp771},
	URL = {https://doi.org/10.1214/26-ecp771},
}

@article {chendalang14,
	AUTHOR = {Chen, Le and Dalang, Robert C.},
	TITLE = {H\"older-continuity for the nonlinear stochastic heat equation
	with rough initial conditions},
	JOURNAL = {Stoch. Partial Differ. Equ. Anal. Comput.},
	FJOURNAL = {Stochastics and Partial Differential Equations. Analysis and
	Computations},
	VOLUME = {2},
	YEAR = {2014},
	NUMBER = {3},
	PAGES = {316--352},
	ISSN = {2194-0401,2194-041X},
	MRCLASS = {60H15 (35B65 35K15 35R06 35R60 60G60)},
	MRNUMBER = {3255231},
	MRREVIEWER = {Anna\ Karczewska},
	DOI = {10.1007/s40072-014-0034-6},
	URL = {https://doi.org/10.1007/s40072-014-0034-6},
}

@article {chendalang15,
	AUTHOR = {Chen, Le and Dalang, Robert C.},
	TITLE = {Moments and growth indices for the nonlinear stochastic heat
	equation with rough initial conditions},
	JOURNAL = {Ann. Probab.},
	FJOURNAL = {The Annals of Probability},
	VOLUME = {43},
	YEAR = {2015},
	NUMBER = {6},
	PAGES = {3006--3051},
	ISSN = {0091-1798,2168-894X},
	MRCLASS = {60H15 (35R60 60G60)},
	MRNUMBER = {3433576},
	MRREVIEWER = {Mathew\ Joseph},
	DOI = {10.1214/14-AOP954},
	URL = {https://doi.org/10.1214/14-AOP954},
}

@misc{clark24continuum,
	title={Continuum polymer measures corresponding to the critical 2d stochastic heat flow}, 
	author={Jeremy Clark and Barkat Mian},
	year={2024},
	eprint={2409.01510},
	archivePrefix={arXiv},
	primaryClass={math.PR},
	url={https://arxiv.org/abs/2409.01510}, 
}

@misc{clark25conditional,
	title={Conditional GMC within the stochastic heat flow}, 
	author={Jeremy Clark and Li-Cheng Tsai},
	year={2025},
	eprint={2507.16056},
	archivePrefix={arXiv},
	primaryClass={math.PR},
	url={https://arxiv.org/abs/2507.16056}, 
}

@article {comets03localization,
	AUTHOR = {Comets, Francis and Shiga, Tokuzo and Yoshida, Nobuo},
	TITLE = {Directed polymers in a random environment: path localization
	and strong disorder},
	JOURNAL = {Bernoulli},
	FJOURNAL = {Bernoulli. Official Journal of the Bernoulli Society for
	Mathematical Statistics and Probability},
	VOLUME = {9},
	YEAR = {2003},
	NUMBER = {4},
	PAGES = {705--723},
	ISSN = {1350-7265,1573-9759},
	MRCLASS = {60K37 (82B41 82B44)},
	MRNUMBER = {1996276},
	MRREVIEWER = {Elena\ A.\ Zhizhina},
	DOI = {10.3150/bj/1066223275},
	URL = {https://doi.org/10.3150/bj/1066223275},
}

@article {comets06polymer,
	AUTHOR = {Comets, Francis and Yoshida, Nobuo},
	TITLE = {Directed polymers in random environment are diffusive at weak
	disorder},
	JOURNAL = {Ann. Probab.},
	FJOURNAL = {The Annals of Probability},
	VOLUME = {34},
	YEAR = {2006},
	NUMBER = {5},
	PAGES = {1746--1770},
	ISSN = {0091-1798,2168-894X},
	MRCLASS = {60K37 (60J60 60K35 82C44 82D60)},
	MRNUMBER = {2271480},
	MRREVIEWER = {Marina\ Vachkovskaia},
	DOI = {10.1214/009117905000000828},
	URL = {https://doi.org/10.1214/009117905000000828},
}

@book {comets17directed,
	AUTHOR = {Comets, Francis},
	TITLE = {Directed polymers in random environments},
	SERIES = {Lecture Notes in Mathematics},
	VOLUME = {2175},
	NOTE = {Lecture notes from the 46th Probability Summer School held in
	Saint-Flour, 2016},
	PUBLISHER = {Springer, Cham},
	YEAR = {2017},
	PAGES = {xv+199},
	ISBN = {978-3-319-50486-5; 978-3-319-50487-2},
	MRCLASS = {60K37 (60F10 60H05 60J10 82-01 82B41 82D60)},
	MRNUMBER = {3444835},
	MRREVIEWER = {Julien\ Poisat},
	DOI = {10.1007/978-3-319-50487-2},
	URL = {https://doi.org/10.1007/978-3-319-50487-2},
}

@article {corwin12kpz,
	AUTHOR = {Corwin, Ivan},
	TITLE = {The {K}ardar-{P}arisi-{Z}hang equation and universality class},
	JOURNAL = {Random Matrices Theory Appl.},
	FJOURNAL = {Random Matrices. Theory and Applications},
	VOLUME = {1},
	YEAR = {2012},
	NUMBER = {1},
	PAGES = {1130001, 76},
	ISSN = {2010-3263,2010-3271},
	MRCLASS = {82B31 (60B20 60K35 60K37)},
	MRNUMBER = {2930377},
	DOI = {10.1142/S2010326311300014},
	URL = {https://doi.org/10.1142/S2010326311300014},
}

@article {barraquand21loggamma,
	AUTHOR = {Barraquand, Guillaume and Corwin, Ivan and Dimitrov, Evgeni},
	TITLE = {Fluctuations of the log-gamma polymer free energy with general
	parameters and slopes},
	JOURNAL = {Probab. Theory Related Fields},
	FJOURNAL = {Probability Theory and Related Fields},
	VOLUME = {181},
	YEAR = {2021},
	NUMBER = {1-3},
	PAGES = {113--195},
	ISSN = {0178-8051,1432-2064},
	MRCLASS = {60K37 (82B23)},
	MRNUMBER = {4341071},
	DOI = {10.1007/s00440-021-01073-1},
	URL = {https://doi.org/10.1007/s00440-021-01073-1},
}

@article {csz17universality,
	AUTHOR = {Caravenna, Francesco and Sun, Rongfeng and Zygouras, Nikos},
	TITLE = {Universality in marginally relevant disordered systems},
	JOURNAL = {Ann. Appl. Probab.},
	FJOURNAL = {The Annals of Applied Probability},
	VOLUME = {27},
	YEAR = {2017},
	NUMBER = {5},
	PAGES = {3050--3112},
	ISSN = {1050-5164,2168-8737},
	MRCLASS = {82B44 (60K35 82D60)},
	MRNUMBER = {3719953},
	DOI = {10.1214/17-AAP1276},
	URL = {https://doi.org/10.1214/17-AAP1276},
}

@article {csz19moments,
	AUTHOR = {Caravenna, Francesco and Sun, Rongfeng and Zygouras, Nikos},
	TITLE = {On the moments of the {$(2+1)$}-dimensional directed polymer
	and stochastic heat equation in the critical window},
	JOURNAL = {Comm. Math. Phys.},
	FJOURNAL = {Communications in Mathematical Physics},
	VOLUME = {372},
	YEAR = {2019},
	NUMBER = {2},
	PAGES = {385--440},
	ISSN = {0010-3616,1432-0916},
	MRCLASS = {82D60 (60H15)},
	MRNUMBER = {4032870},
	DOI = {10.1007/s00220-019-03527-z},
	URL = {https://doi.org/10.1007/s00220-019-03527-z},
}

@article {csz20twokpz,
	AUTHOR = {Caravenna, Francesco and Sun, Rongfeng and Zygouras, Nikos},
	TITLE = {The two-dimensional {KPZ} equation in the entire subcritical
	regime},
	JOURNAL = {Ann. Probab.},
	FJOURNAL = {The Annals of Probability},
	VOLUME = {48},
	YEAR = {2020},
	NUMBER = {3},
	PAGES = {1086--1127},
	ISSN = {0091-1798,2168-894X},
	MRCLASS = {60H15 (35K59 35R60 82B44 82D60)},
	MRNUMBER = {4112709},
	MRREVIEWER = {Jiang\ Lun\ Wu},
	DOI = {10.1214/19-AOP1383},
	URL = {https://doi.org/10.1214/19-AOP1383},
}

@article {csz23stochasticheatflow,
	AUTHOR = {Caravenna, Francesco and Sun, Rongfeng and Zygouras, Nikos},
	TITLE = {The critical 2d stochastic heat flow},
	JOURNAL = {Invent. Math.},
	FJOURNAL = {Inventiones Mathematicae},
	VOLUME = {233},
	YEAR = {2023},
	NUMBER = {1},
	PAGES = {325--460},
	ISSN = {0020-9910,1432-1297},
	MRCLASS = {82B44 (35R60 60H15 82D60)},
	MRNUMBER = {4602000},
	DOI = {10.1007/s00222-023-01184-7},
	URL = {https://doi.org/10.1007/s00222-023-01184-7},
}

@article {dimock2004multi,
	AUTHOR = {Dimock, J. and Rajeev, S. G.},
	TITLE = {Multi-particle {S}chr\"odinger operators with point
	interactions in the plane},
	JOURNAL = {J. Phys. A},
	FJOURNAL = {Journal of Physics. A. Mathematical and General},
	VOLUME = {37},
	YEAR = {2004},
	NUMBER = {39},
	PAGES = {9157--9173},
	ISSN = {0305-4470,1751-8121},
	MRCLASS = {81Q05 (47N50)},
	MRNUMBER = {2090290},
	MRREVIEWER = {Igor\ Yurievich\ Popov},
	DOI = {10.1088/0305-4470/37/39/008},
	URL = {https://doi.org/10.1088/0305-4470/37/39/008},
}

@misc{drillick26universality,
	title={A universality result for the critical 2d stochastic heat flow}, 
	author={Drillick, Hindy and Hou, Jonathan and Parekh, Shalin},
	year={2026},
}

@article {dunlap22forward,
	AUTHOR = {Dunlap, Alexander and Gu, Yu},
	TITLE = {A forward-backward {SDE} from the 2{D} nonlinear stochastic
	heat equation},
	JOURNAL = {Ann. Probab.},
	FJOURNAL = {The Annals of Probability},
	VOLUME = {50},
	YEAR = {2022},
	NUMBER = {3},
	PAGES = {1204--1253},
	ISSN = {0091-1798,2168-894X},
	MRCLASS = {60H15 (35R60)},
	MRNUMBER = {4413215},
	DOI = {10.1214/21-aop1563},
	URL = {https://doi.org/10.1214/21-aop1563},
}

@article {dunlap25edwards,
	AUTHOR = {Dunlap, Alexander and Graham, Cole},
	TITLE = {Edwards-{W}ilkinson fluctuations in subcritical 2{D}
	stochastic heat equations},
	JOURNAL = {Electron. Commun. Probab.},
	FJOURNAL = {Electronic Communications in Probability},
	VOLUME = {30},
	YEAR = {2025},
	PAGES = {Paper No. 96, 11},
	ISSN = {1083-589X},
	MRCLASS = {60H15 (35R60 60F05 82B44)},
	MRNUMBER = {4999690},
	DOI = {10.1214/25-ecp735},
	URL = {https://doi.org/10.1214/25-ecp735},
}

@misc{dunlap2026renormalizationflow2dnonlinear,
	title={Renormalization flow for the 2D nonlinear stochastic heat equation: pointwise statistics and universality}, 
	author={Alexander Dunlap and Cole Graham},
	year={2026},
	eprint={2308.11850},
	archivePrefix={arXiv},
	primaryClass={math.PR},
	url={https://arxiv.org/abs/2308.11850}, 
}

@misc{ganguly2025sharpmomentuppertail,
	title={Sharp moment and upper tail asymptotics for the critical $2d$ Stochastic Heat Flow}, 
	author={Shirshendu Ganguly and Kyeongsik Nam},
	year={2025},
	eprint={2507.22029},
	archivePrefix={arXiv},
	primaryClass={math.PR},
	url={https://arxiv.org/abs/2507.22029}, 
}

@book {gradshteyn2007,
	AUTHOR = {Gradshteyn, I. S. and Ryzhik, I. M.},
	TITLE = {Table of integrals, series, and products},
	EDITION = {Seventh},
	PUBLISHER = {Elsevier/Academic Press, Amsterdam},
	YEAR = {2007},
	PAGES = {xlviii+1171},
	ISBN = {978-0-12-373637-6; 0-12-373637-4},
	MRCLASS = {00A22 (33-00 65-00 65A05)},
	MRNUMBER = {2360010},
}

@article {gu2021moments,
	AUTHOR = {Gu, Yu and Quastel, Jeremy and Tsai, Li-Cheng},
	TITLE = {Moments of the 2{D} {SHE} at criticality},
	JOURNAL = {Probab. Math. Phys.},
	FJOURNAL = {Probability and Mathematical Physics},
	VOLUME = {2},
	YEAR = {2021},
	NUMBER = {1},
	PAGES = {179--219},
	ISSN = {2690-0998,2690-1005},
	MRCLASS = {60H15 (35R60 46N30 82D60)},
	MRNUMBER = {4404819},
	DOI = {10.2140/pmp.2021.2.179},
	URL = {https://doi.org/10.2140/pmp.2021.2.179},
}

@article {gu2020kpz,
	AUTHOR = {Gu, Yu},
	TITLE = {Gaussian fluctuations from the 2{D} {KPZ} equation},
	JOURNAL = {Stoch. Partial Differ. Equ. Anal. Comput.},
	FJOURNAL = {Stochastics and Partial Differential Equations. Analysis and
	Computations},
	VOLUME = {8},
	YEAR = {2020},
	NUMBER = {1},
	PAGES = {150--185},
	ISSN = {2194-0401,2194-041X},
	MRCLASS = {35R60 (60H07 60H15)},
	MRNUMBER = {4058958},
	DOI = {10.1007/s40072-019-00144-8},
	URL = {https://doi.org/10.1007/s40072-019-00144-8},
}

@misc{gu2025stochasticheatflowblack,
	title={Stochastic heat flow is a black noise}, 
	author={Yu Gu and Li-Cheng Tsai},
	year={2025},
	eprint={2506.16484},
	archivePrefix={arXiv},
	primaryClass={math.PR},
	url={https://arxiv.org/abs/2506.16484}, 
}

@misc{gu2026loglog,
	title={Log Log Fluctuations of the Stochastic Heat Flow}, 
	author={Yu Gu and Li-Cheng Tsai},
	year={2026},
	eprint={2603.03246},
	archivePrefix={arXiv},
	primaryClass={math.PR},
	url={https://arxiv.org/abs/2603.03246}, 
}

@article {gubinelli15,
	AUTHOR = {Gubinelli, Massimiliano and Imkeller, Peter and Perkowski,
	Nicolas},
	TITLE = {Paracontrolled distributions and singular {PDE}s},
	JOURNAL = {Forum Math. Pi},
	FJOURNAL = {Forum of Mathematics. Pi},
	VOLUME = {3},
	YEAR = {2015},
	PAGES = {e6, 75},
	ISSN = {2050-5086},
	MRCLASS = {60H15 (35S50)},
	MRNUMBER = {3406823},
	DOI = {10.1017/fmp.2015.2},
	URL = {https://doi.org/10.1017/fmp.2015.2},
}

@article {hairer13kpz,
	AUTHOR = {Hairer, Martin},
	TITLE = {Solving the {KPZ} equation},
	JOURNAL = {Ann. of Math. (2)},
	FJOURNAL = {Annals of Mathematics. Second Series},
	VOLUME = {178},
	YEAR = {2013},
	NUMBER = {2},
	PAGES = {559--664},
	ISSN = {0003-486X,1939-8980},
	MRCLASS = {35K59 (35B10 35B65 35R60 60G22 60H15 60K35)},
	MRNUMBER = {3071506},
	MRREVIEWER = {Alp\ O.\ Eden},
	DOI = {10.4007/annals.2013.178.2.4},
	URL = {https://doi.org/10.4007/annals.2013.178.2.4},
}

@article {hairer14regularity,
	AUTHOR = {Hairer, M.},
	TITLE = {A theory of regularity structures},
	JOURNAL = {Invent. Math.},
	FJOURNAL = {Inventiones Mathematicae},
	VOLUME = {198},
	YEAR = {2014},
	NUMBER = {2},
	PAGES = {269--504},
	ISSN = {0020-9910,1432-1297},
	MRCLASS = {60H15 (35R60 60H40 81S20 82C28)},
	MRNUMBER = {3274562},
	MRREVIEWER = {Dora\ Sele\v si},
	DOI = {10.1007/s00222-014-0505-4},
	URL = {https://doi.org/10.1007/s00222-014-0505-4},
}

@article{husehenley85pinning,
	title = {Pinning and Roughening of Domain Walls in Ising Systems Due to Random Impurities},
	author = {Huse, David A. and Henley, Christopher L.},
	journal = {Phys. Rev. Lett.},
	volume = {54},
	issue = {25},
	pages = {2708--2711},
	numpages = {0},
	year = {1985},
	month = {Jun},
	publisher = {American Physical Society},
	doi = {10.1103/PhysRevLett.54.2708},
	url = {https://link.aps.org/doi/10.1103/PhysRevLett.54.2708}
}

@article {junk23stability,
	AUTHOR = {Junk, Stefan},
	TITLE = {Stability of weak disorder phase for directed polymer with
	applications to limit theorems},
	JOURNAL = {ALEA Lat. Am. J. Probab. Math. Stat.},
	FJOURNAL = {ALEA. Latin American Journal of Probability and Mathematical
	Statistics},
	VOLUME = {20},
	YEAR = {2023},
	NUMBER = {1},
	PAGES = {861--883},
	ISSN = {1980-0436},
	MRCLASS = {60K37 (60K35)},
	MRNUMBER = {4615656},
	MRREVIEWER = {Xinghua\ Zheng},
	DOI = {10.30757/alea.v20-31},
	URL = {https://doi.org/10.30757/alea.v20-31},
}

@article {mitoma83,
	AUTHOR = {Mitoma, Itaru},
	TITLE = {Tightness of probabilities on {$C([0,1];{\cal S}\sp{\prime}
	)$}\ and {$D([0,1];{\cal S}\sp{\prime} )$}},
	JOURNAL = {Ann. Probab.},
	FJOURNAL = {The Annals of Probability},
	VOLUME = {11},
	YEAR = {1983},
	NUMBER = {4},
	PAGES = {989--999},
	ISSN = {0091-1798,2168-894X},
	MRCLASS = {60B11 (60B12)},
	MRNUMBER = {714961},
	MRREVIEWER = {Peter\ Z.\ Daffer},
	URL =
	{http://links.jstor.org/sici?sici=0091-1798(198311)11:4<989:TOPOA>2.0.CO;2-P&origin=MSN},
}

@article{rajeev1999condensation,
	author = "Rajeev, S. G.",
	title = "{A Condensation of interacting bosons in two-dimensional space}",
	eprint = "hep-th/9905120",
	archivePrefix = "arXiv",
	month = "5",
	year = "1999"
}

@article {seppalainen12loggamma,
	AUTHOR = {Sepp\"al\"ainen, Timo},
	TITLE = {Scaling for a one-dimensional directed polymer with boundary
	conditions},
	JOURNAL = {Ann. Probab.},
	FJOURNAL = {The Annals of Probability},
	VOLUME = {40},
	YEAR = {2012},
	NUMBER = {1},
	PAGES = {19--73},
	ISSN = {0091-1798,2168-894X},
	MRCLASS = {60K35 (60K37 82B41 82D60)},
	MRNUMBER = {2917766},
	MRREVIEWER = {Jonathon\ R.\ Peterson},
	DOI = {10.1214/10-AOP617},
	URL = {https://doi.org/10.1214/10-AOP617},
}

@article {sudhtsai25deltabose,
	AUTHOR = {Surendranath, Sudheesh and Tsai, Li-Cheng},
	TITLE = {Two-dimensional delta {B}ose gas in a weighted space},
	JOURNAL = {Electron. Commun. Probab.},
	FJOURNAL = {Electronic Communications in Probability},
	VOLUME = {30},
	YEAR = {2025},
	PAGES = {Paper No. 35, 10},
	ISSN = {1083-589X},
	MRCLASS = {46N30},
	MRNUMBER = {4893420},
	DOI = {10.1214/25-ecp685},
	URL = {https://doi.org/10.1214/25-ecp685},
}

@article {tao24gaussian,
	AUTHOR = {Tao, Ran},
	TITLE = {Gaussian fluctuations of a nonlinear stochastic heat equation
	in dimension two},
	JOURNAL = {Stoch. Partial Differ. Equ. Anal. Comput.},
	FJOURNAL = {Stochastics and Partial Differential Equations. Analysis and
	Computations},
	VOLUME = {12},
	YEAR = {2024},
	NUMBER = {1},
	PAGES = {220--246},
	ISSN = {2194-0401,2194-041X},
	MRCLASS = {60H15 (35K05 35R60 60F05 60H07)},
	MRNUMBER = {4709542},
	DOI = {10.1007/s40072-022-00282-6},
	URL = {https://doi.org/10.1007/s40072-022-00282-6},
}

@misc{tsai25stochasticheatflow,
	title={Stochastic heat flow by moments}, 
	author={Li-Cheng Tsai},
	year={2025},
	eprint={2410.14657},
	archivePrefix={arXiv},
	primaryClass={math.PR},
	url={https://arxiv.org/abs/2410.14657}, 
}

@article {vargas07localization,
	AUTHOR = {Vargas, Vincent},
	TITLE = {Strong localization and macroscopic atoms for directed
	polymers},
	JOURNAL = {Probab. Theory Related Fields},
	FJOURNAL = {Probability Theory and Related Fields},
	VOLUME = {138},
	YEAR = {2007},
	NUMBER = {3-4},
	PAGES = {391--410},
	ISSN = {0178-8051,1432-2064},
	MRCLASS = {60K37},
	MRNUMBER = {2299713},
	MRREVIEWER = {Rongfeng\ Sun},
	DOI = {10.1007/s00440-006-0030-5},
	URL = {https://doi.org/10.1007/s00440-006-0030-5},
}

@incollection {walsh84,
	AUTHOR = {Walsh, John B.},
	TITLE = {An introduction to stochastic partial differential equations},
	BOOKTITLE = {\'Ecole d'\'et\'e{} de probabilit\'es de {S}aint-{F}lour,
	{XIV}---1984},
	SERIES = {Lecture Notes in Math.},
	VOLUME = {1180},
	PAGES = {265--439},
	PUBLISHER = {Springer, Berlin},
	YEAR = {1986},
	ISBN = {3-540-16441-3},
	MRCLASS = {60H15 (35R60 60G20 60J80)},
	MRNUMBER = {876085},
	MRREVIEWER = {Luis\ G.\ Gorostiza},
	DOI = {10.1007/BFb0074920},
	URL = {https://doi.org/10.1007/BFb0074920},
}

@article {zygouras22kpz,
	AUTHOR = {Zygouras, Nikos},
	TITLE = {Some algebraic structures in {KPZ} universality},
	JOURNAL = {Probab. Surv.},
	FJOURNAL = {Probability Surveys},
	VOLUME = {19},
	YEAR = {2022},
	PAGES = {590--700},
	ISSN = {1549-5787},
	MRCLASS = {05-02 (05Exx 60K35 82B23)},
	MRNUMBER = {4524523},
	DOI = {10.1214/19-ps335},
	URL = {https://doi.org/10.1214/19-ps335},
}

@article {zygouras24directed,
	AUTHOR = {Zygouras, Nikos},
	TITLE = {Directed polymers in a random environment: a review of the
	phase transitions},
	JOURNAL = {Stochastic Process. Appl.},
	FJOURNAL = {Stochastic Processes and their Applications},
	VOLUME = {177},
	YEAR = {2024},
	PAGES = {Paper No. 104431, 34},
	ISSN = {0304-4149,1879-209X},
	MRCLASS = {82B44 (82D60)},
	MRNUMBER = {4793446},
	MRREVIEWER = {Dimitri\ Petritis},
	DOI = {10.1016/j.spa.2024.104431},
	URL = {https://doi.org/10.1016/j.spa.2024.104431},
}

\end{document}